\documentclass[reqno,11pt]{amsart}

\usepackage{graphicx, amsmath, amssymb, amscd, amsthm, euscript, psfrag, amsfonts,bm}
\DeclareMathAlphabet{\mathpzc}{OT1}{pzc}{m}{it}

\newtheorem{thm}[equation]{Theorem}

\newtheorem{Con}[equation]{Conjecture}
\newtheorem{rmk}[equation]{Remark}
\newtheorem{cor}[equation]{Corollary}
\newtheorem{prop}[equation]{Proposition}

\newtheorem{lem}[equation]{Lemma}
\newtheorem{dfn}[equation]{Definition}
\numberwithin{equation}{section}

\newcommand{\be}{begin{equation}}

\newcommand{\bH}{\mathbb H}
\newcommand{\LG}{\Lambda(\Gamma)}
\newcommand{\tS}{\tilde S}

\newcommand{\q}{\mathbb{Q}}

\newcommand{\e}{\epsilon}
\newcommand{\z}{\mathbb{Z}}
\renewcommand{\q}{\mathbb{Q}}
\newcommand{\N}{\mathbb{N}}
\renewcommand{\c}{\mathbb{C}}
\newcommand{\br}{\mathbb{R}}

\newcommand{{\grinv}}{{\Cal G}^{-r}}

\newcommand{\ba}{\backslash}\newcommand{\bs}{\backslash}

\newcommand{\G}{\Gamma}

\newcommand{\g}{\gamma}

\newcommand{\Cal}{\mathcal}
\renewcommand{\P}{\mathcal P}\newcommand{\B}{\mathcal B}
\newcommand{\bG}{\mathbf G}
\newcommand{\la}{\langle}
\newcommand{\ra}{\rangle}

\newcommand{\SL}{\operatorname{SL}}

\newcommand{\bp}{\begin{pmatrix}}
\newcommand{\ep}{\end{pmatrix}}
\renewcommand{\be}{\begin{equation}}
\newcommand{\ee}{\end{equation}}

\renewcommand{\bp}{{\rm bp}}

\newcommand{\mC}{\mathcal C}

\renewcommand{\O}{\operatorname{O}}
\newcommand{\SO}{\operatorname{SO}}

\renewcommand{\H}{\operatorname{H}}
\newcommand{\T}{\operatorname{T}}

\newcommand{\bh}{\partial(\mathbb{H}^n)}

\newcommand{\Vol}{\op{Vol}}

\newcommand{\PS}{\op{PS}}\newcommand{\Haar}{\op{Haar}}

\newcommand{\op}{\operatorname}\newcommand{\supp}{\operatorname{supp}}

\renewcommand{\deg}{\text{DEP}}

\newcommand{\BR}{\operatorname{BR}}
\newcommand{\BMS}{\operatorname{BMS}}

\renewcommand{\setminus}{-}

\newcommand{\Om}{\Omega}

\renewcommand{\S}{\mathcal S}

\def\bbc{\mathbb{C}}
\def\bbr{\mathbb{R}}

\def\bbz{\mathbb{Z}}
\def\om{\omega}

\def\scal{\mathcal{S}}
\def\ocal{\mathcal{O}}

\def\B{\mathcal{B}}

\def\M{\mathcal{M}}

\def\P{{\rm P}}
\newcommand\Hom{{\rm Hom}}
\def\lam{\lambda}

\def\cfun{{\rm c}}

\newcommand{\cA}{\mathsf A}

\begin{document}

\title[Effective counting]
{Matrix coefficients, counting and primes for orbits of geometrically finite groups}

\author{Amir Mohammadi}
\address{Department of Mathematics, The University of Texas at Austin, Austin, TX 78750}
\email{amir@math.utexas.edu}

\author{Hee Oh}
\address{Mathematics department, Brown university, Providence, RI 02912
and Korea Institute for Advanced Study, Seoul, Korea}
\email{heeoh@math.brown.edu}
\thanks{The authors are supported in part by NSF grants.}




\begin{abstract}
 Let $G:=\SO(n,1)^\circ$ and $\G<G$ be a geometrically finite Zariski dense subgroup with critical exponent $\delta$
bigger than $(n-1)/2$. Under a spectral gap  hypothesis on $L^2(\G\ba G)$, which is always satisfied when $\delta>(n-1)/2$ for $n=2,3$ and when $\delta>n-2$ for $n\geq 4$,
we obtain an {\it effective} archimedean counting result for a discrete orbit of $\G$ in a homogeneous space $H\ba G$
where $H$ is the trivial group, a symmetric subgroup or a horospherical subgroup.
More precisely,
we show that for any effectively well-rounded family $\{\B_T\subset H\ba G \}$ of compact subsets,
there exists $\eta>0$ such that $$\#[e]\G\cap \B_T=\mathcal M(\B_T) +O(\mathcal M(\B_T)^{1-\eta})$$
 for an explicit measure $\mathcal M$ on $H\ba G$ which depends on $\G$.
 We also apply the affine sieve and describe the distribution of almost primes on orbits of $\G$ in arithmetic settings.


 One of key ingredients in our approach is
an effective asymptotic formula for the matrix coefficients of $L^2(\G\ba G)$
that we prove by combining methods from spectral analysis, harmonic analysis and ergodic theory.
We also prove exponential mixing of the frame flows with respect to the Bowen-Margulis-Sullivan measure.
\end{abstract}

\maketitle
\tableofcontents
\section{Introduction}
Let $n\ge 2$ and let $G$ be the identity component of the special orthogonal group $\SO(n,1)$. As well known,
$G$ can be considered as the group of orientation preserving isometries of the hyperbolic space
$\bH^n$. A discrete subgroup $\G$ of $G$ is called {\it geometrically finite} if the unit neighborhood of
its convex core\footnote{The {\em convex core} $C_\G\subset \G\bs\bH^n$ of $\G$ is the
image of the minimal convex subset of $\bH^n$ which contains all geodesics connecting
any two points in the limit set of $\G$.} has finite Riemannian volume.
As any discrete subgroup admitting a finite sided polyhedron as a fundamental domain in $\bH^n$ is geometrically finite,
this class of discrete subgroups provides a natural generalization of lattices in $G$.
In particular, for $n=2$, a discrete subgroup of $G$ is geometrically finite if and only if it is finitely generated.

In the whole introduction, let $\G$ be a torsion-free geometrically finite, Zariski dense,
discrete subgroup of $G$.
We denote by $\delta$ the critical exponent of $\G$.
Note that any discrete subgroup of $G$ with $\delta>(n-2)$ is Zariski dense in $G$.
 The main aim of this paper is to obtain {\it effective} counting results for discrete orbits of $\G$ in
$H\ba G$,  where $H$ is the trivial group, a
 symmetric subgroup or a horospherical subgroup of $G$,
 and to discuss their applications in
the affine sieve on $\G$-orbits in an arithmetic setting.
Our results are formulated under a suitable spectral gap hypothesis for $L^2(\G\ba G)$ (see Def. \ref{sng} and \ref{sngg}). This hypothesis on $\G$ is known to be true
if the critical exponent $\delta$ is strictly bigger than $n-2$.
 Though we believe that the condition $\delta>(n-1)/2$ should be sufficient to guarantee this hypothesis,
it is not yet known in general (see \ref{conj}).



For $\G$ lattices, i.e., when $\delta=n-1$, both the effective counting and
 applications to an affine sieve  have been extensively studied (see \cite{DRS}, \cite{EM}, \cite{BeO}, \cite{Mau}, \cite{GN}, \cite{LS},\cite{NS}, \cite{GN1}, etc. as well as survey articles \cite{Oh}, \cite{Oh1} \cite{Kow}, \cite{Kow1}).
Hence our main focus is when $\G$ is of infinite co-volume in $G$.

\medskip

\subsection{Effective asymptotic of Matrix coefficients for $L^2(\G\ba G)$}\label{section1}
We begin by describing an effective asymptotic result on the matrix coefficients for $L^2(\G\ba G)$, which
is a key ingredient in our approach as well as of independent interest.
When $\G$ is not a lattice,
a well-known theorem of Howe and Moore \cite{HM}
implies that for any $\Psi_1, \Psi_2\in L^2(\G\ba G)$,
the matrix coefficient $$\la a \Psi_1, \Psi_2\ra:=\int_{\G\ba G} \Psi_1(ga)\overline{\Psi_2(g)} dg$$
 decays to zero as $a\in G$ tends to
infinity (here $dg$ is a $G$-invariant measure on $\G\ba G$).
Describing the precise asymptotic is much more involved.
Fix a Cartan decomposition $G=KAK$ where $K$ is a maximal compact subgroup and
 $A$ is a one-parameter subgroup of diagonalizable elements. Let $M$ denote the centralizer of $A$ in $K$.
The quotient spaces $G/K$ and $G/M$ can be respectively identified with $\bH^n$ and its unit tangent bundle $\T^1(\bH^n)$, and
we parameterize elements of $A=\{a_t:t\in \br\}$ so that the right translation action of $a_t$ in $G/M$ corresponds to
the geodesic flow on $\T^1(\bH^n)$ for time $t$.

We let $\{m_x: x\in \bH^n\}$ and $\{\nu_x: x\in \bH^n\}$ be $\G$-invariant conformal densities of
dimensions $(n-1)$ and $\delta$ respectively, unique up to scalings.
Each $\nu_x$ is a finite measure on the limit set of $\G$, called the Patterson-Sullivan measure viewed from $x$.
Let $m^{\BMS}, m^{\BR}, m^{\BR}_*$ and $m^{\Haar}$ denote, respectively, the Bowen-Margulis-Sullivan measure,
 the Burger-Roblin measures for the expanding and the contracting
horospherical foliations, and the Liouville-measure on the unit tangent bundle $\T^1(\G\ba \bH^n)$, all defined with respect to the fixed
 pair of $\{m_x\}$ and $\{\nu_x\}$ (see Def. \ref{bm}).
Using the identification $\T^1(\G\ba \bH^n)=\G\ba G/M$, we may extend these measures to right $M$-invariant measures
on $\G\ba G$, which we will denote by the same notation and call them the BMS, the BR, the BR$_*$,
the Haar measures for simplicity.
We note that for $\delta<n-1$, only the BMS measure has finite mass \cite{OS}.



In order to formulate a notion of a spectral gap for $L^2(\G\ba G)$,
denote by $\hat G$ and $\hat M$ the unitary dual of $G$ and $M$ respectively.
A representation $(\pi,\mathcal H) \in \hat G$ is called {\it tempered} if for any 
$K$-finite $v\in\mathcal H$, the associated matrix coefficient function $g\mapsto \la \pi(g)v, v\ra$ belongs to $L^{2+\epsilon}(G)$ for any $\e>0$; {\it non-tempered} otherwise.
The non-tempered part
of $\hat G$ consists of the trivial representation, and complementary series representations $\mathcal U(\upsilon, s-n+1)$
parameterized by $\upsilon\in \hat M$ and $s\in I_\upsilon$,
where $I_\upsilon\subset (\tfrac{n-1}2, n-1)$ is an interval depending on $\upsilon$.
This was obtained by Hirai \cite{Hi} (see also \cite[Prop. 49, 50]{KS}).
Moreover $\mathcal U(\upsilon, s-n+1)$ is spherical (i.e., has a non-zero $K$-invariant vector) if and only if $\upsilon$ is the trivial representation $1$;
 see discussion in section \ref{ss;standard-comp-series}.

By the works of  Lax-Phillips \cite{LaxPhillips} and Sullivan \cite{Sullivan1979}, if $\delta>\tfrac{n-1}2$,
$\mathcal U(1,\delta-n+1)$ occurs as a subrepresentation of $L^2(\G\ba G)$ with multiplicity one,
and  $L^2(\G\ba G)$ possesses {\it spherical} spectral gap, meaning that
there exists $\tfrac{n-1}2< s_0<\delta$ such that $L^2(\G\ba G)$ does not weakly contain{\footnote{
 for two unitary representations $\pi$ and $\pi'$ of $G$,
  $\pi$ is said to be weakly contained in $\pi'$ (or $\pi'$ weakly contains $\pi$) if every diagonal matrix coefficient of $\pi$ can be approximated, uniformly on compact subsets, by
convex combinations of diagonal matrix coefficients of $\pi'$.
}}
any {\it spherical} complementary series representation $\mathcal U(1,s-n+1)$, $s\in (s_0, \delta)$.
The following notion of a spectral gap concerns both the spherical and non-spherical parts of $L^2(\G\ba G)$.

\begin{dfn}\label{intro_strong_gap}\label{sng}  \rm  We say that $L^2(\G\ba G)$ has a {\it strong spectral gap} if
   \begin{enumerate}
    \item $L^2(\G\ba G)$ does not contain any $\mathcal U(\upsilon,\delta-n+1)$ with $\upsilon\ne 1$;
\item there exist
$\tfrac{n-1}{2}<s_0(\G)<\delta$ such that  $L^2(\G\ba G)$ does not weakly contain any $\mathcal U(\upsilon,s-n+1)$ with $s\in (s_0(\G), \delta)$ and $\upsilon\in \hat M$. \end{enumerate}
\end{dfn}

For $\delta\le \tfrac{n-1}2$, the Laplacian spectrum of $L^2(\G\ba \bH^n)$ is continuous \cite{LaxPhillips}; this implies that
there is no spectral gap for $L^2(\G\ba G)$.

\begin{Con}[Spectral gap conjecture]\label{conj}
 If $\G$ is a geometrically finite and Zariski dense subgroup of $G$ with $\delta>\tfrac{n-1}2$,  $L^2(\G\ba G)$ has a strong spectral gap.
\end{Con}

If $\delta>(n-1)/2$ for $n=2,3$, or if $\delta>(n-2)$ for $n\ge 4$, then
$L^2(\G\ba G)$ has a strong spectral gap (Theorem \ref{nm2}).

Our main theorems are proved under the following slightly weaker spectral gap property assumption:
\begin{dfn}\label{introsg}\label{sngg}  \rm  We say that $L^2(\G\ba G)$ has a {\it spectral gap} if there exist
$\tfrac{n-1}{2}<s_0=s_0(\G)<\delta$ and $n_0=n_0(\G)\in \N$ such that 
   \begin{enumerate}
    \item the multiplicity of $\mathcal U(\upsilon,\delta-n+1)$ contained in $L^2(\G\ba G)$ is at most $\dim(\upsilon)^{n_0}$ for any $\upsilon\in \hat M$;
\item  $L^2(\G\ba G)$ does not weakly contain any  $\mathcal U(\upsilon,s-n+1)$ with $s\in (s_0, \delta)$ and  $\upsilon\in \hat M$.\end{enumerate}
  The pair $(s_0(\G), n_0(\G))$ will be referred to as the spectral gap data for $\G$.
\end{dfn}

In the rest of the introduction, we impose the following hypothesis on $\G$:
$$\text{ $L^2(\G\ba G)$ has a spectral gap.}$$


\begin{thm}\label{harmixing}
  There exist $\eta_0>0$ and $\ell\in \N$ (depending only on the spectral gap data for $\G$) such that
 for any real-valued $\Psi_1, \Psi_2\in C_c^\infty(\G\ba G)$, as $t\to \infty$,
\begin{multline*} e^{(n-1-\delta)t} \int_{\G\ba G} \Psi_1 (ga_t) {\Psi_2}(g) dm^{\Haar}(g)
 \\= \frac{m^{\BR}(\Psi_1)\cdot m^{\BR}_*(\Psi_2)}{|m^{\BMS}|}
+O(\mathcal S_\ell (\Psi_1)\mathcal S_\ell(\Psi_2) e^{-\eta_0 t} )\end{multline*}
where $\mathcal S_\ell(\Psi_i)$ denotes the $\ell$-th Sobolev $L^2$-norm of $\Psi_i$ for each $i=1,2$.
\end{thm}

\begin{rmk}\label{kk}\rm
We remark that if either $\Psi_1$ or $\Psi_2$ is $K$-invariant, then
Theorem \ref{harmixing} holds for any Zariski dense $\G$ with $\delta> \tfrac{n-1}2$ (without
the spectral gap hypothesis), as the spherical spectral gap of $L^2(\G\ba G)$ is sufficient
to study the matrix coefficients associated to spherical vectors.
\end{rmk}

Let $\mathcal H_{\delta}^\dag$ denote the sum of of all complementary series representations of parameter $\delta$
contained in $L^2(\G\ba G)$, and let
$P_{\delta}$ denote the projection operator from $L^2(\G\ba G)$
to $\mathcal H_{\delta}^\dag$.
By the spectral gap hypothesis on $L^2(\G\ba G)$,
the main work in the proof of Theorem \ref{harmixing}
 is to understand the asymptotic
of $\la a_t P_{\delta}(\Psi_1), P_{\delta}(\Psi_2)\ra$ as $t\to \infty$.
Building up on the work of Harish-Chandra on the asymptotic
behavior of the Eisenstein integrals (cf. \cite{Wa}, \cite{Wa2}), we first obtain an asymptotic formula for
$\la a_t v, w\ra $ for all $K$-finite vectors $v,w\in \mathcal H_{\delta}^\dag$ (Theorem \ref{key}).
This extension alone does not give the formula of the leading term
of $\la a_t P_{\delta} (\Psi_1), P_{\delta} (\Psi_2)\ra $ in terms of functions $\Psi_1$ and $\Psi_2$;
however, an ergodic theorem of Roblin \cite{Roblin2003} enables us to identify the main term as given in Theorem \ref{harmixing}.

\medskip






\subsection{Exponential mixing of frame flows}
Via the identification of the space $\G\ba G$ with the frame bundle over the hyperbolic manifold $\G\ba \bH^n$,
the right translation action of $a_t$  on $\G \ba G$ corresponds to the frame flow for time $t$.
The BMS measure $m^{\BMS}$ on $\G\ba G$ is known to be mixing for the frame flows (\cite{FS}, \cite{Wi}).
We deduce the following exponential mixing from Theorem \ref{harmixing}:
for a compact subset $\Om$ of $\G\ba G$, we denote by $C^\infty(\Om)$ the set of all smooth functions on $\G\ba G$
with support contained in $\Om$.

\begin{thm}\label{mmix}
There exist $\eta_0>0$
and $\ell\in \N$ such that for any compact subset $\Om \subset\G\ba G$, and
for any $\Psi_1, \Psi_2\in C^\infty(\Om)$, as $t\to \infty$,
\begin{multline*}\int_{\G\ba G} \Psi_1 (ga_t) \Psi_2(g) dm^{\BMS}(g)
\\ = \frac{m^{\BMS}(\Psi_1)\cdot m^{\BMS}(\Psi_2)}{|m^{\BMS}|}
+O(\mathcal S_\ell (\Psi_1)\mathcal S_\ell(\Psi_2) e^{-\eta_0 t} )\end{multline*}
where the implied constant depends only on $\Om$.
\end{thm}

For $\G$ convex co-compact, Theorem \ref{mmix} was known  for any $\delta>0$ by
Stoyanov \cite{St}, based on the approach developed by Dolgopyat \cite{Do}; however
when $\G$ has cusps, this theorem seems to be new even for $n=2$.

\medskip

\subsection{Effective equidistribution of orthogonal translates of an $H$-orbit}
When $H$ is a horospherical subgroup
or a symmetric subgroup of $G$, we can relate the asymptotic distribution
of orthogonal translates of a closed orbit $\G\ba \G H$ to the matrix coefficients of $L^2(\G\ba G)$.
We fix a generalized Cartan decomposition $G=HAK$.
 We parameterize $A=\{a_t\}$ as in section \ref{section1}, and for $H$ horospherical, we will assume that $H$ 
 is the expanding horospherical subgroup for $a_t$, that is, $H=\{g\in G: a_tga_{-t}\to e\text{ as $t\to \infty$}\}$.
Let $\mu_H^{\Haar}$ and $\mu_H^{\PS}$ be respectively the
$H$-invariant measure on $\G\ba \G H$ defined with respect to $\{m_x\}$ and the skinning measure on $\G\ba \G H$ defined with respect to $\{\nu_x\}$, introduced in \cite{OS} (cf. \eqref{psdef}).

\begin{thm}\label{effa} Suppose
 that $\G\ba \G H$ is closed and that $|\mu_H^{\PS}|<\infty$.
There exist $\eta_0>0$ and $\ell \in \N$ such that
for any compact subset $\Om \subset \G\ba G$, any $\Psi\in C^\infty(\Om)$ and any bounded $\phi\in C^\infty(\G\cap H\ba H)$, as $t\to \infty$,
\begin{multline*}  e^{(n-1-\delta)t} \int_{h\in \G \ba \G H} \Psi(ha_t ) \phi(h) d\mu_H^{\Haar}(h)
\\  =\frac{1 }{|m^{\BMS}|}\mu^{\PS}_{H}(\phi)m^{\BR}(\Psi)   + O(\mathcal S_\ell(\Psi)\cdot S_\ell(\phi)  e^{-\eta_0 t})\end{multline*}
with the implied constant depending only on  $\Om$. \end{thm}

For $H$ horospherical, $|\mu_H^{\PS}|<\infty$ is automatic for $\G\ba \G H$ closed. For
$H$ symmetric (and hence locally isomorphic to $\SO(k,1)\times \SO(n-k)$),  the criterion for the finiteness of $\mu_H^{\PS}$ 
has been obtained in \cite{OS} (see  Prop. \ref{cri});
in particular,  $|\mu_H^{\PS}|<\infty$ provided $\delta>n-k$.

Letting $Y_\Om:=\{h\in (\G\cap H) \ba H: ha_t \in \Om \text{ for some $t>0$}\}$,
 note that $$\int\Psi(ha_t) \phi(h) d\mu^{\Haar}_H= \int_{Y_\Om}\Psi(ha_t) \phi(h) d\mu^{\Haar}_H$$ since $\Psi$ is supported in $\Om$.
In the case when $\mu_H^{\PS}$ is compactly supported, $Y_\Om$ turns out to be a compact subset and in this case,
  the so-called thickening method (\cite{EM}, \cite{KM}) is sufficient to deduce Theorem \ref{effa} from Theorem \ref{harmixing}, using the wave front property introduced in \cite{EM} (see \cite{BeO} for the effective version).
The case of $\mu_H^{\PS}$ not compactly supported is much more
intricate to be dealt with. Though we obtain a thick-thin decomposition of $Y_\Om$ with the thick part being compact and control
both the Haar measure and the skinning measure of the thin part (Theorem \ref{yo}),
the usual method of thickening the thick part does not suffice,
as the error term coming from the thin part overtakes the leading term.
The main reason for this phenomenon is because we are taking the integral with respect to $\mu^{\Haar}_H$
as well as multiplying the weight factor $e^{(n-1-\delta)t}$ in the left hand side of Theorem \ref{effa},
whereas the finiteness assumption is made on the skinning measure $\mu^{\PS}_H$. However
we are able to proceed by comparing the two measures $(a_t)_*\mu^{\PS}_H$ and $(a_t)_*\mu^{\Haar}_H$ via the the transversal intersections
of the orbits $\G\ba \G Ha_t$ with the weak-stable horospherical foliations (see the proof of Theorem \ref{meq} for more details). 

In the special case of $n=2,3$ and $H$ horospherical,
Theorem \ref{effa} was proved in \cite{KO}, \cite{KO2} and \cite{LO} by a different method.

\medskip



\subsection{Effective counting for a discrete $\G$-orbit in $H\ba G$}
In this subsection, we let $H$ be the trivial group, a horospherical subgroup or a symmetric subgroup, and assume
 that the orbit $[e]\G$ is discrete in $H\ba G$.
Theorems \ref{harmixing} and \ref{effa} are key ingredients in understanding
the asymptotic of the number $\#([e]\G \cap \B_T)$  for a given family $\{\B_T\subset H\ba G \}$
of growing compact subsets, as observed in \cite{DRS}.

We will first describe a Borel measure $\mathcal M_{H\ba G}=\mathcal M_{\H\ba G}^\G$ on $H\ba G$, depending on $\G$, which
turns out to describe the distribution of $[e]\G$.
Let  $o\in \bH^n$ be the point fixed by $K$,  $X_0\in \T^1(\bH^n)$ the vector fixed by $M$, $X_0^{+},
X_0^-\in \partial(\bH^n)$
  the forward and the backward endpoints of $X_0$ by the geodesic flow, respectively
  and $\nu_o$ the Patterson-Sullivan measure on $\partial(\bH^n)$ supported on the limit set of $\G$, viewed from $o$.

\begin{dfn}\label{group} \rm For $H$ the trivial subgroup $\{e\}$,
define a Borel measure $\mathcal M_{G}=\mathcal M_G^\G$ on $G$ as follows: for $\psi\in C_c(G)$,
\begin{equation*}\label{mdeg}
\mathcal M_{G} (\psi)=\tfrac{ 1}{|m^{\BMS}|} \int_{k_1a_tk_2\in K A^+K} \psi(k_1a_tk_2)
  e^{\delta t} d\nu_o(k_1X_0^+)dt d\nu_o(k_2^{-1}X_0^-).\end{equation*}
\end{dfn}

\begin{dfn} \label{sd} \rm For $H$ horospherical or symmetric,  we have either $G=HA^+K$ or $G=HA^+K\cup HA^-K$ (as a disjoint union except for the identity element)
where $A^\pm=\{a_{\pm t}: t\ge 0\}$.

Define a Borel measure $\mathcal M_{H\ba G}=\mathcal M_{H\ba G}^\G$ on $H\ba G$ as follows: for $\psi\in C_c(H\ba G)$,
\begin{equation*}
\mathcal M_{H\ba G} (\psi)=\begin{cases}\tfrac{ |\mu_{H}^{\PS}|}{|m^{\BMS}|} \int_{a_tk\in A^+K} \psi([e]a_tk)
  e^{\delta t} dt d\nu_o(k^{-1}X_0^-)&\text{if $G=HA^+K$}
\\ \sum \tfrac{ |\mu_{H,\pm}^{\PS}|}{|m^{\BMS}|}\int_{a_{\pm t}k\in A^\pm K}  \psi([e]a_{\pm t}k)
e^{\delta t} dt d\nu_o(k^{-1}X_0^{\mp})
 &\text{otherwise}, \end{cases}\end{equation*}
where $\mu_{H,-}^{\PS}$ is the skinning measure on $\G\cap H\ba H$
in the negative direction, as defined
in \eqref{hmi}.\end{dfn}



\begin{dfn}\label{adm} \rm
  For a family $\{\B_T\subset H\ba G\}$ of compact subsets
 with $\mathcal M_{H\ba G}(\B_T)$ tending to infinity as $T\to \infty$, we say that $\{\B_T\}$
is {\it effectively well-rounded with respect to $\G$}
if there exists $p>0$ such that
for all small $\e>0$ and large $T\gg 1$:
$$\mathcal M_{H\ba G} (\B_{T,\e}^+ - \B_{T,\e}^-) = O(\e^p \cdot \mathcal M_{H\ba G} (\B_T)) $$
where $\B_{T,\e}^+=G_\e \B_T G_\e$ and $\B_{T,\e}^-=\cap_{g_1, g_2\in G_\e} g_1\B_T g_2$ if $H=\{e\}$;
and $\B_{T,\e}^+=\B_{T}G_\e$ and $\B_{T,\e}^-= \cap_{g\in G_\e}\B_T g$ if $H$ is horospherical or symmetric.
Here $G_\e$ denotes a symmetric $\e$-neighborhood of $e$ in $G$ with respect to a left invariant
 Riemannian metric on $G$.
\end{dfn} Since any two left-invariant
Riemannian metrics on $G$ are Lipschitz equivalent to each other,
the above definition is independent of the choice of a Riemannian metric used in the definition of $G_\e$.

 See Propositions \ref{exam}, \ref{sectore} and \ref{normw} for examples of effectively well-rounded families.
For instance, if $G$ acts linearly from the right on a finite dimensional linear space $V$ and $H$ is the stabilizer of $w_0\in V$,
then the family of norm balls $B_T:=\{Hg\in H\ba G: \|w_0g\|<T\}$ is effectively well-rounded.

If $\G$ is a lattice in $G$,
then $\mathcal M_{H\ba G}$ is essentially the leading term of the invariant measure in $H\ba G$
and hence the definition \ref{adm} is equivalent to
the one given in \cite{BeO}, which is an effective version of the well-roundedness condition given in \cite{EM}.
Under the additional assumption that $H\cap \G$ is a lattice in $H$, it is known that
 if $\{\B_T\}$ is effectively well-rounded, then
\begin{equation} \label{lc} \# ([e]\G \cap \B_T )=
\op{Vol}(\B_T)  + O( \op{Vol}(\B_T)^{1-\eta_0})\end{equation}
for some $\eta_0>0$, where $\Vol$ is computed with respect to a suitably normalized invariant measure on $H\ba G$
 (cf. \cite{DRS}, \cite{EM}, \cite{Mau}, \cite{GN1}, \cite{BeO}).

We present a generalization of \eqref{lc}. In the next two theorems \ref{mc} and \ref{bisec}, we let
$\{\G_d: d\in I\}$ be a family of subgroups of $\G$ of finite index such that $\G_d\cap H=\G\cap H$. We assume that
 $\{\G_d:d\in I\}$ has a uniform spectral gap in the sense that
$\sup_d s_0(\G_d)<\delta$ and $\sup_d n_0(\G_d)<\infty$.

For our intended application to the affine sieve, we formulate our effective results
uniformly for all $\G_d$'s.

\begin{thm}\label{mc}
Let $H$ be the trivial group, a horospherical subgroup or a symmetric subgroup.
When $H$ is symmetric, we also assume that $|\mu_H^{\PS}|<\infty$.
If $\{\B_T\}$ is effectively well-rounded with respect to $\G$,
 then there exists $\eta_0>0$ such that for any $d\in I$ and for any $\gamma_0\in \G$,
$$\# ([e]\G_d\gamma_0\cap \B_T )= \tfrac{ 1}{[\G:\G_d] }
\mathcal M_{H\ba G}(\B_T)  + O( \mathcal M_{H\ba G}(\B_T)^{1-\eta_0})$$
where $\mathcal M_{H\ba G}=\mathcal M_{H\ba G}^\G$ and the implied constant is independent of $\G_d$ and $\gamma_0\in \G$.
\end{thm}

See Corollaries \ref{sector} and \ref{vsector} where we have applied Theorem \ref{mc} to
 sectors and norm balls.

\begin{rmk} \rm
Theorem \ref{mc} can be used to provide an effective version of circle-counting theorems
studied in \cite{KO}, \cite{LO}, \cite{OS1} and \cite{Oh0} (as well as its higher dimensional analogues for 
sphere packings discussed in \cite{Oh}).
\end{rmk}

We also formulate our counting statements for bisectors in the $HAK$ decomposition, motivated by recent applications
in \cite{BK1} and \cite{BK2}.
Let ${\tau}_1\in C^\infty_c(H)$ and $\tau_2\in C^\infty(K)$,
and define $\xi_T^{\tau_1,\tau_2}\in C^\infty(G)$ as follows: for $g=hak\in HA^+K$,
$$\xi_T^{\tau_1,\tau_2}(g)= \chi_{A_T^+}(a) \cdot \int_{H\cap M} \tau_1(hm)\tau_2( m^{-1}k) dm  $$
where $\chi_{A_T^+}$ denotes the characteristic function of $A_T^+=\{a_t: 0<t<\log T\}$
and $d_{H\cap M}$ is the probability Haar measure of $H\cap M$.
Since $hak=h'ak'$ implies that $h=h'm$ and $k=m^{-1}k'$ for some $m\in H\cap M$,
$\xi_T^{\tau_1,\tau_2}$ is well-defined.

\begin{thm}\label{bisec}


There exist $\eta_0>0$ and $\ell \in \N$ such that
for any compact subset $H_0$ of $H$ which injects to $\G\ba G$,
 any ${\tau}_1\in C^\infty(H_0)$, $\tau_2\in C^\infty(K)$,
  any $\gamma_0\in \G$ and any $d\in I$,
$$\sum_{\gamma\in \G_d\gamma_0}
\xi_T^{\tau_1,\tau_2}(\gamma) = \frac{\tilde \mu_{H}^{\PS}(\tau_1)\cdot  \nu_o^* (\tau_2)}{\delta\cdot
[\G:\G_d]\cdot |m^{\BMS}|}
T^\delta + O( \S_\ell(\tau_1) \S_\ell(\tau_2)  T^{\delta -\eta_0})$$
where $\nu_o^* (\tau_2)=\int_K\tau_2(k) d\nu_o(k^{-1}X_0^-)$ and $\tilde \mu_H^{\PS}$ is the skinning measure
on $H$ with respect to $\G$ and the implied constant depends only on $\G$ and $H_0$.
\end{thm}



Theorem \ref{bisec} also holds when $\tau_1$ and $\tau_2$ are characteristic functions of
the so-called {\it admissible} subsets (see Corollary \ref{bik} and Proposition \ref{exam}).

 We remark that unless $H=K$,
 Corollary \ref{sector}, which is a special case of Theorem \ref{mc} for sectors in $H\ba G$,
  does not follow from Theorem \ref{bisec}, as the latter
deals only with compactly supported functions $\tau_1$.
For $H=K$, Theorem \ref{bisec} was earlier shown in \cite{BKS} and \cite{V} for $n=2,3$ respectively.

\begin{rmk}\rm Non-effective versions of Theorems \ref{effa}, \ref{mc}, and \ref{bisec} were obtained in \cite{OS}
for a more general class of discrete groups, that is, any non-elementary discrete subgroup admitting finite BMS-measure. \end{rmk}
\medskip




\medskip

\subsection{Application to Affine sieve}
One of the most exciting applications of Theorem \ref{mc} can be found in connection with Diophantine
problems on orbits of $\G$.
Let ${\bf G}$ be a $\q$-form of $G$, that is, $\bG$ is a connected algebraic group defined over $\q$ such that $G={\bf G}(\br)^\circ$.
Let $\bf G$ act on an affine space $V$ via a $\q$-rational representation in a way that ${\bf G}(\z)$ preserves $V(\z)$.
 Fix a non-zero vector $w_0\in V(\z)$
and denote by ${\bf H}$ its stabilizer subgroup and set $H=\bf H(\br)$.
We consider one of the following situations: (1) $H$ is a symmetric
subgroup of $G$ or the trivial subgroup; (2) $w_0{\bf G} \cup \{0\}$ is Zariski closed and
 $H$ is a compact extension of a horospherical subgroup of $G$.

 In the case (1), $w_0\bG$ is automatically Zariski closed by \cite{GOS1}.
  Set $W:=w_0\bG$ and $w_0\bG\cup\{0\}$ respectively for (1) and (2).

Let $\G$ be a geometrically finite and Zariski dense subgroup of $G$ with $\delta >\tfrac{n-1}2$, which is contained
in ${\bf G}(\z)$. If $H$ is symmetric, we assume that $|\mu_H^{\PS}|<\infty$.

For a positive integer $d$, we denote by $\G_d$ the congruence subgroup of $\G$ which consists
of $\gamma\in \G$ such that $\gamma \equiv e$ mod $d$.
For the next two theorems \ref{affine_u} and \ref{affine_l}
we assume that
there exists a finite set $S$ of primes that
the family $\{\G_d: \text{$d$ is square-free with no prime factors in $S$}\}$
has a uniform spectral gap.
This property always holds if $\delta>(n-1)/2$ for $n=2,3$ and if $\delta>n-2$ for $n\ge 4$ via the recent works of Bourgain,Gamburd, Sarnak (\cite{BGS}, \cite{BGS2}) and
Salehi-Golsefidy and Varju \cite{SV} together with the classification of the unitary dual of $G$ (see Theorem \ref{usp}).

Let $F\in \q [W]$ be an integer-valued polynomial on the orbit $w_0\G$.
 Salehi-Golsefidy  and Sarnak \cite{SS}, generalizing \cite{BGS}, showed that for some $R>1$,
the set of ${\bf x}\in w_0\G$ with $F({\bf x})$ having at most $R$ prime factors
is Zariski dense in $w_0{\bf G}$.
The following are quantitative versions:
Letting $F=F_1 F_2 \cdots F_r$ be a factorization into irreducible
polynomials in $\q[W]$, assume that all $F_j$'s are irreducible in $\c[W]$ and integral on $w_0\G$.
Let $\{\B_T\subset w_0G : T\gg 1 \}$ be an effectively well-rounded family of subsets with respect to $\G$.

\begin{thm}[Upper bound for primes] \label{affine_u}
For all $T\gg 1$,
$$\{{\bf x}\in w_0\G\cap \B_T : F_j({\bf x}) \text{ is prime for $j=1, \cdots, r$}\}\ll
\frac{\M_{w_0G}(\B_T) }{(\log \M_{w_0G}(\B_T)) ^r} .$$
\end{thm}

\begin{thm} [Lower bound for almost primes] \label{affine_l}
Assume further that $\max_{x\in \B_T} \|x\| \ll \M_{w_0G}(\B_T)^\beta$ for some $\beta>0$, where $\|\cdot \|$ is any norm on $V$.
Then
there exists $R=R(F, w_0\G, \beta ) \ge 1$ such that for all $T\gg 1$,
$$\{{\bf x}\in w_0\G\cap \B_T: F({\bf x}) \text{ has at most $R$ prime factors}\}\gg
\frac{\M_{w_0G}(\B_T)}{(\log \M_{w_0G}(\B_T)) ^r}  .$$
\end{thm}

Observe that these theorems provide
a description of the asymptotic distribution of almost prime vectors, as $\B_T$ can be taken arbitrarily.

\begin{rmk}
 \rm In both theorems above, if $\mathcal B_T$ are all $K$-invariant subsets,
our hypothesis on the uniform spectral gap for the family $\{\G_d\}$ can be disposed again,
as the uniform {\it spherical} spectral gap property proved in \cite{SV} and \cite{BGS} is sufficient in this case.
\end{rmk}

For instance, Theorems \ref{affine_u} and \ref{affine_l} can be
applied to the norm balls $\B_T=\{{\bf x}\in w_0G: \|{\bf x}\|<T\}$ and in this case
$\M_{w_0G}(\B_T)\asymp T^{\delta/\lam}$
where $\lambda$ denotes the log of the largest eigenvalue of $a_1$ on the $\br$-span of $w_0G$.

In order to present a concrete example, we consider an integral quadratic form $Q(x_1, \cdots, x_{n+1})$ of signature $(n,1)$
and for an integer $s\in \z$, denote by $W_{Q,s}$ the affine quadric given by
$$\{{\bf x}: Q({\bf x})=s\} .$$
As well-known, $W_{Q,s}$ is a one-sheeted hyperboloid if $s>0$, a two-sheeted hyperboloid if $s<0$ and a cone if $s=0$.
We will assume that $Q({\bf x})=s$ has a non-zero integral solution, so pick $w_0\in W_{Q,s}(\z)$.
 If $s\ne 0$, the stabilizer subgroup $G_{w_0}$ is symmetric; more precisely, locally isomorphic to $\SO(n-1,1)$ (if $s>0$) or $\SO(n)$ (if $s<0$)
 and if $s=0$, $G_{w_0}$ is a compact extension of a horospherical subgroup.
By the remark following Theorem \ref{effa}, the skinning measure $\mu_{G_{w_0}}^{\PS}$ is
finite if $n\ge 3$.
For $n=2$ and $s>0$,
$G_{w_0}$ is a one-dimensional subgroup consisting of diagonalizable elements, and
 $\mu_{G_{w_0}}^{\PS}$ is infinite only when the geodesic in $\bH^2$ stabilized by $G_{w_0}$
is divergent and goes into a cusp of a fundamental domain of $\G$ in $\bH^2$; in this case, we call $w_0$
 externally $\G$-parabolic, following \cite{OS}.
Therefore
the following are special cases of Theorems \ref{affine_u} and \ref{affine_l}:
\begin{cor}\label{affine2}
Let $\G$ be a geometrically finite and Zariski dense subgroup of $\SO_Q(\z)$ with $\delta>\tfrac{n-1}{2}$.
In the case when $n=2$ and $s>0$, we additionally assume that $w_0$ is not externally $\G$-parabolic.
Fixing a $K$-invariant norm $\|\cdot \|$ on $\br^{n+1}$, we have, for any $1\le r\le n+1$,
\begin{enumerate}
 \item $\{{\bf x}\in w_0\G:  \|{\bf x}\|<T, \;\; \text{$x_j$ is prime for all $j=1, \cdots, r$}\}\ll
\frac{T^\delta}{(\log T) ^r} ;$
\item for some $R>1$, $$\{{\bf x}\in w_0\G :  \|{\bf x}\|<T, \; x_1\cdots x_r \text{ has at most $R$ prime factors}\}\gg
\frac{T^\delta}{(\log T) ^r} .$$
\end{enumerate}
\end{cor}
The upper bound in (1) is sharp up to a multiplicative constant. 
The lower bound in (2) can also be stated for admissible sectors under the uniform spectral gap hypothesis
(cf. Corollary \ref{vsector}). 
Corollary \ref{affine2}  was previously obtained in cases when $n=2,3$ and $s\le 0$ (\cite{BGS2}, \cite{Ko}, \cite{KO},\cite{KO2}, \cite{LO}).


To explain how Theorems \ref{affine_u} and \ref{affine_l} follow from Theorem \ref{mc},
 let $\G_{w_0}(d)=
\{\gamma\in \G: w_0\gamma=w_0 \;\text{mod (d)}\}$ for each square-free integer $d$. Then
 $\op{Stab}_{\G_{w_0}(d)}(w_0)=\op{Stab}_\G(w_0)$
and the family $\{\G_{w_0}(d)\}$ admits a uniform spectral gap property as $\G_d<\G_{w_0}(d)$.
Hence Theorem \ref{mc} holds for the congruence family $\{\G_{w_0}(d): \text{$d$ is square-free, with no small primes}\}$,
providing a key axiomatic condition in executing the combinatorial
sieve (see \cite[6.1-6.4]{IK}, \cite[Theorem 7.4]{HR}, as well as \cite[Sec. 3]{BGS}).
When an explicit uniform spectral gap for $\{\G_d\}$ is known (e.g., \cite{G}, \cite{Magee}),
the number $R(F, w_0\G)$ can also be
explicitly computed in principle.


\medskip

The paper is organized as follows. In section \ref{de},
we recall the ergodic result of Roblin which gives the leading term of the matrix
coefficients for $L^2(\G\ba G)$. In section \ref{se:ma}, we  obtain an effective asymptotic
expansion for the matrix coefficients of the complementary series representations of $G$ (Theorem \ref{key})
 as well as for those
of $L^2(\G\ba G)$, proving Theorem \ref{harmixing}.
In section \ref{nw}, we study the reduction theory for
the non-wandering component of $\G\ba \G Ha_t$, describing its
 thick-thin decomposition; this is needed
as  $\G\ba \G H$ has infinite Haar-volume in general. We will see that the non-trivial dynamics of
$\G\ba \G Ha_t$ as $t\to\infty$ can be seen only within a subset of finite PS-measure.
In section \ref{localt}, for $\phi$ compactly supported,
 we prove Theorem \ref{effa} using Theorem \ref{harmixing} via thickening.
For a general bounded $\phi$, Theorem \ref{effa} is obtained via a careful study of the transversal intersections
in section \ref{sec;eff-equi}. Theorem \ref{mmix} is also proved in section \ref{sec;eff-equi}.
Counting theorems \ref{mc} and \ref{bisec} are proved in section \ref{s:count}
and Sieve theorems  \ref{affine_u} and \ref{affine_l} are proved in the final section \ref{sieve}.

\medskip

\noindent{\bf Acknowledgment:}
We are grateful to Gregg Zuckerman for pointing out an important mistake regarding the spectrum of
$L^2(\Gamma\backslash G)$ in an earlier version of this paper.
We would like to thank Peter Sarnak for helpful comments and Sigurdur Helgason for the reference on the work of Harish-Chandra.
We also thank Dale Winter for helpful conversations on related topics.

\section{Matrix coefficients in $L^2(\G\ba G)$ by ergodic methods}\label{de}
Throughout the paper, let $G$ be $\SO(n,1)^\circ=\op{Isom}^+(\bH^n)$ for $n\ge 2$, i.e., the group of orientation preserving isometries
of $(\bH^n, d)$, and
$\G< G$ be a non-elementary torsion-free geometrically finite group.
Let $\partial(\bH^n)$ denote the geometric boundary of
$\bH^n$. Let $\Lambda(\G)\subset \partial(\bH^n)$ denote the limit set of $\G$, and $\delta$ the critical exponent of $\G$,
which is known to be equal to the Hausdorff dimension of $\Lambda(\G)$ \cite{Sullivan1984}.

 A family of measures
$\{\mu_x:x\in \bH^n\}$ is called  a {\em $\G$-invariant conformal
density\/}
 of dimension $\delta_\mu > 0$  on $\partial(\bH^n)$, if  each
$\mu_x$ is a non-zero finite Borel measure on $\partial(\bH^n)$
satisfying for any $x,y\in \bH^n$, $\xi\in \partial(\bH^n)$ and
$\gamma\in \G$,
$$\gamma_*\mu_x=\mu_{\gamma x}\quad\text{and}\quad
 \frac{d\mu_y}{d\mu_x}(\xi)=e^{-\delta_\mu \beta_{\xi} (y,x)}, $$
where $\gamma_*\mu_x(F)=\mu_x(\gamma^{-1}(F))$ for any Borel
subset $F$ of $\partial(\bH^n)$. Here
$\beta_{\xi} (y,x)$ denotes the Busemann function:
$\beta_{\xi}(y,x)=\lim_{t\to\infty} d(\xi_t, y) -d(\xi_t, x)$ where
$\xi_t$ is a geodesic ray tending to $\xi$ as $t\to \infty$.


 We denote by $\{\nu_x\}$ the
Patterson-Sullivan density, i.e.,
a {\em $\G$-invariant conformal
density\/}
 of dimension $\delta$ and
 by $\{m_x:x\in \bH^n\}$ a {\em
Lebesgue density}, i.e., a $G$-invariant conformal density on the boundary
$\partial(\bH^n)$ of dimension $(n-1)$. Both densities
are determined unique up to scalar multiples.

Denote by $\{\mathcal G^t:t\in \br\}$ the geodesic flow on $\T^1(\bH^n)$.
For $u\in \T^1(\bH^n)$, we denote by $u^{\pm}\in \partial(\bH^n)$ the forward and the backward endpoints of the geodesic determined by
$u$, i.e., $u^{\pm}=\lim_{t\to \pm \infty}\mathcal G^t(u)$.
Fixing $o\in \bH^n$, the map
\[
u \mapsto (u^+, u^-, s=\beta_{u^-} (o,\pi(u)))
\]
is a homeomorphism between $\op{T}^1(\bH^n)$ with
 \[
 (\partial(\bH^n)\times \partial(\bH^n) - \{(\xi,\xi):\xi\in \partial(\bH^n)\})  \times \br.
 \]

Using this homeomorphism,
we define measures
$\tilde m^{\BMS}, \tilde m^{\BR}, \tilde m^{\BR}_*$,  $\tilde m^{\Haar}$ on $\T^1(\bH^n)$ as follows
(\cite{Bowen1971}, \cite{Margulisthesis}, \cite{Sullivan1984},
\cite{Burger1990}, \cite{Roblin2003}):

\begin{dfn} \label{bm} Set
\begin{enumerate}
 \item  $
d \tilde m^{\BMS}(u)= e^{\delta \beta_{u^+}(o,
\pi(u))}\;
 e^{\delta \beta_{u^-}(o, \pi(u)) }\; d\nu_o(u^+) d\nu_o(u^-) ds; $
\item
$
 d \tilde m^{\BR}(u)= e^{(n-1) \beta_{u^+}(o,
\pi(u))}\;
 e^{\delta \beta_{u^-}(o, \pi(u)) }\; dm_o(u^+) d\nu_o(u^-) ds;$

\item $
 d \tilde m^{\BR}_*(u)= e^{\delta \beta_{u^+}(o,
\pi(u))}\;
 e^{(n-1) \beta_{u^-}(o, \pi(u)) }\; d\nu_o(u^+) dm_o(u^-) ds; $

\item  $d \tilde m^{\Haar}(u)= e^{(n-1)\beta_{u^+}(o,
\pi(u))}\;
 e^{(n-1) \beta_{u^-}(o, \pi(u)) }\; dm_o(u^+) dm_o(u^-) ds.$
\end{enumerate}
 \end{dfn}
The conformal properties of $\{\nu_x\}$ and $\{m_x\}$ imply that these definitions are independent of the choice of $o\in \bH^n$.
We will extend these measures to $ G$; these extensions depend on the choice of
 $o\in \bH^n$ and $X_0\in \T_o^1(\bH^n)$.
Let $K:=\op{Stab}_G(o)$ and $M:=\op{Stab}_G(X_0)$, so that
$\bH^n\simeq G/K$ and $\T^1(\bH^n)\simeq G/M$.
Let $A=\{a_t: t\in \br\}$ be the one-parameter subgroup of diagonalizable elements in the centralizer
of $M$ in $G$ such that  $\mathcal G^t(X_0)=[M]a_t=[a_tM]$.

Using the identification $\T^1(\bH^n)$ with $G/M$, we lift the above measures to $G$, which will be denoted by the same
notation by abuse of notation, so that
they are all invariant under $M$ from the right.

 These measures are all left $\G$-invariant, and hence
induce locally finite Borel measures on $\G \ba G$, which we denote by $m^{\BMS}$ (the BMS-measure), $m^{\BR}$ (the BR-measure), $m^{\BR}_*$
(the BR$_*$ measure), $m^{\Haar}$ (the Haar measure) by abuse of notation.

Let $N^+$ and $N^-$ denote the expanding and the contracting horospherical subgroups, i.e.,
$$N^\pm =\{g\in G: a_t ga_{-t}\to e\text{ as $t\to \pm \infty$}\}.$$
For $g\in G$, define $$g^{\pm}:=(gM)^{\pm}\in \partial(\bH^n).$$

We note that $m^{\BMS}$, $m^{\BR}$, and $m^{\BR}_*$ are invariant under $A$, $N^+$ and $N^-$ respectively
and their supports are given respectively
by $\{g\in \G\ba G: g^+, g^-\in
\Lambda(\G)\}$,
$\{g\in \G\ba G: g^-\in \LG\} $ and
$\{g\in \G\ba G: g^+\in \LG\} $.
The measure  $m^{\Haar}$ is invariant under both $N^+$ and $N^-$, and hence
under $G$, as $N^+$ and $N^-$ generate $G$ topologically.
That is, $m^{\Haar}$ is a Haar measure of $G$.

We consider the action of $G$ on $L^2(\G\ba G, m^{\Haar})$ by right translations, which
gives rise to the unitary action for the inner product:
$$\la \Psi_1, \Psi_2\ra =\int_{\G\ba G} \Psi_1(g)\overline{\Psi_2(g)} dm^{\Haar}(g).$$


\begin{thm}\label{lo} Let $\G$ be Zariski dense.
For any functions $\Psi_1,\Psi_2\in C_c(\G\ba G)$,
$$\lim_{t\to \infty}
e^{(n-1 -\delta)t} \la a_t \Psi_1, \Psi_2\ra = \frac{m^{\BR}(\Psi_1)\cdot m^{\BR}_*(\Psi_2)}{|m^{\BMS}|} .$$
\end{thm}

\begin{proof} Roblin \cite{Roblin2003} proved this for $M$-invariant functions
$\Psi_1$ and $\Psi_2$. His proof is based on the mixing of the geodesic flow
on $\T^1(\G\ba \bH^n)=\G\ba G/M$.
For $\G$ Zariski dense, the mixing of $m^{\BMS}$ was extended to
 the frame flow on $\G\ba G,$ \cite{FS}, see also \cite{Wi} for more detailed analysis. 
Based on this, the proof given in \cite{Roblin2003} can be repeated verbatim to prove the claim.
\end{proof}


\section{Asymptotic expansion of Matrix coefficients}\label{se:ma}
\subsection{Unitary dual of $G$}
Let $G=\SO(n,1)^\circ$ for $n\ge 2$ and $K$  a maximal compact subgroup of $G$. Denoting by $\mathfrak g$ and $\mathfrak k$ the Lie algebras of $G$ and $K$
respectively,
let $\mathfrak g=\mathfrak k \oplus \mathfrak p$ be the corresponding Cartan decomposition of  $\mathfrak g$.
 Let $A=\exp (\mathfrak a)$ where
$\mathfrak a$ is a maximal abelian subspace of $\mathfrak p$
and let $M$ be the centralizer of $A$ in $K$.

Define the symmetric bi-linear form $\la \cdot, \cdot \ra$ on $\mathfrak g$ by
\be \label{dinner} \la X, Y\ra :=\frac{1}{2(n-1)} B(X,Y)\ee
 where $B(X,Y)=\op{Tr}(\text{ad}X \text{ad} Y)$ denotes the Killing form for $\mathfrak g$.
The reason for this normalization is so that the Riemannian metric on $G/K\simeq \bH^n$ induced by
 $\la \cdot, \cdot \ra$ has constant curvature $-1$.

 Let $\{X_i\}$ be a basis for $\mathfrak g_{\c}$ over $\c$;
put $g_{ij}=\la X_i, X_j\ra $ and let $g^{ij}$ be
the $(i,j)$ entry of the inverse matrix of $(g_{ij})$.
 The element
$$ \mathcal C =\sum g^{ij}X_iX_j$$
is called the Casimir element of $\mathfrak g_{\c}$ (with respect to $\la \cdot, \cdot \ra$).
 It is well-known that this definition is independent of the choice of a basis and that
$\mC$ lies in the center of the universal enveloping algebra $U(\mathfrak g_\c)$ of $\mathfrak g_{\c}$.

Denote by $\hat G$ the unitary dual, i.e., the set of equivalence classes of irreducible unitary representations of $G$.
A representation $\pi\in \hat G$ is said to be {\it tempered} if for any $K$-finite vectors $v_1, v_2$
of $\pi$, the matrix coefficient function $g\mapsto \la \pi(g)v_1, v_2\ra$ belongs to $L^{2+\e}(G)$ for any $\e>0$.
We describe the non-tempered part of $\hat G$ in the next subsection.

\subsection{Standard representations and complementary series}\label{ss;standard-comp-series}
Let $\alpha$ denote the simple relative root for $(\mathfrak g, \mathfrak a)$.
The root subspace $\mathfrak n$ of $\alpha$ has dimension $n-1$ and
 hence $\rho$, the half-sum of all positive roots of $(\mathfrak g, \mathfrak a)$ with multiplicities,
is given by $\tfrac{n-1}{2}\alpha$. Set $N=\exp \mathfrak n$.
By the Iwasawa decomposition, every element $g\in G$ can be written uniquely as $g=kan$
with $k\in K$, $a\in A$ and $n\in N$. We write $\kappa(g)=k$, $\exp H(g)=a$ and $n(g)=n$.








For any $g\in G$ and $k\in K$, we let
 $\kappa_g(k)=\kappa(gk),$ and $H_g(k)=H(gk)$ so that
\[
\mbox{$gk=\kappa_g(k)\exp(H_g(k))n(gk).$ }
\]

Given a complex valued linear function $\lam$
on $\mathfrak{a}$, we define a $G$-representation $ U^\lambda$ on $L^2(K)$ by the prescription:
for $\phi\in L^2(K)$ and $g\in G$,
\be\label{eq:rep-l2k}
U^\lam(g)\phi=e^{(-(\lam+2\rho)\circ H_{g^{-1}})} \cdot \phi\circ\kappa_{g^{-1}}.
\ee


This is called a standard representation of $G$ (cf. \cite[Sec. 5.5]{Wa}).
Observe that the restriction of $U^\lambda$ to $K$ coincides with the left regular
representation of $K$ on $L^2(K)$: $U^\lam (k_1) f(k)=f(k_1^{-1}k)$.
If $R$ denotes the right regular representation of $K$ on $L^2(K)$, then
$R(m)U^\lam (g)=U^\lam (g) R(m)$ for all $m\in M$. In particular
each $M$-invariant subspace of $L^2(K)$ for the right translation action
is a $G$-invariant subspace of $U^{\lam}$.

Following~\cite{Wa}, for any $\upsilon\in\hat M$, we let
${\bf Q}(\upsilon)L^2(K)$ denote the isotypic $R( M)$ submodule of
$L^2( K)$ of type $\upsilon.$ Choosing a finite dimensional vector space, say, $V$ on which $M$ acts irreducibly via $\upsilon$,
it is shown in~\cite{Wa}
that the $\upsilon$-isotypic space ${\bf Q}(\upsilon)L^2(K)$ can be written as a sum of $\dim(\upsilon)$ copies of $\mathcal U_\upsilon(\lambda)$
where
\[
\mathcal U_\upsilon(\lambda)=\left\{f\in L^2(K,V): \begin{array}{c} \text{for each $m\in M$, }\\
f(km)=\upsilon(m) f(k),\text{for almost all } k\in K\end{array} \right\}.
\]
If $\lambda\in (\tfrac{n-1}2+i\br)\alpha$, then 
$\mathcal U_\upsilon (\lambda)$ is unitary with respect to the inner product
$\la f_1, f_2 \ra=\int_K \la f_1(k), f_2(k)\ra_V dk$, and called a unitary principal series representation.
These representations are tempered.
A representation 
$\mathcal U_\upsilon (\lambda)$ with $\lambda \notin (\tfrac{n-1}2+i\br)\alpha$
is called a {\it complementary series} representation if it is unitarizable.
For $\lambda=r\alpha$, we will often omit $\alpha$ for simplicity. 
For $n=2$, the complementary series representations of $G=\SO(2,1)^\circ$ are 
$\mathcal U_1( s-1)$ with $1/2<s<1$; in particular they are all spherical.
For $n\ge 3$, a representation $\upsilon\in \hat M$ is
 specified by its highest weight, which can be regarded as a sequence of $\tfrac{n-1}2$ integers with $j_1\ge j_2\ge \cdots
 \ge |j_{(n-1)/2}|$ if $n$ is odd, and as a sequence of $\tfrac{n-2}2$ integers with $j_1\ge j_2\ge \cdots
 \ge j_{(n-2)/2}\ge 0$ if $n$ is even. 
In both cases, let $\ell=\ell(\upsilon)$ be the largest index such that $j_\ell\ne 0$ and put $\ell(\upsilon)=0$ if $\upsilon$ is the trivial representation. Then the complementary series representations are precisely
$\mathcal U_\upsilon(s-n+1)$, $s\in I_\upsilon:=(\tfrac{n-1}2,  (n-1)-\ell)$, up to equivalence.

In particular, 
 the
spherical complementary series representations are exhausted by
$\{ \mathcal U_1( s-n+1): (n-1)/2<s < n-1\}$.

The complementary 
representation $\mathcal U_\upsilon(\lambda)$ contains the minimal $K$-type, say, $\sigma_\upsilon$ with multiplicity
one.

The classification of $\hat G$ says that if $\pi\in \hat G$ is non-trivial
and non-tempered, then $\pi$ is (equivalent to) the unique irreducible subquotient of 
the complementary series representation $\mathcal U_\upsilon(s-n+1)$, $s\in I_\upsilon$,
 containing the $K$-type
$\sigma_\upsilon$, which we will denote by $\mathcal U(\upsilon, s-n+1)$.
This classification was obtained by Hirai \cite{Hi}; see also \cite[Prop. 49 and 50]{KS}) and \cite{BS}. 

 Note that $\mathcal U(\upsilon, s-n+1)$ is spherical
if and only if $\mathcal U_\upsilon(s-n+1)$ is spherical if and only if $\upsilon=1$.
For convenience, we will call $\mathcal U(\upsilon, s-n+1 )$
 a complementary series representation of parameter $(\upsilon, s)$.

Observe that the non-spherical complementary series representations exist only when $n\ge 4$.
	For $\tfrac{n-1}2<s < n-1$, we will set $\mathcal H_s:= \mathcal U({1},{s-n+1}) $, i.e., 
	the spherical complementary series
representations of parameter $s$.
Our normalization of the Casimir element $\mC$ is so that $\mC$ acts on $\mathcal H_s$ as the scalar $s (s-n+1)$.

In order to study the matrix coefficients of complementary series representations, we work
with the standard representations, which we first relate with Eisenstein integrals.
\subsection{Generalized spherical functions and Eisenstein integrals}
Fix a complex valued linear function $\lam$
on $\mathfrak{a},$ and the standard representation $U^\lambda$. By the Peter-Weyl theorem, we may decompose
the left-regular representation $V=L^2(K)$
as
$V=\oplus_{\sigma\in \hat K} V_\sigma,$
where $V_{\sigma}=L^2(K;\sigma)$ denotes the
isotypic $K$-submodule of type $\sigma$,
and $V_\sigma \simeq d_\sigma \cdot \sigma$ where $d_\sigma$ denotes the dimension of $\sigma$.

Set $\Omega_K=1+ \omega_K= 1-\sum X_i^2$
where $\{X_i\}$ is an orthonormal basis of $\mathfrak k_\c$.
 It belongs to the center of the universal enveloping algebra
of $\mathfrak k_{\c}$.
By Schur's lemma, $\Omega_K$ acts on $V_\sigma$ by a scalar, say, $c(\sigma)$.
Since $\Omega_K$ acts as a skew-adjoint operator, $c(\sigma)$ is real.
Moreover $c(\sigma)\ge 1$, see~\cite[p. 261]{Wa}, and
$\|\Omega_K^\ell v\|=c(\sigma)^{\ell} \|v\|$ for any smooth vector $v\in  V_\sigma$.
Furthermore it is shown in~\cite[Lemma 4.4.2.3]{Wa}
that if $\ell $ is large enough, then
\be
\label{cd} \sum_{\sigma\in \hat K} d_\sigma^2 \cdot c(\sigma)^{-\ell}<\infty .
\ee

For $\sigma\in \hat K$ and $k\in K$
 define $$\chi_\sigma(k):=d_\sigma\cdot \op{Tr}(\sigma(k))$$ 
where $\op{Tr}$ is the trace operator.

For any continuous representation $W$ of $K$, and $\sigma\in \hat K$, the projection operator
from $W$ to the $\sigma$-isotypic space $W(\sigma)$ is given as follows:
$$
\P_\sigma= \int_K { \overline{\chi}_\sigma(k)} W(k) dk.
$$

For $\upsilon\in \hat M$,
we write $\upsilon\subset \sigma$ if $\upsilon$ occurs in $\sigma|_M$,
and we write $\upsilon\subset \sigma\cap \tau$ if  $\upsilon$ occurs both in $\sigma|_M$ and $\tau|_M$.
We remark that for $\sigma\in \hat K$, given $\upsilon\in \hat M$ occurs at most once in
$\sigma|_M$ (\cite{Di}; \cite[p.41]{Wa}).
For $\upsilon \subset \sigma$, we denote by
$\P_\upsilon$ the projection operator from $V_\sigma$ to the $\upsilon$-isotypic
subspace $V_\sigma(\upsilon)\simeq d_\sigma \cdot \upsilon$ so that any $w\in V_\sigma$ can be written as 
$w=\sum_{\upsilon\subset \sigma}\P_\upsilon(w)$.
By the theory of representations of compact Lie groups, we have
for any $f\in L^2(K;\sigma)$ 
we have
$$\la f,  \bar \chi_\sigma\ra=f(e).$$


In the rest of this subsection, we fix $\sigma,\tau\in\hat K$. 
Define an $M$-module homomorphism $T_0:V_{\sigma}\rightarrow V_{\tau}$ by
\[ T_0(w )=\sum_{\upsilon\subset \sigma\cap \tau}
\la \P_\upsilon(w),\P_{\upsilon}(\bar{\chi}_\sigma) \ra\P_{\upsilon}(\bar{\chi}_\tau) .\]

Set $E:=\Hom_\bbc(V_\sigma,V_\tau)$. Then
$E$ is a $(\tau,\sigma)$-double representation space, a
left $\tau$ and right $\sigma$-module. We put
$$E_M:=\{T\in E: \tau(m)T= T \sigma(m)  \text{ for all $m\in M$}\} .$$
Denote by $U^\lam_\sigma$ and $U^\lam_\tau$ the representations of $K$
obtained by restricting $U^\lambda|_ K$ to
the subspace $V_\sigma$ and $V_\tau$ respectively.
Define $T_\lam\in E$ by
 $$T_\lam :=\int_M  U_\tau^\lam(m) T_0 U_\sigma^\lam(m^{-1})dm $$
where $dm$ denotes the probability Haar measure of $M$.
It is easy to check that $T_\lam \in E_M$.

An integral of the form
$\int_K U_\tau^\lam (\kappa(a k))T_\lam U_\sigma^\lam(k^{-1})e^{\lam(H(a k))}dk$ is called an {\it Eisenstein integral}.

Clearly, the matrix coefficients of the representation $U^\lambda$
are understood if we understand $P_\tau U^\lambda (a) P_\sigma$ for all $\tau,\sigma\in \hat K$, 
which can be proved to be an Eisenstein integral:

\begin{thm}\label{lem:harish-chandra-integral}
For any $a\in A$, we have
$$
\P_\tau U^\lam (a )\P_\sigma=\int_K U_\tau^\lam (\kappa(a k))T_\lam U_\sigma^\lam(k^{-1})e^{\lam(H(a k))}dk.$$
\end{thm}

\proof
For $\bullet\in \hat K$ and
 $\phi\in L^2(K;\bullet),$  we write
$
U_\bullet^\lam(k^{-1})\phi=
\sum_{\upsilon\subset\bullet}\phi_{k,\upsilon}$, that is,
$\phi_{k,\upsilon}=\P_\upsilon(U_\bullet^\lam(k^{-1})\phi)$.
In particular,
$\phi(k)=\sum_{\upsilon\subset\bullet}\phi_{k,\upsilon}(e)$
and
\[
\phi_{k,\upsilon}(e)= \la\phi_{k,\upsilon},\bar{\chi}_\bullet\ra=\la U_\bullet^\lam(k^{-1})\phi,\P_\upsilon(\bar{\chi}_\bullet) \ra.
\]
Let $\varphi\in V_\sigma$ and
$\psi\in V_\tau$. For any $g\in G$, we have
\begin{align}\label{uc1}
\langle U_\tau^\lam(\kappa(gk)) T_0 U_\sigma^\lam(k^{-1})\varphi,\psi\rangle
&=\sum_{\upsilon \subset \sigma\cap \tau}
\la U_\sigma^\lam(k^{-1})\varphi),\P_{\upsilon}(\bar{\chi}_\sigma)\ra
\la \P_{\upsilon}(\bar{\chi}_\tau),U_\tau^\lam(\kappa(gk))^{-1}\psi\ra \notag\\
&=\sum_{\upsilon \subset \sigma\cap \tau}
\varphi_{k,\upsilon}(e)\overline{\psi_{\kappa(gk),\upsilon}(e)}.
\end{align}
 On the other hand,  we have
\begin{align}\label{uc2}
\langle U^\lam(a )\varphi,\psi\rangle&=
\int_K\varphi(\kappa(a^{-1}k)\overline{\psi(k)}
e^{-(\lam+2\rho)H(a^{-1}k)} dk \notag \\
&=\int_K\varphi(k)\overline{\psi(\kappa(ak))} e^{\lam(H(ak))}dk \notag \\
&=\int_K (\sum_{\upsilon \subset \sigma}\varphi_{k,\upsilon}(e))
\overline{(\sum_{\upsilon\subset \tau}\psi_{\kappa(ak),\upsilon}(e))}e^{\lam(H(ak))}dk.
\end{align}
We now claim that; if $\upsilon_1\ne \upsilon_2,$ then
\[
\int_K \varphi_{k,\upsilon_1}(e)\overline{\psi_{\kappa(ak),\upsilon_2}(e)}
e^{\lam(H(ak))}dk=0
\]
To see this, first note that $M$ and
$a$ commute, and hence
$ H(amk)=H(ak),$ and $\kappa(amk)=m\kappa(ak).$
We also note that
\[
\mbox{$\varphi_{k,\upsilon_1}\in V_\sigma(\upsilon_1),$ and
$\psi_{\kappa(ak),\upsilon_2}\in V_\tau(\upsilon_2).$ }
\]
Now if $\upsilon_1\ne \upsilon_2,$ then by Schur's orthogonality of matrix coefficients,  
\begin{multline*}
\int_M \varphi_{k,\upsilon_1}(m^{-1})\overline{\psi_{\kappa(ak),\upsilon_2}(m^{-1})}dm=\\
\int_M\la U^\lambda_\sigma (m)\varphi_{k,\upsilon_1},\P_{\upsilon_1}(\bar\chi_\sigma)\ra\overline{\la U^\lambda_\tau (m)\psi_{k',\upsilon_2},\P_{\upsilon_2}(\bar\chi_\tau)\ra} dm=0;
\end{multline*}
we get
\begin{align*}
&\int_K \varphi_{k,\upsilon_1}(e)\overline{\psi_{\kappa(ak),\upsilon_2}(e)}
e^{\lam(H(ak))}dk\\
&=\int_{M\backslash K}(\int_M \varphi_{mk,\upsilon_1}(e)\overline{\psi_{\kappa(amk),\upsilon_2}(e)}
e^{\lam(H(amk))}dm)dk\\
&=\int_{M\backslash K}e^{\lam(H(ak))}
(\int_M \varphi_{k,\upsilon_1}(m^{-1})\overline{\psi_{\kappa(ak),\upsilon_2}(m^{-1})}dm)
dk=0,
\end{align*}
implying the claim.

Therefore, it follows from \eqref{uc1} and \eqref{uc2}
that for any $\varphi \in V_\sigma$ and $\psi\in V_\tau$,
\begin{align*}
 &\la \P_\tau U^\lam(a) \P_\sigma \varphi, \psi\ra =\langle U^\lam(a )\varphi,\psi\rangle\\&=\int_K \sum_{\upsilon \subset \sigma\cap \tau}
\varphi_{k,\upsilon}(e)\overline{\psi_{\kappa(gk),\upsilon}(e)} e^{\lam(H(ak))}dk \\&=
\int_K\langle U_\tau^\lam(\kappa(ak)) T_0 U_\sigma^\lam(k^{-1})\varphi,\psi\rangle
e^{\lam(H(ak))}dk \\
 &=\int_K\langle U_\tau^\lam(\kappa(ak)) T_\lam  U_\sigma^\lam(k^{-1})\varphi,\psi\rangle
e^{\lam(H(ak))}dk ;\end{align*}
 we have used $\kappa(a k m)=\kappa(ak)m$ and $H(akm)=H(ak)$ for the last equality.

It follows that
$$ P_\tau U^\lam(a) P_\sigma=\int_K U_\tau^\lam(\kappa(ak)) T_\lam  U_\sigma^\lam(k^{-1})
e^{\lam(H(ak))}dk. $$
\qed

For the special case of
$\tau=\sigma$, this theorem
was proved by Harish-Chandra (see \cite[Thm. 6.2.2.4]{Wa2}), where
$T_0$ was taken to be $T_0(w)=(w,\bar \chi_\sigma)\bar \chi_\sigma$ and $T_\lambda=\int_M  U_\sigma^\lam(m) T_0 U_\sigma^\lam(m^{-1})dm $.

\begin{lem}\label{tbdd} For any $\lam\in \c$,
$$\|T_\lam\|\le d_\sigma^2 \cdot d_\tau^2 $$
where $\|T_\lam\|$ denotes the operator norm of $T_\lam$.
\end{lem}
\begin{proof} Since
 $\|\chi_\upsilon \|\le  d_\upsilon^2$ for any $\upsilon\in \hat M$,
\begin{align*}&\|T_0\|^2 \le
 \sum_{\upsilon\subset \sigma\cap  \tau} \|\chi_\upsilon\|^2 \le 
 \sum_{\upsilon\subset \sigma\cap \tau} d_\upsilon^4
\\&\le (\sum_{\upsilon\subset \sigma} d_\upsilon ^2\cdot
(\sum_{\upsilon\subset \tau} d_\upsilon ^2) \le   d_\sigma^2 \cdot d_\tau^2 .
\end{align*}

Since $\|T_\lam\|\le \|T_0\| \cdot \|\sigma\|\cdot \|\tau\|=\|T_0\|$,
the claim follows.
\end{proof}

\subsection{Harish-Chandra's expansion of Eisenstein integrals}
Fix $\sigma, \tau\in \hat K$.
Let $E$ and $E_M$ to be as in the previous subsection.

Given $T\in E_M$, $r\in\bbc$, and $a_t\in A^+$,  we investigate an Eisenstein integral:
$$\int_K \tau(\kappa(a_tk))T_{ir\alpha -\rho }\sigma(k^{-1})e^{(ir\alpha-\rho) (H(a_tk))}dk .$$

We recall the following fundamental result of Harish-Chandra:

\begin{thm}\label{thm:harish-chandra-asymptotic}\label{precc}
(Cf.~\cite[Theorem 9.1.5.1]{Wa2})
There exists
a subset $\ocal_{\sigma,\tau}$ of $\bbc,$
whose complement is a locally finite set,
such that for any $r\in\ocal_{\sigma,\tau}$ there exist uniquely determined
functions $\cfun_+(r), \cfun_-(r)\in\Hom_\bbc(E_M,E_M)$
such that for all $T\in E_M$,
\begin{multline*}
\rho(a_t )\int_K
\tau(\kappa(a_tk))T\sigma(k^{-1}) e^{(ir\alpha -\rho)H(a_tk)} dk \\
=\Phi(r:a_t)\cfun_{+}( r)T+\Phi(-r:a_t)\cfun_{-}( r)T
\end{multline*}
where $\Phi$ is a function on $\ocal_{\sigma,\tau}\times A^+$ taking
values in $\Hom_\bbc(E_M,E_M)$, defined as in \eqref{e;phi-function}.
\end{thm}
Let us note that, fixing $T$ and $a_t$,
the Eisenstein integral on the left hand side of the above
is an entire function of $r$; see~\cite[Section 9.1.5]{Wa2}.

Much of the difficulties lie in the fact that the above formula holds only
for $\mathcal O_{\sigma,\tau}$ but not for the entire complex plane, as we have no
knowledge of which complementary series representations appear in $L^2(\Gamma\ba G)$.
We need to understand the Eisenstein integral $\int_K
\tau(\kappa(a_tk))T_{s\alpha -2\rho}\sigma(k^{-1}) e^{(s\alpha -2\rho)H(a_tk)} dk$ for every
 $s\in (\tfrac{n-1}{2}, n-1)$. We won't be able to have as precise as a formula as Theorem \ref{precc}
but will be able to determine a main term with an exponential error term.

We begin by discussing the definition and properties of the functions $\Phi$ and $\cfun_\pm$.
\subsubsection{The function $\Phi$.}
As in~\cite[page 287]{Wa2}, we will recursively define
 rational functions
$\{\Gamma_\ell:\ell\in\bbz_{\geq0}\}$ which are
holomorphic except at a locally finite subset, say $\scal_{\sigma,\tau}$.
The subset $\ocal_{\sigma,\tau}$ in Theorem~\ref{thm:harish-chandra-asymptotic}
is indeed
$
\bbc\setminus\cup_{r\in\scal_{\sigma,\tau}}\{\pm r\}.
$

More precisely, let $\mathfrak l$ be the Lie algebra of the Cartan subalgebra
(=the centralizer of $A$). Let $H_\alpha\in \mathfrak l_\c$ be such that
$B(H, H_\alpha)=\alpha(H)$ for all $H\in \mathfrak l_\c$.

Let $X_{\pm \alpha}\in \mathfrak g_\c^{\pm \alpha}$ be chosen
so that $[X_\alpha, X_{-\alpha} ]=H_\alpha$ and $[H, X_\alpha]=\alpha(H) X_\alpha$.
In particular, $B(X_\alpha, X_{-\alpha})=1$.
Write $X_{\pm \alpha}=Y_{\pm \alpha}+Z_{\pm \alpha}$ where $Y_{\pm \alpha}\in \mathfrak k_\c$
and $Z_{\pm \alpha}\in \mathfrak p_\c$.

Letting $\Omega_M$ denote the Casimir element of $M$,
given $S\in\Hom_\bbc(E_M,E_M),$ we define $f(S)$ by
\[
f( S)T=ST\sigma(\Omega_M),\text{ for all } T\in E_M.
\]


We now 
define $\Gamma_\ell:=\Gamma_\ell(ir-\tfrac{n-1}2)$'s in $ \mathcal Q(\mathfrak a_{\c})\otimes \op{Hom}_\c(E_M, E_M)$ by
the following recursive relation (see \cite[p. 286]{Wa2}
for the def. of $\mathcal Q(\mathfrak a_{\c}$)): $\Gamma_0=I$ and  
\begin{multline*}
\left\{ \ell (2ir-n+1)-\ell (\ell -n+1)  -f \right\} \Gamma_\ell =
 \sum_{j\ge 1}( ( 2ir-n+1)  -2(\ell -2j)) \Gamma_{\ell -2j}\\ +
8 \sum_{j\ge 1} (2j-1) \tau(Y_\alpha)\sigma(Y_{-\alpha}) \G_{\ell -(2j-1)}
-8 \sum_{j\ge 1} j \left\{ \tau(Y_\alpha Y_{-\alpha}) +\sigma(Y_\alpha  Y_{-\alpha})\right\} \Gamma_{\ell -2j} .
\end{multline*}

The set $\ocal_{\sigma,\tau}$ consists of $r$'s such that
$\left\{ \ell (2ir-n+1)-\ell (\ell -n+1)  -f \right\}$ is invertible so that
the recursive definition of the $\Gamma_\ell$'s is meaningful.

\begin{lem} \label{gbdd}
Fix any $t_0>0$ and a compact subset $\omega\subset \mathcal O_{\sigma,\tau}$.
There exist $b_\omega$ (depending only on $t_0$ and $\omega$)
and $N_0>1$ (independent of $\sigma,\tau\in \hat K$)
such that
for any $r\in \omega$ and $\ell \in \N$,
$$
\|\Gamma_\ell( ir-\tfrac{n-1}{2})\|\leq b_\omega d_{\sigma}^{N_0}d_{\tau}^{N_0} e^{\ell t_0}.$$
\end{lem}
\begin{proof}
Our proof uses an idea of the proof of \cite[Lemmas 9.1.4.3-4]{Wa2}.
For $s=ir-\tfrac{n-1}{2}$, and $T\in \Hom_\c(E_M,E_M)$,
define
$$\Lambda_\ell (T):=\left(-\ell^2+\ell(2s-n+1)  -f \right) T .$$

For $q_\sigma$ and $q_\tau$ which are respectively the highest weights for $\sigma$ and
$\tau$,  since $q_\sigma \ll d_\sigma$ with implied constant independent of $\sigma\in \hat K$,\begin{align*}&
\max \{ \|\tau(Y_\alpha)\sigma(Y_{-\alpha})\|, \|\tau(Y_\alpha Y_{-\alpha})\|,
\sigma(Y_\alpha  Y_{-\alpha})\| \}\\ & \le c_0 \cdot (q_\sigma q_\tau +q_\sigma^2 +q_\sigma^2)
d_\sigma d_\tau  \le c_0' d_\sigma^3 d_\tau^3
 \end{align*}
for some $c_0, c_0'>0$ independent of $\sigma$ and $\tau$. Hence for some $c_1=c_1(\omega)$, for all $r\in \omega$,
\begin{align}\label{glb} \|\Gamma_\ell(ir-\tfrac{n-1}{2})\| &\le \ell \cdot \|\Lambda_\ell^{-1}\| \cdot
c_1 d_\sigma^3 d_\tau^3  \sum_{j<\ell} \|\Gamma_{\ell-j}\|.
\end{align}

Let $N_1$ be an integer such that $ \ell^2 \cdot \|\Lambda_\ell^{-1}\| \cdot
 (1-e^{-t_0})^{-1} c_1 d_\sigma^3 d_\tau^3 \le N_1$ for all $\ell\ge 1$.
Since $\|\Lambda_\ell^{-1}\|\ll \ell^{-2}$ as $\ell \to\infty$
and the coefficients of $f$ depend only on the eigenvalues of $\Omega_M$
for those $\upsilon\in \hat M$ contained in $ \sigma$,
we can take $N_1=N_1(\omega)$ so that
$N_1 \le c_2 d_\sigma^{N_2}d_\tau^{N_2}$
for some $N_2\ge 1$ and $c_2=c_2(\omega)>1$
(independent of $\sigma$ and $\tau$).

Set
 $$M(t_0,\omega):=\max_{\ell\le N_1, r\in \omega} {\| \Gamma_\ell(ir-\tfrac{n-1}{2})\|}{ e^{-\ell t_0}}.$$

By \eqref{glb} together with
the observation that both $N_1=N_1(\omega)$ and
$\max_{\ell \le N_1}\| \Lambda_\ell^{-1}\|$ are bounded by
a polynomial in $ d_\sigma$ and $d_\tau$,
we  have $M(t_0,\om) \le b_\omega d_\sigma^{N_0} d_\tau^{N_0}$ for some $N_0\ge 1$
and $b_\omega>0$.

We shall now show by induction that for all $r\in \omega$ and for all $\ell \ge 1$,
\be\label{ind2} \|\Gamma_\ell (ir-\tfrac{n-1}{2}) \| \le M(t_0,\om) e^{\ell t_0} .\ee
First note that \eqref{ind2} holds for all $\ell \le N_1$ by the definition of $M(t_0,\om)$.
Now for any $N_1>N$, suppose \eqref{ind2} holds for each $\ell <N$.
Then
\begin{align*} \|\Gamma_{N}(ir-\tfrac{n-1}{2} )\| & \le N^{-1} (N^2 \|\Lambda_N^{-1}\| c_1 d_\sigma^3 d_\tau^3)  \sum_{j<N}
 \|\Gamma_{N-j}(ir-\tfrac{n-1}{2} )\|\\&
 \le N^{-1}N_1 (1-e^{-t_0}) M(t_0,\om) \sum_{j<N} e^{(N-j) t_0}
\\ & \le M(t_0,\om )e^{N t_0}, \end{align*}
finishing the proof.
\end{proof}

Following Warner (Cf.~\cite[Theorem 9.1.4.1]{Wa2}), we define
\be\label{e;phi-function}
\Phi(r:a_t)=e^{ir t}\sum_{\ell\geq0}
\Gamma_\ell(ir-\tfrac{n-1}{2})e^{-\ell t}
\ee
which converges for all large enough $t$ by Lemma \ref{gbdd}.

\medskip

\subsubsection{ The function $\cfun_\pm$.}
 Let $N^-=\exp(\mathfrak n^-)$
be the root subspace corresponding to $-\alpha$, and
 $d_{N^-}$ denote a Haar measure on $N^-$ normalized so that
$\int_{N^-}e^{-2\rho(H(n))}d_{N^-}(n)=1.$


The following is due to Harish-Chandra; see \cite[Thm. 9.1.6.1]{Wa2}.
\begin{thm}\label{thm:asymp-c-function}
For $r\in \mathcal O_{\sigma, \tau}$ with  $\Im ( r)<0$, $\cfun_+(r)$ is holomorphic and given by
 \[
\cfun_+( r)T=\int_{N^-}T\sigma(\kappa(n)^{-1})e^{-(ir\alpha +\rho)(H(n))} d_{N^-}(n).
\]
\end{thm}


 The integral $\int_{N^-} e^{-(ir\alpha +\rho) H(n)} d_{N^-}(n)$ is absolutely convergent iff
 $\Im ( r)<0$, shown 
by Gindikin and Karperlevic (\cite[Coro.9.1.6.5]{Wa2}).
\begin{cor}\label{cp}
For any $r\in \mathcal O_{\sigma,\tau}$  with  $\Im ( r)<0$,
the operator norm $\|\cfun_{+}(r )\|$ is bounded above by $  \int_{N^-} e^{(\Im(r)\alpha -\rho) H(n)} d_{N^-}(n).$
 \end{cor}

\begin{proof}
Since the operator norm $\|\sigma(k)\|$ is $1$ for any $k\in K$, the claim is immediate from Theorem \ref{thm:asymp-c-function}.
\end{proof}







\begin{prop} \label{cpm2}
Fix a compact subset $\omega$ contained in $\mathcal O_{\sigma, \tau}\cap \{ \Im ( r)<0\} $.
There exist $d_1=d_1(\omega)$ and $N_2>1$ such that
for any $r\in \omega$,
we have
$$\|\cfun_{\pm }(r ) T_{\pm ir\alpha -\rho} \| \le d_1 \cdot d_\tau^{N_2}d_\sigma^{N_2} .$$
\end{prop}
\begin{proof}
By the assumption on $\omega$, the integral $\int_{N^-} e^{-(\Re(ir)\alpha+\rho) H(n)} d_{N^-}(n)$ converges
uniformly for all $r\in \omega$. Hence
the bound for $\cfun_+(r)T_{ir\alpha-\rho}$ follows from Corollary \ref{cp} together with Lemma \ref{tbdd}.
To get a bound for $\cfun_-(r)T_{-ir\alpha-\rho}$, we recall that
\begin{multline*} e^{(-ir+\tfrac{n-1}{2})t} \int_K
\tau(\kappa(a_tk))T\sigma(k^{-1}) e^{(ir\alpha -\rho)(H(a_tk))} dk \\
=e^{-irt} \Phi(r:a_t)\cfun_{+}( r)T+e^{-irt} \Phi(-r:a_t)\cfun_{-}( r)T .\end{multline*}

Then
$ e^{-irt} \Phi(-r:a_t)= I + \sum_{\ell\ge 1} \Gamma_\ell (-ir-\tfrac{n-1}{2})e^{-\ell t} $,
and applying Lemma~\ref{gbdd} with $t_0=1$, we get
$$\sum_{\ell\ge 1} \|\Gamma_\ell (-ir-\tfrac{n-1}{2})e^{-\ell t}\|\le b_\omega
 d_\sigma^{N_0} d_\tau^{N_0} \sum_{\ell \ge 1} e^{\ell (1-t)}.$$
Fix $t_0>0$ so that  $b_\omega
 d_\sigma^{N_0} d_\tau^{N_0} \sum_{\ell \ge 1} e^{\ell (1-t_0)}<1/2;$
then $t_0\gg \log (d_\sigma d_\tau).$
Now  $A_r:=e^{-irt_0} \Phi(-r:a_{t_0})$ is invertible and  for some $N_1$ and $b_\omega'$,
\be \label{ar} \|A_r^{-1}\| \le b'_\omega
 d_\sigma^{N_1} d_\tau^{N_1} .\ee

Since the map $k\mapsto H(a_{t_0}k)$ is continuous,
we have
 $\int_K | e^{(ir \alpha -\rho)(H(a_{t_0}k))}| dk<d_\omega  $ for all $r\in \omega$, and hence
\begin{align*}
&\| \cfun_{-}( r)T_{ir\alpha-\rho}\| \\& \le \|A_r^{-1}\| \cdot |e^{(-ir+\tfrac{n-1}{2})t_0}\int_K
\tau(\kappa(a_{t_0}k))T_{ir\alpha-\rho}\sigma(k^{-1}) e^{(ir\alpha -\rho)(H(a_{t_0}k))} dk| \\ &   +\|A_r^{-1}\| \cdot \|e^{-irt_0} \Phi(r:a_{t_0})\cfun_{+}( r)T_{ir\alpha -\rho}\|
\\ & \le \|A_r^{-1}\| \cdot d_\omega \left(\max_{k\in K}\| \tau(\kappa(a_{t_0}k))T_{ir\alpha-\rho}\sigma(k^{-1})\|  +
\|\cfun_{+}( r)T_{ir\alpha-\rho}\|\right).
 \end{align*}
Hence the claim on $\| \cfun_{-}( r)T_{ir\alpha-\rho}\|$ now follows from \eqref{ar},
Lemma \ref{tbdd} and the bound for $\|\cfun_+(r)T_{ir\alpha-\rho}\|$.
\end{proof}

\subsection{Asymptotic expansion of the matrix coefficients of the complementary series}
Fix a parameter $(n-1)/2<s_0<(n-1)$, and
recall that $2\rho =(n-1)\alpha.$
We apply the results of the previous subsections to  the standard representation $U^{(s_0-n+1)\alpha}=U^{s_0\alpha-2\rho}$.

The following theorem is a key ingredient of the proof of Theorem \ref{mc2}.
\begin{thm}\label{keyyy} There exist $\eta_0>0$ and $N>1$ such that for any $\sigma,\tau\in \hat K$, for all $t>2$, we have
$$P_{\tau}U^{(s_0-n+1)\alpha}(a_t) P_{\sigma} =   e^{(s_0 -n+1)t} \cfun_+(r_{s_0}) T_{(s_0-n+1)\alpha} + O(
d_{\sigma}^{N}\cdot d_{\tau}^{N}  e^{(s_0-n+1-\eta_0)t}),$$
with the implied constant independent of $\sigma, \tau$.
\end{thm}

\begin{proof}
Set $r_s:=-i(s-\rho)\in \c,$ for all $s\in\bbc$.
In particular, $\Im(r_s)<0$ for $(n-1)/2<s<(n-1)$.

Fix $t>0$, and define $F_t: \bbc \to E_M$ by
$$F_t(s ):= P_{\tau}U^{s\alpha-2\rho}(a_t) P_{\sigma}.$$
By Theorem \ref{lem:harish-chandra-integral},
$$F_t(s)=\int_K
\tau(\kappa(a_tk))T_{s\alpha -2\rho}\sigma(k^{-1}) e^{(s\alpha -2\rho)H(a_tk)} dk.$$
As was remarked following Theorem~\ref{thm:harish-chandra-asymptotic},
for each fixed $t>0$, the function $F_t( s)$ is analytic on $\bbc$.
 Thus in view of
Theorem~\ref{thm:asymp-c-function}, we have, whenever $r_s\in \ocal_{\sigma,\tau}$ and $\Im(r_s)<0,$
\be\label{e;remainder-analytic}
\mbox{$F_t(s )-e^{(s -n+1)t} \cfun_+(r_{s}) T_{s\alpha-2\rho}$
is analytic}.
\ee

Recall the notation $\scal_{\sigma,\tau}$, that is, $\ocal_{\sigma,\tau}=\bbc\setminus\cup_{r\in\scal_{\sigma,\tau}}\{\pm r\},$
 and set $\tilde \scal_{\sigma,\tau}=\{s:r_s\in\scal_{\sigma,\tau}\}.$
Define $$G_t(s )=F_t(s )-e^{(s -n+1)t} \cfun_+(r_{s}) T_{s\alpha-2\rho}.$$
Indeed the map $s\mapsto G_t(s)$ is analytic on $\{s: \Im(r_s)<0\}-\tilde \scal_{\sigma,\tau}.$
Since
$\cup_{\sigma', \tau'\in \hat K} \pm  \tilde \scal_{\sigma',\tau'}$ is countable,
we may choose a small circle $\omega'$ of radius at most $1/2$
centered at $s_0$ such that
$\{r_s: s\in\omega'\} \cap \left(\cup_{\sigma', \tau'\in \hat K} \pm \mathcal
\scal_{\sigma',\tau'}\right) = \emptyset.$

Observe that the intersection of $ \omega'$
and the real axis is contained in the interval $((n-1)/2, n-1)$.
Note that there exists $\eta>0$ such that for all $s\in \omega'$,
\be\label{fe}(n-1)-s_0+\eta <\Re(s)<s_0+1-\eta .\ee

Then $G_t( s)$
is analytic on the disc bounded by $\omega'.$
Hence by the maximum principle, we get
\be\label{e;error-est1}
\|G_t(s_0)\|\leq\max_{s\in\omega'}\|G_t( s)\|.
\ee
Since $\omega:=\{r_s: s\in \omega'\} \subset\ocal_{\sigma,\tau},$
Theorems \ref{lem:harish-chandra-integral} and ~\ref{thm:harish-chandra-asymptotic} imply that
for all $s\in\omega',$ we have
\begin{multline*}
G_t(s )=e^{(s-n+1)t} (\sum_{\ell\geq 1}
e^{-\ell t} \Gamma_\ell (ir_s-\tfrac{n-1}{2})  \cfun_+(r_s)
T_{s\alpha -2\rho}) \\
 +e^{-st} (\sum_{\ell\geq0 } e^{-\ell t} \Gamma_\ell (-ir_s-\tfrac{n-1}{2}) \cfun_-(r_s) T_{s\alpha -2\rho}).
\end{multline*}

Fixing any $t_0>0$,
by Lemma \ref{gbdd},  there exists $b_0=b_0(t_0,\omega)>0$ such that for all
$r\in\omega$,
\be\label{e;asymp-phi}
\|\Gamma_\ell( ir-\tfrac{n-1}{2})\|\leq b_0 \cdot d_{\sigma}^{N_0} \cdot d_{\tau}^{N_0} \cdot   e^{\ell t_0}.
\ee

By Proposition \ref{cpm2},  for all $r\in \omega$,
$$ \|\cfun_{\pm}(r )T_{\pm ir\alpha -\rho} \|\leq d_1 \cdot d_{\sigma}^{N_2} \cdot d_{\tau}^{N_2}.$$


Let $t>t_0+1$ so that
$\sum_{\ell\geq 0}e^{-\ell (t-t_0)}\le 2$. Then
we have for any $t>0$ and $s\in \omega'$,
\begin{align*}
&\|\sum_{\ell\geq 1} e^{-\ell t} \Gamma_\ell (ir_s-\tfrac{n-1}{2}) \cfun_+ (r_s)
T_{s\alpha-2\rho}\| \\ & \leq  \cdot d_{\sigma}^{N_0+N_2} d_{\tau}^{N_0+N_2}\cdot  b_{\sigma,\tau} \cdot
 e^{-t}e^{t_0} \sum_{\ell\geq 0}e^{-\ell (t-t_0)}
\\ &\le \left( 2e^{t_0} \cdot b_0\cdot d_{\sigma}^{N_0+N_2}\cdot d_{\tau}^{N_0+N_2} \right) e^{-t}
\end{align*}
and

\be\label{e;error-est3}
\|\sum_{\ell\geq 0} e^{-\ell t} \Gamma_\ell (- ir_s-\tfrac{n-1}{2}) \cfun_- (r_s)
T_{s\alpha -2\rho}\|\leq 2  b_0 \cdot  d_{\sigma}^{N_0+N_2}\cdot  d_{\tau}^{N_0+N_2} .
\ee
We now combine these
and the expression for $G_t( s),$ for $s\in\omega'$ and get
for all $t>t_0+1$, \begin{align*}
& \|G_t(s_0)\| \\ & \leq  2 b_0(e^{t_0} +1)d_{\sigma}^{N_0+N_2}\cdot d_{\tau}^{N_0+N_2} \cdot  \max_{s\in\omega'}
 ( e^{(\Re (s)-n)t}+
e^{-\Re(s) t})
\\ &\le  b'\cdot  d_{\sigma}^{N_0+N_2}\cdot d_{\tau}^{N_0+N_2}  e^{(s_0-(n-1)-\eta)t}
\end{align*}
where $\eta>0$ is as in \eqref{fe} and $b'>0$ is a constant independent of $\sigma,\tau\in \hat K$.

Since
$P_{\tau}U^{(s_0-n+1)\alpha}(a_t) P_{\sigma} =e^{(s_0 -n+1)t} \cfun_+(r_{s_0}) T_{(s_0-n+1)\alpha}+G_t(s_0)$, this finishes the proof.
\end{proof}

By Theorem \ref{lem:harish-chandra-integral}, Theorem \ref{keyyy} implies:
\begin{thm}\label{hmc} \label{key}Let $(n-1)/2<s_0<(n-1)$. There exist $\eta_0>0$ and $N>1$ such that
for all $t\ge 2$ and for any unit vectors $v_\sigma\in V_\sigma$ and $v_\tau\in V_\tau$,
\begin{multline*}
\la U^{(s_0-n+1)\alpha}(a_t) (v_\sigma), v_\tau\ra \\ =
e^{(s_0 -n+1)t} \la \cfun_+(r_{s_0}) T_{(s_0-n+1)\alpha} (v_\sigma), v_\tau\ra
+ O( d_{\sigma}^N d_{\tau}^N e^{(s_0 -n+1 -\eta_0)t}),
\end{multline*}
with the implied constant independent of $\sigma, \tau,v_\sigma,v_\tau$.
\end{thm}

\subsection{Decay of matrix coefficients for $L^2(\G\ba G)$}\label{asde}
Let $\G<G$ be a torsion-free
geometrically finite group with $\delta>(n-1)/2$.

By Lax-Phillips \cite{LaxPhillips} and Sullivan \cite{Sullivan1979},
 $\mathcal U(1,\delta -n+1)$
occurs as a subrepresentation of $L^2(\G\ba G)$ with multiplicity one, and
$L^2(\G\ba G)$ does not weakly contain
any spherical complementary series $\mathcal U(1, s-n+1)$ of parameter $s$ strictly bigger than $\delta$.
In particular, $\delta$ is the maximum $s$ such that
 $\mathcal U(1, s-n+1)$ is weakly contained in $L^2(\G\ba G)$.
 
The following proposition then follows from \cite[Prop. 3.5]{Sh} and Theorem \ref{hmc}:
\begin{prop}
$L^2(\G\ba G)$ does not weakly contain any complementary
series $\mathcal U(\upsilon, s-n+1)$ with $\upsilon\in \hat M$ and $s> \delta$.
\end{prop}

\begin{dfn}[Spectral Gap] \label{snggg}
 We say $L^2(\G\ba G)$ has a spectral gap if
the following two conditions hold:
\begin{enumerate}
 \item there exists $n_0\ge 1$ such that the multiplicity of any complementary series $\mathcal U(\upsilon, \delta -n+1)$ of parameter $\delta$ occurring in $L^2(\G\ba G)$ is at most
$\text{dim}(\upsilon)^{n_0}$ for all $\upsilon \in \hat M$;
\item there exists $(n-1)/2< s_0<\delta$ such that
no complementary series with parameter $s_0< s<\delta$ is weakly contained in $L^2(\G\ba G)$.
\end{enumerate}

We set  $n_0(\G)$ and $s_0(\G)$ to be the infima of all $n_0$ and of all $s_0$
satisfying (1) and (2) respectively.
The pair $(n_0(\G), s_0(\G))$ will be referred as the spectral gap data for $\G$.
\end{dfn}
In other words, the spectral gap property of $L^2(\G\ba G)$ is equivalent to
the following decomposition:
 \be \label{stg} L^2(\G\ba G)=\mathcal H_\delta^\dag \oplus \mathcal W \ee
where $ \mathcal H_{\delta}^\dag= \oplus_{\upsilon\in \hat M} m(\upsilon) \mathcal U(\upsilon, \delta-n+1)$ with
$m(\upsilon)\le \text{dim}(\upsilon)^{n_0}$, and
no complementary series representation with parameter $s_0(\Gamma)<s<\delta $ is weakly contained in $\mathcal W$.


We recall the strong spectral gap property from Def. \ref{sng}.
\begin{thm}\label{nm2}
 Suppose that $\delta>(n-1)/2$ for $n=2,3$ or that $\delta >n-2$ for $n\ge 4$. Then $L^2(\G\ba G)$ has a strong
 spectral gap property. 
\end{thm}
\begin{proof}
By the the classification of the unitary dual $\hat G$ explained in the subsection \ref{ss;standard-comp-series},
any non-spherical complementary series representation
is of the form $\mathcal U(\upsilon, s-n+1)$ for some $\upsilon\in \hat M-\{1\}$ and $s\in  (\tfrac{n-1}2, n-2)$
(see \cite{Hi} and \cite{KS}). Together with the aforementioned work of Lax-Phillips on the spherical complementary series
representations occurring in $L^2(\G\ba G)$, this implies the claim.
\end{proof}

For $\Psi\in C^\infty(\G\ba G)$, $\ell\in \N$ and $1\le p\le \infty$,
we consider the following Sobolev norm:
\be \label{dso} \mathcal S_{p, \ell}(\Psi)=\sum \| X (\Psi)\|_{p}\ee
where the sum is taken over all monomials $X$ in a fixed basis $\mathcal B$ of $\mathfrak g$ of order at most $\ell$
and $\|X(\Psi)\|_p$ denotes the $L^p(\G\ba G)$-norm of $X(\Psi)$.
Since we will be using $\S_{2,\ell}$ most often, we will set $\S_\ell=\S_{2,\ell}$.

For a unitary $G$-representation space $W$ and a smooth vector $w\in W$,
$\S_\ell(w)$ is defined similarly: $\S_\ell(w)=\sum \| X .w\|_{2}$
where the sum is taken over all monomials $X$ in $\mathcal B$ of order at most $\ell$.

\begin{prop}\label{shh}
Fix $(n-1)/2< s_0 <(n-1)$. Let $W$ be a unitary representation of $G$
which does not weakly contain any complementary series representation $\mathcal U(\upsilon, s-n+1)$
with parameter $s> s_0$ and $\upsilon\in \hat M$.
Then for any $\e>0$, there exists $c_\e>0$ such that
 for any smooth vectors $w_1, w_2\in W$ and for any $t>0$,
we have $$|\la a_t w_1, w_2\ra|\le c_\e \cdot \S_{\ell_0}(w_1) \cdot \S_{\ell_0}(w_2)\cdot  e^{(s_0-n+1+\e)t}$$
where $\ell_0\ge 1$ depends only on $G$ and $K$.
\end{prop}
\begin{proof} This proposition is proved in \cite[Proof of Prop. 5.3]{KO} for $n=3$ (based on an earlier idea of
\cite{Sh}), and its proof easily extends to a general $n\ge 2$.
\end{proof}

 In the following two theorems, we assume that
  $\G$ is Zariski dense in $G$ and that
 $L^2(\G\ba G)$ has a spectral gap with the spectral gap data $(s_0(\G), n_0(\G))$.
\begin{thm} \label{mc2} 
There exist $\eta>0$ (depending only on $s_0(\G)$), and  $\ell\in \N$ (depending only on $n_0(\G)$) such that
for any real-valued functions $\Psi_1,\Psi_2\in C_c^\infty(\G\ba G)$ as $t\to +\infty$,
$$ e^{(n-1 -\delta)t} \la a_t \Psi_1, \Psi_2\ra = \frac{m^{\BR}(\Psi_1)\cdot m^{\BR}_*(\Psi_2)}{|m^{\BMS}|}
+O( e^{-\eta t}\S_\ell(\Psi_1)\S_\ell(\Psi_2)). $$
\end{thm}

\begin{proof} As in \eqref{stg},
we write
$$L^2(\G\ba G)=\mathcal H_{\delta}^\dag \oplus \mathcal W$$ where
 $ \mathcal H_{\delta}^\dag= \oplus_{\upsilon\in \hat M} m(\upsilon) \mathcal U(\upsilon, (\delta-n+1)\alpha)$ with
$m(\upsilon)\le  \text{dim}(\upsilon)^{n_0(\G)}$, and
no complementary series representation with parameter $s_0(\G) <s$ is weakly contained in $\mathcal W$.
For simplicity, set $s_0:=s_0(\G)$ and $n_0:=n_0(\G)$.
we set $V=\mathcal H_{\delta}^\dag $ and $V^\perp=\mathcal W$.
Given $\Psi_1,\Psi_2\in C_c^\infty(\G\ba G)$,
we write $\Psi_i=\Psi_i' +\Psi_i^\perp$, where
$\Psi_i'$ and $\Psi_i^\perp$ are the projections of $\Psi_i$ to $\mathcal H_{\delta}^\dag$ and $\mathcal W$ respectively.
Then by Proposition \ref{shh}, there exist $\ell_0\ge 1$ such that for any $\e>0$,
\be\label{fff}\la a_t \Psi_1^\perp,\Psi_2^\perp\ra =O(\S_{\ell_0}(\Psi_1)\S_{\ell_0}(\Psi_2) e^{(s_0 -n+1+\e)t}).\ee

If $\delta=n-1$ and hence  if $\mathcal H_{\delta}^\dag=\c$, it is easy to see that \eqref{fff} finishes the proof.
Now suppose $\delta<n-1$.

For each $\upsilon \in \hat M$,
 the $K$-representation $\mathcal U_\upsilon( \delta-n+1 )|_K$ is isomorphic to the induced representation 
$\op{Ind}_M^K (\upsilon)$
and hence by the Frobenius reciprocity, the multiplicity of $\sigma$ in 
$\mathcal U_\upsilon( s-n+1)|_K$ is equal
to the multiplicity of $\upsilon$ in $\sigma|_M$, which is denoted by $[\sigma: \upsilon]$.
Therefore as a $K$-module,
$$\mathcal U(\upsilon, \delta-n+1 )|_K=\oplus_{\sigma\in \hat K} m_\upsilon(\sigma) \sigma $$
where $m_\upsilon(\sigma)\le [\sigma:\upsilon]$.

As a $K$-module, we write
\begin{align*}\mathcal H_{\delta}^\dag &=\oplus_{\upsilon\in \hat M} m(\upsilon) \mathcal U(\upsilon, (\delta-n+1)\alpha)
\\ &=\oplus_{\upsilon\in \hat M} m(\upsilon) \left(\oplus_{\sigma\in \hat K} m_\upsilon(\sigma) \sigma\right)
\\ &=\oplus_{\sigma\in \hat K}m(\sigma) \sigma
\end{align*}
where $m(\sigma)\le \sum_{\upsilon\in \hat M, \upsilon \subset \sigma} m(\upsilon) [\sigma:\upsilon] $.
Note that $m(\sigma)\le  d_{\sigma}^{n_0+1} $ for $n_0=n_0(\G)$.

For each $\sigma\in \hat K$,
let $\Theta_\sigma$ be an orthonormal
basis in the $K$-isotypic component, say, $V_\sigma$, of $\mathcal H_\delta^\dag$,
which is formed by taking the union of orthonormal bases of each irreducible component of $V_\sigma$.
Then $\# \Theta_\sigma \le   d_{\sigma}^{n_0+2} $.

By Theorem \ref{hmc} and our discussion in
section \ref{ss;standard-comp-series}, there exist $\eta_0>0$ and
 $N\in \N$ such that for any $v_\sigma \in \Theta_\sigma$ and $v_\tau\in \Theta_\tau$, we have for all $t\gg 1$,
\be \label{mae}
\la a_t v_\sigma, v_\tau\ra: =c(v_\sigma, v_\tau)
e^{(\delta -n+1)t}
+ O( d_{\sigma}^N d_{\tau}^N e^{(\delta -n+1-\eta_0)t})
\ee
where $c(v_\sigma, v_\tau)= \la \cfun_+(r_{\delta}) T_{(\delta-n+1)\alpha} v_\sigma, v_\tau\ra $.

As $\Psi_i'=\sum_{\sigma\in \hat K}\sum_{v_\sigma\in\Theta_\sigma}\la \Psi_i, v_\sigma\ra v_\sigma$,
we have for each $t\in \br$,
$$\la a_t \Psi_1', \Psi_2'\ra =\sum_{\sigma, \tau\in \hat K}\sum_{v_\sigma\in \Theta_\sigma, v_\tau\in \Theta_\tau}
 \la \Psi_1, v_\sigma\ra \cdot \overline{\la \Psi_2, v_\tau\ra} \cdot \la a_t v_\sigma, v_\tau\ra $$
where the convergence follows from the Cauchy-Schwartz inequality.

Therefore, by \eqref{mae},
\begin{align*}
&\la a_t \Psi_1', \Psi_2'\ra \\
&=
\left(\sum_{\sigma, \tau\in \hat K}\sum_{v_\sigma\in \Theta_\sigma, v_\tau\in \Theta_\tau}
\la \Psi_1, v_\sigma \ra  \overline{\la \Psi_2, v_\tau\ra} c(v_\sigma,v_\tau) \right) e^{(\delta-n+1)t}
\\ & +\sum_{\sigma, \tau\in \hat K}\sum_{v_\sigma\in \Theta_\sigma, v_\tau\in \Theta_\tau}
 d_{\sigma}^{N} d_{\tau}^{N}\la \Psi_1, v_\sigma \ra  \overline{\la \Psi_2, v_\tau\ra} O(e^{(\delta-n+1-\eta_0) t})).
\end{align*}

Set $$\mathcal E(\Psi_1, \Psi_2):=
\left(\sum_{\sigma, \tau\in \hat K}\sum_{v_\sigma\in \Theta_\sigma, v_\tau\in \Theta_\tau}
\la \Psi_1, v_\sigma \ra  \overline{\la \Psi_2, v_\tau\ra} c(v_\sigma,v_\tau) \right) .$$

By \eqref{cd}, there exists $\ell \ge \ell_0$ (depending only on $n_0$) such that
\be \label{ell}
\sum_{\sigma, \tau\in \hat K}d_{\sigma}^{N+n_0+2} d_{\tau}^{N+n_0 +2}c(\sigma)^{-\ell} c(\tau)^{-\ell}
 <\infty
\ee
where $c(\sigma)$ is as in~\eqref{cd}.
Since for any unit vector $v\in V_\sigma$,
$$ |{\la \Psi, v\ra}|\ll  c(\sigma)^{-\ell}\mathcal \S_\ell(\Psi),$$
we now deduce that
$$\la a_t \Psi_1', \Psi_2'\ra =\mathcal E(\Psi_1, \Psi_2)
e^{(\delta-n+1)t} +  O(e^{(\delta-n+1-\eta_0) t} \mathcal \S_\ell(\Psi_1)\mathcal \S_\ell(\Psi_2)).$$

Hence, together with  \eqref{fff}, it implies that there exists $\eta>0$ such that
$$ \la a_t \Psi_1, \Psi_2\ra
=\mathcal E(\Psi_1, \Psi_2)
e^{(\delta-n+1)t} + O(e^{(\delta-n+1-\eta) t} \mathcal \S_\ell(\Psi_1)\mathcal \S_\ell(\Psi_2)).$$

On the other hand, by Theorem~\ref{lo},
$$\lim_{t\to \infty} e^{(n-1-\delta)t} \la a_t \Psi_1, \Psi_2\ra =
\frac{m^{\BR}(\Psi_1)\cdot m^{\BR}_*(\Psi_2)}{|m^{\BMS}|}.$$

It follows that
the infinite sum $\mathcal E(\Psi_1, \Psi_2)$ converges
and  $$\mathcal E(\Psi_1, \Psi_2)=\frac{m^{\BR}(\Psi_1)\cdot m^{\BR}_*(\Psi_2)}{|m^{\BMS}|}.$$
This finishes the proof.
\end{proof}

As $\la a_{-t} \Psi_1, \Psi_2\ra =\la a_t \Psi_2, \Psi_1\ra$ for $\Psi_i$'s real-valued, we
deduce the following from Theorem \ref{mc2}:
\begin{cor}\label{reverse} There exist $\eta>0$ and $\ell \in \N$ such that, as $t\to +\infty$,
$$ e^{(n-1 -\delta)t} \la a_{-t} \Psi_1, \Psi_2\ra = \frac{m^{\BR}_*(\Psi_1)\cdot m^{\BR}(\Psi_2)}{|m^{\BMS}|}
+O( e^{-\eta t}\S_\ell(\Psi_1)\S_\ell(\Psi_2)). $$
\end{cor}

\section{Non-wandering component of $\G\ba \G Ha_t$ as $t\to \infty$}\label{nw}
\subsection{Basic setup}\label{basicset} Let $H$ be either a symmetric subgroup or a horospherical subgroup of $G$.
For the rest of the paper,
we will set $K, M, A=\{a_t\}$ in each case as follows.
 If $H$ is symmetric, that is, $H$ is equal to the group of $\sigma$-fixed points
for some involution $\sigma$ of $G$,
up to conjugation and commensurability, $H$ is $\SO(k,1)\times \SO(n-k)$ for some $1\le k\le n-1$.
Let $\theta$ be a Cartan involution of $G$ which commutes with $\sigma$
and set $K$ to be the maximal compact subgroup fixed by $\theta$.
Let $G=K\exp \frak p$ be the Cartan decomposition and write $\mathfrak g$
as a direct sum of $d\sigma$ eigenspaces:
 $\mathfrak g=\mathfrak h\oplus \mathfrak q$ where $\mathfrak h$ is the Lie algebra of $H$
and $\mathfrak q$ is the $-1$ eigenspace for $d\sigma$.
Let $\frak a\subset \frak p\cap \frak q$ be a maximal abelian subspace and set
$A=\exp \frak a=\{a_t:=\exp (tY_0): t\in \br\}$ where $Y_0$
is a norm one element in $\frak a$ with respect to the Riemannian metric
induced by $\la, \ra$ defined in \eqref{dinner}. Let $M$ be the centralizer of $A$ in $K$.

 If $H$ is a horospherical subgroup of $G$, we let $A=\{a_t\}$ be a one-parameter subgroup
of diagonalizable elements so that $H$ is the expanding horospherical subgroup for $a_t$. Letting $M$ be
the maximal
compact subgroup in the centralizer of $A$, we may assume that
the right translation action of $a_t$ corresponds to the geodesic flow
on $\T^1(\bH^n)=G/M$.
Let $K$ be the stabilizer of the base point of the vector in $\T^1(\bH^n)$ corresponding to $M$.

In both cases, let $o\in \bH^n$ and $X_0\in \T^1(\bH^n)$ be
points stabilized by $K$ and $M$ respectively.
Let $N^+$ and $N^-$ be the expanding and contracting
horospherical subgroups of $G$ with respect to $a_t$, respectively.

\medskip

\subsection{Measures on $gH$ constructed from conformal densities}\label{psdef}
Set $P:=MAN^-$, which is the stabilizer of $X_0^+$.
Via the visual map $g\mapsto g^+$, we have $G/P\simeq \partial(\bH^n)$.
Since $G/P\simeq K/M$, we may consider the visual map as a map from $G$ to $K/M$.
In both cases, the restriction of the visual map to $H$ induces a diffeomorphism from $H/H\cap M$ to its image
inside $K/M$.



Letting $\{\mu_x: x\in G/K\}$ be a $\G$-invariant conformal density of dimension $\delta_\mu$,
 we will define
an $H\cap M$-invariant measure $\tilde \mu_{gH}$ on each $g\in G/H$. Setting $\bar H=H/(H\cap M)$,
 first consider the measure on $g\bar H$:
$$d\tilde \mu_{g\bar H}(gh)= e^{\delta_\mu \beta_{(gh)^+}(o,gh)} d\mu_o((gh)^+).$$
We denote by $\tilde\mu_{gH}$ the $H\cap M$-invariant extension of this measure on $gH$, that is,
for $f\in C_c(gH)$,
$$\int f(gh)\; d\tilde\mu_{gH}(gh)= \int_{g\bar H} \int_{H\cap M}f(ghm)d_{H\cap M}(m) d\tilde\mu_{g\bar H}(gh)$$
where $d_{H\cap M}(m)$ is the probability Haar measure on $H\cap M$.

By the $\G$-invariant conformality of $\{\mu_x\}$,
this definition is independent of $o\in \bH^n$ and
$\tilde \mu_{gH}$ is invariant under $\G$ and hence if $\G\ba \G gH$ is closed,
$\tilde \mu_{gH}$ induces a locally finite Borel measure $\mu_{gH}$ on $\G\ba \G g H$.

Recall the Lebesgue density $\{m_x\}$ of dimension $n-1$ and the Patterson-Sullivan density
$\{\nu_x\}$ of dimension $\delta$.
 We normalize them so that $|m_o|=|\nu_o|=1$. We set
 $$\tilde \mu^{\Haar}_{gH}=\tilde m_{gH}\quad\text{and}\quad \tilde \mu^{\PS}_{gH}=\tilde \nu_{gH} ,$$
and for a closed orbit $\G\ba \G gH$, we denote by
$\mu_{gH}^{\Haar}$ and $\mu_{gH}^{\PS}$ the measures on $\G\ba \G gH$ induced by them respectively.



\begin{lem}\label{inv}
For each $g\in G$,
 $d\tilde\mu^{\Haar}_{gH}(gh)=d\tilde\mu^{\Haar}_{H}(h)$ and
$dh:=d\tilde\mu^{\Haar}_{H}(h)$ is a Haar measure on $H$.
\end{lem}
\begin{proof}
As $m_o$ is $G$-invariant, we have
$$dm_o((gh)^+)=dm_{g^{-1}(o)}(h^+)= e^{(n-1)\beta_{h^+} (o, g^{-1}(o))}dm_o(h^+).$$
Since
$$\beta_{h^+} (o, g^{-1}(o))+ \beta_{(gh)^+}(o,gh)=\beta_{h^+} (o, g^{-1}(o))+ \beta_{h^+}(g^{-1}(o),h)
=\beta_{h^+}(o,h),$$
we have
$$d\tilde \mu^{\Haar}_{gH}(gh)=e^{(n-1) \beta_{(gh)^+}(o,gh)} d\mu_o((gh)^+)=
e^{(n-1) \beta_{h^+}(o,h)} d\mu_o(h^+)
=d\tilde \mu^{\Haar}_{H}(h)$$
proving the first claim.
The first claim shows that
$d\tilde\mu_{\bar H}$ is left $H$-invariant. Since $d_{H\cap M}$ is an $H\cap M$-invariant measure,
the product measure $d\tilde\mu_{H}(hm)=d\tilde\mu_{\bar H}(h) d_{H\cap M} (m)$ is a Haar measure of $H$.
 \end{proof}

\subsection{}\label{ldef}
Let $\G$ be a torsion-free, non-elementary, geometrically finite
subgroup of $G$. For any given compact subset $\Omega$ of $\G\ba G$,
the goal of the rest of this
 section is to describe the set $$\{h\in \G\ba \G H: ha_t\in \Omega \; \text{ for some $t>0$}\}.$$
 For $H$ horospherical, this turns out to be a compact subset.
For $H$ symmetric, we will obtain a thick-thin decomposition, and give estimates of the size of thin parts
with respect to the measures $\mu_H^{\PS}$ and $\mu_H^{\Haar}$ (Theorem \ref{yo}).

An element $\gamma\in \G$ is called {\it parabolic} if there exists a unique fixed point of $\gamma$ in $\partial(\bH^n)$, and
an element $\xi\in \Lambda(\G)$ is called a parabolic fixed point if it is fixed by a parabolic element of $\G$.
Let $\Lambda_p(\G)\subset\Lambda(\G)$ denote the set of all parabolic fixed points
of $\G.$ Since $\G$ is geometrically finite, each parabolic fixed point $\xi$ is
bounded, i.e., $\op{Stab}_\G(\xi)$ acts co-compactly
on $\Lambda(\G)-\{\xi\}$.
Recall the notation $g^+= g(X_0)^+$ and $g^-=g(X_0)^-$.

 Consider the upper half space
model for $\bH^n$: $\bH^n=\{(x, y): x\in \br^{n-1}, y\in \br_{>0}\}$.
We set $\br^{n}_+:=\{(x,y): x\in \br^{n-1}, y\in \br_{>0}\}$, so that
$\partial (\br^n_+)=\{(x,0):x\in \br^{n-1}\}$.
Suppose that $\infty$ is a parabolic fixed point for $\G.$
Set $\G_\infty:=\op{Stab}_\G(\infty)$ and let $\G'(\infty)$ be
a
normal abelian subgroup of $\G_\infty$ which is of finite index; this exists
by a theorem of Biberbach.
Let $L$ be a minimal
$\G_\infty$-invariant subspace in $\br^{n-1}$.
By Prop. 2.2.6 in \cite{Bo}, $\G'(\infty)$ acts as translations and cocompactly on $L$.
We note that $L$ may not be unique, but any two such are Euclidean-parallel.

The notation $d_{\text{Euc}}$ and $\|\cdot\|$ denote the Euclidean distance and the Euclidean norm in $\br^{n}$ respectively.
Following Bowditch \cite{Bo},
we write for each $r>0$:
\begin{equation}
C(L, r):=\{x\in \bbr^n_+\cup \partial (\bbr^n_+): d_{\text{Euc}} (x, L)\ge r\} .\end{equation}

Each $C(L,r)$ is $\G_\infty$-invariant and called
 a standard parabolic region (or an extended horoball) associated to $\xi=\infty$.
\begin{thm}\cite[Prop. 4.4]{Bo}\label{bow1}\label{cinf}
 For any $\e_0>0$, there exists $r_0>0$ such that for any $r\ge r_0$,
\begin{enumerate}
 \item $\gamma C(L, r)=C(L, r)$ if $\gamma\in \G_\infty$;
\item if $\gamma\in \G - \G_\infty$, $d_{\text{Euc}}(C(L,r), \gamma
    C(L, r))>\e_0$.
\end{enumerate}
\end{thm}

\begin{cor}\label{proper}
Suppose that $\infty$ is a bounded parabolic fixed point for $\G$. Then
for any sufficiently large $r$, the natural projection map
$${\G}_\infty\ba (C(L,r)\cap \bH^n) \to \G\ba \bH^n$$
is injective and proper.
\end{cor}
\begin{proof}
We fix $\e_0>0$, and let $r_0>0$ be as in the above theorem \ref{bow1}. Let $r>r_0$, and set
$C_\infty=C(L, r)\cap \bH^n$ for simplicity.
 The injectivity is immediate from Theorem \ref{bow1}(2).
 Since $C_\infty$ is closed in $\bH^n$, so is $\gamma C_\infty$ 
 for all $\gamma\in \G$. Hence to prove the properness of the map, it is sufficient to show that
 if $F$ is a compact subset of $\bH^n$, then $F$ intersects at most finitely many distinct $\gamma C_\infty$'s.
 Now suppose there exists an infinite sequence $\{\gamma_i\in \G\}$ such that $\gamma_i \G_\infty$'s are all distinct from each other and
  $F\cap \gamma_i C_\infty\ne\emptyset$. Choosing $y_i\in F\cap \gamma_i C_\infty$,
by Theorem \ref{bow1}(2), we have
$d(y_i, y_j)\ge \e_0$ for all $i\ne j$, which contradicts the assumption that $F$ is compact.
\end{proof}

\subsection{$H$ horospherical}

\begin{thm}\label{horo} Let $H=N$ be a  horospherical subgroup.
Suppose that $\G\ba \G NM$ is closed in $\G\ba G$.
For any compact subset $\Omega$ of $\G\ba G$,
the set $\G\ba \G NM\cap  \Omega A$
is relatively compact.
\end{thm}

\begin{proof} We may assume without loss of generality that
the horosphere $NK/K$ in $G/K\simeq \bH^n$ is based at $\infty$.
Note that $\G_\infty\subset NM$ and that the closedness of $\G\ba \G NM$ implies
that $\G_\infty \ba NM \to \G\ba G$ is a proper map.

Therefore, if the claim does not hold,
 there
exists a sequence $n_i\in NM$ which is unbounded
 modulo $\G_\infty$ such that
$\gamma_i n_ia_{t_i}  \to x$ for some $t_i\in
\br$, $\gamma_i\in \G$ and $x\in G$.

It follows that, passing to a subsequence, $n_i a_{t_i}(o)\to \infty$ and
$d(n_ia_{t_i}, \gamma_i^{-1} x)\to 0$ as $i\to \infty$.
Therefore $\gamma_i^{-1}x(o)\to \infty$ and hence $\infty\in \Lambda(\G)$.

Since the image of the horosphere $N(o)$ in $\G\ba \bH^n=\G\ba G/K$ is closed,
it follows that $\infty$ is a bounded parabolic fixed point for $\G$ by \cite{Da}.
 Therefore
 $\G_\infty$ acts co-compactly on an $r$
neighborhood of a minimal $\G_\infty$-invariant subspace $L$ in $\bh\setminus\{\infty\}=\br^{n-1}$ for some $r>0$.
Write $n_ia_{t_i}(o)=(x_i, y_i)\in \bbr^{n-1}\times\bbr_{>0}$.
Since $\{n_i\}$ is unbounded
modulo $\G_\infty$, after passing to a subsequence if necessary,
we have $d_{\op{Euc}}(x_i, L)\to\infty$.
It follows that for any $r$, $(x_i, y_i)\in C(L,r)$ for all large $i$.
Since $n_i$ is unbounded modulo $\G_\infty$,
we get  $n_ia_{t_i}=(x_i, y_i)$ is unbounded in $\G_\infty\ba C(L,r)$.
Thus by Corollary
\ref{proper},  $n_ia_{t_i}$ must be unbounded modulo $\G,$ which is a
contradiction.
\end{proof}


\subsection{$H$ symmetric}\label{sec;return-sym}
We now consider the case when $H$ is symmetric.
Then $H(o)=H/H\cap K$ is a totally geodesic submanifold in $G(o)=G/K=\bH^n$.
We denote by $\pi$ the canonical projection from $G$ to $G/K=\bH^n$.
We set  $\tilde S=H(o)$.

Fixing a compact subset $\Omega$ of $\G\ba G$,
define
$$H_\Omega:=\{h\in H: \G\ba \G ha_t\in \Om\text{ for some $t>0$}\}$$
and set $\tilde S_\Om=H_\Om (o)$.

\begin{lem}\label{l;return-limitset-close}
Let $\xi\in \partial(\tS)$. If $\xi\notin \Lambda(\G)$, then
there exists a neighborhood $U$ of $\xi$ in $\overline {\bH^n}$ such
that $U\cap \tS_\Om=\emptyset$.
\end{lem}
\begin{proof}  Let $\Om_0$ be a compact subset of $G$ such
that $\Om=\G\ba \G \Om_0$.
If the claim does not hold, then there exist $h_n\in H$, $\gamma_n\in \G$ and $t_n>0$
such that $\gamma_n h_na_{t_n}\in \Om_0$
and $h_n (o)\to \xi$.
Note that $\{h_na_t(o ): t>0\}$ denotes the (half) geodesic emanating from $\pi(h_n)$
and orthogonal to $\tS.$  Since $h_n(o )$ converges to
$\xi\in\partial\bH^n,$  it follows that $h_n a_{t_n}(o)\to \xi$.

Now since $\Omega_0$ is compact, by passing to a subsequence in necessary,
we may assume $\gamma_nh_na_{t_n}\to x$.
As $G$ acts by isometries on $\bH^n$, we get $\gamma_n^{-1} (x( o))\to \xi.$ This implies $\xi\in \Lambda(\G)$ which is a
contradiction.
\end{proof}


Fix a Dirichlet domain $\mathcal D$ for $H\cap \G$
in $\tS$  and set \be\label{do}\mathcal D_\Om=\mathcal D\cap \tS_\Om.\ee

\begin{cor}\label{c;return-close-to-ps} Assume that the orbit $\G\ba \G H$ is closed in $\G\ba G$.
There exist a compact subset $Y_0$ of $\mathcal D$ and a finite subset
 $\{\xi_1,\ldots\xi_m\}\subset \Lambda_p(\G)\cap \partial(\tilde S)$
such that
$$ \mathcal D_\Om \subset Y_0\cup (\cup_{i=1}^m U(\xi_i))$$
where $U(\xi_i)$ is a neighborhood of $\xi_i$ in $\overline{\bH^n}$. In particular
if  $\Lambda_p(\G)\cap \partial(\tilde S)=\emptyset$, then $\mathcal D_\Om$
is relatively compact.
\end{cor}
\begin{proof}
For each $\xi \in \partial (\tS)\cap \partial (\mathcal D)$,
let $U(\xi)$ be a neighborhood of $\xi$ in $\overline{\bH^n}$. When $\xi\notin \Lambda(\G)$,
we may assume by Lemma \ref{l;return-limitset-close} that
$U(\xi)\cap \tS_\Om =\emptyset$.
By the compactness of
$\partial (\tS)\cap \partial (\mathcal D)$,
there exists a  finite covering $\cup _{i\in I} U(\xi_i)$.
Set $Y_0:=\mathcal D-\cup _{i\in I} U(\xi_i)$, which is a compact subset.
Now $\mathcal D_\Om - Y_0\subset \cup_{i\in I, \xi_i\in \Lambda(\G)} U(\xi_i)$ by the choice of $U(\xi_i)$'s.
On the other hand, by \cite[Proposition 5.1]{OS},
 we have $\Lambda(\G)\cap\partial\mathcal{D}\subset \Lambda_p(\G)$.
Hence the claim follows.
\end{proof}



In the rest of this subsection,
we fix $\xi\in \Lambda_p(\G)\cap \partial(\tS)$,
and investigate $\mathcal D_\Om\cap U(\xi)$.
We consider the upper half space model for $\bH^n$ and assume that $\xi=\infty$.
In particular, $\tS$ is a vertical plane.
Let $\G_\infty$, $\G'(\infty)$, $L$ and $C(L, r)$
be as in the subsection \ref{ldef}. Without loss of generality, we assume $0\in L$.
We consider the orthogonal decomposition $\br^{n-1}=L\oplus L^\perp$ and let
$P_{L^\perp}:\br^{n-1}\to L^{\perp}$ denote the orthogonal projection map.
\begin{lem}\label{l;cusp-return}
There exist $R_0>0$ 
such that
for any $h \in H_\Om$, we have
$\|P_{L^\perp}(h^+)\|\le R_0$.
 \end{lem}
\begin{proof}
  Let $\Om_0$ be a compact subset of $G$ such
that $\Om=\G\ba \G \Om_0$.
Then by Corollary \ref{proper}, and part (1) of Theorem \ref{cinf},
there exists $R'_0>0$ depending on $\Omega,$ such that
\be\label{e;cusp-gamma-inv}
\mbox{if $x\in C(L, R'_0)\cap \bH^n,$ then $x\not\in \G \Omega_0.$}
\ee

Suppose now that $h\in H_\Om,$ thus $ha_{t_0}(o)\in \G \Om_0$ for some $t_0>0$.
This, in view of~\eqref{e;cusp-gamma-inv} and the definition of $C(L,R'_0),$ implies $d_{\rm Euc}(ha_{t_0}(o), L)<R'_0$.

As discussed above, $\{ha_t: t>0\}$ is the geodesic ray
emanating from $h(o )$ and orthogonal to $\tS$ i.e. a Euclidean semicircle
orthogonal to the vertical plane.
Hence there exists some absolute constant $s_0$ such that
\[
\lim_{t\rightarrow\infty}d_{\rm Euc}(ha_{t}(o), L)\leq d_{\rm Euc}(ha_{t_0}(o), L)+s_0 \leq R'_0+s_0,
\]
which implies $\|P_{L^\perp}(h^+)\|\leq R_0:=R'_0+s_0,$ as we wanted to show.
\end{proof}

For $N\ge 1$,
set
\be\label{uni} U_N(\infty):=\{x\in \br^{n}_+\cup\partial(\br^{n}_+): \|x\|_{\rm Euc}>N \}.\ee

Let $\Delta:=\G'(\infty)\cap H$ and
let $p$ be the difference of the rank of $\G'(\infty)$ and the rank of $\Delta$.
Suppose $p\ge 1$.
Let ${\bf{\gamma}}=(\gamma_1,\ldots,\gamma_p)$ be a $p$-tuple of elements of $\G'$ such that
the subgroup generated by ${\bf \gamma}\cup\Delta$ has finite index in $\G'$.
For ${\bf k}=(k_1, \cdots , k_p)\in \z^p$,
we write ${\bf \gamma}^{\bf k}=\gamma_1^{k_1}\cdots \gamma_p^{k_p}$.
The notation $|\bf k|$ means the maximum norm
of $(k_1, \cdots, k_p)$.

The following gives a description of cuspidal neighborhoods of $\mathcal D_\Om$:
\begin{thm}\label{inc2}
There exist $c_0\ge 1$ and a compact subset
$\mathcal F$ of $\mathbb{R}^{n-1}$ such that for all large $N\gg 1$,
$$
\{h^+\in \br^{n-1}: h\in H, \pi(h)\in \mathcal D_\Om \cap U_{c_0N}(\infty)\}\subset
\cup_{|{\bf k}|\ge N} \Delta  {\bf \gamma}^{\bf k} \mathcal F.
$$
\end{thm}
\begin{proof}
In \cite[Prop. 5.8]{OS}, it is shown that for some $c_0\ge 1$ and a compact subset
$\mathcal F$ of $\mathbb{R}^{n-1}$,
\begin{equation} \label{osc} \{h^+\in \Lambda(\G):h\in H, \pi(h)\in\mathcal D\cap U_{c_0N}(\infty) \}\subset
 \cup_{|{\bf k}|\ge N} \Delta  {\bf \gamma}^{\bf k} \mathcal F \end{equation}
for all large $N\gg 1$.
However the only property of $h^+\in \Lambda(\G)$ used in this proof is
the fact that $\sup_{h\in H, h^+\in \Lambda(\G)}\|P_{L^\perp}(h^+)\|<\infty$.
Since this property holds for the set in concern by Lemma \ref{l;cusp-return},
the proof of Proposition 5.8 of \cite{OS} can be used.
\end{proof}



\subsection{Estimates on the size of thin part}
For $\xi \in \partial(\bH^n)$,
let $U_{N}(\xi)$ be defined to be $g(U_{N}(\infty))$ where $g\in G$ is such that $\xi=g(\infty)$
and $U_N(\infty)$ is defined as in \eqref{uni}.

\begin{prop}\label{mest} Let $\xi \in \partial(\tS)\cap \Lambda_p(\G)$ and $p_\xi:=\op{rank}(\G_\xi)
- \op{rank}(\G_\xi \cap H)$.
For all $N\gg 1$, we have
 $$\tilde \mu^{\PS}_{ H}\{h\in H: \pi(h)\in \mathcal D_\Omega \cap U_{N}(\xi)\} \ll N^{-\delta+p_\xi};$$
 $$\tilde \mu^{\Haar}_{ H}\{ h\in H: \pi(h)\in \mathcal D_\Omega\cap U_{N}(\xi) \} \ll N^{-n+1+p_\xi}$$
with the implied constants independent of $N$.
\end{prop}
\begin{proof}
The first claim is shown in \cite[Proposition 5.2]{OS}.
Without loss of generality, we may assume that $\xi=\infty$.
By replacing $\delta$ by $n-1$ and $\nu_o$ by $m_o$ in the proof of \cite[Proposition 5.2]{OS},
we get
$$\int_{h^+\in {\bf \gamma}^{\bf k} \mathcal F} e^{(n-1)\beta_{h^+}(o,h)}dm_o(h^+)\asymp |{\bf k}|^{-n+1}$$
where  the notation $\bf \gamma$, $\bf k$ and $\mathcal F$  are as in Theorem \ref{inc2} and
$f({\bf k}) \asymp g({\bf k})$ means that the ratio of $f({\bf k})$ and $ g({\bf k})$ lies in between two bounded constants independent of ${\bf k}$.

Hence by Proposition \ref{inc2},
$$\tilde \mu^{\Haar}_{ H}\{ h\in H: \pi(h)\in \mathcal D_\Omega\cap U_{N}(\infty)\} \ll
\sum_{{\bf k}\in \z^{p_\infty}, {|\bf k|}\ge N}|{\bf k}|^{-n+1} \ll N^{-n+1+p_\infty}.$$
\end{proof}

Recall the notion of the parabolic-corank of $\G$
with respect to $H$, introduced in \cite{OS}:
$$\text{Pb-corank}_H(\G):=\max_{\xi\in \Lambda_p(\G)\cap \partial(\tilde S)}
\left( \text{rank}(\G_\xi)-\text{rank}(\G_\xi\cap H) \right).$$

The following is shown in \cite[Thm. 1.14]{OS}:
\begin{prop}\label{cri}
 We have $\op{Pb-corank}_H(\G)=0$ if and only if the support of $\mu_H^{\PS}$ is compact,
and
$\op{Pb-corank}_H(\G)<\delta $ if and only if  $\mu_H^{\PS}$ is finite.
\end{prop}
It is also shown in \cite[Lem. 6.2]{OS} that $\text{Pb-corank}_H(\G)$
is bounded above by $n-\text{dim}(H/(H\cap K))$. Therefore if $H$ is locally isomorphic to $\SO(k,1)\times \SO(n-k)$
and $\delta>n-k$, then  $\mu_H^{\PS}$ is finite.

For $h\in \G\ba G$, we denote by $r_{h}$ the injectivity radius, that is, the map $g\mapsto hg$ is injective on the set $d(g,e)\le r_h$.
By Corollary \ref{c;return-close-to-ps}, \eqref{osc}, Proposition \ref{mest},  and by the structure of the support of $\mu_H^{\PS}$
obtained in \cite{OS}, we have the following:

\begin{thm}\label{yo} Suppose that $\G\ba \G H$ is closed.
For any compact subset $\Om$ of $\G \ba G$, there exists an open subset $Y_\Om\subset\G\ba \G H$
containing the union $\supp(\mu_{H}^{\PS})\cup \{h\in \G\ba \G H: ha_t\in \Om \text{ for some $t>0$}\}$ and
satisfying the following properties:
\begin{enumerate}
\item if $\op{Pb-corank}_H(\G)=0$, $Y_\Om$ is relatively compact;
 \item if $\op{Pb-corank}_H(\G)\ge 1$, then the following hold:
 \begin{enumerate} 
\item $Y_\e:=\{h\in Y_\Om: r_{h}>\e\}$ is relatively compact;
\item  there exist $\xi_1,\cdots, \xi_m\in \Lambda_p(\G)\cap \partial (\tilde S)$ and $c_1>0$ such
that for all small $\e>0$,
$Y_\Om- Y_{\e} \subset \cup_{i=1}^m U_{{c_1}\e^{-1}}(\xi_i);$

\item 
for all small $\e>0$, $$ \mu^{\PS}_{ H} (Y_\Om - Y_{\e}) \ll \e^{\delta-p_0}\quad\text{and}\quad
 \mu^{\Haar}_{ H} (Y_\Om -Y_\e ) \ll \e^{n-1-p_0}$$
for $p_0:=\op{Pb-corank}_H(\G)$.
\end{enumerate}\end{enumerate}
\end{thm}

\section{Translates of a compact piece of $\G\ba \G H$ via thickening}\label{localt}
Let $\G$ be a non-elementary geometrically finite subgroup of $G$.
Let $H$ be either symmetric or horospherical, and let $A=\{a_t\}$, $M$, $K$, $N^{\pm}$, $o$, $X_0$
be as in the subsection \ref{basicset}.

\subsection{Decomposition of measures}
Set $P:=MAN^-$, which is the stabilizer of $X_0^+$.
The measure $$dn_0=e^{(n-1)\beta_{n_0^-}(o, n_0)}dm_o(n_0^-)$$ can be seen to be a Haar measure on $N^-$
by a similar argument as in Lemma \ref{inv}.
Then  for $p=n_0a_tm\in N^-AM$,
$$dp:= dn_0dtdm$$
is a right invariant measure on $P$
where $dm$ is the probability Haar measure of $M$ and $dt$ is the Lebesgue measure on $\br$.




For $g\in G$,  consider the measure
on $gP$ given by
\be\label{nup}
\mbox{$d\nu_{gP}(gp)=e^{\delta t} d\nu_o((gp)^-)dt\;$ for $\;t=\beta_{(gp)^{-1}}(o, gp).$}
\ee

For $\Psi\in C_c(G)$,
we have:
\be\label{mea} \tilde m^{\Haar}(\Psi)=\int_{gP}\int_{N}\Psi(gpn) dn\; dp;\ee
 \be\label{mea3} \tilde m^{\BR}(\Psi)=\int_{gP}\int_{N}\Psi(gpn)d\tilde \mu^{\Haar}_{gpN}(gpn) d\nu_{gP}(gp);\ee
\be \label{mea2} \tilde m^{\BMS}(\Psi)=\int_{gP}\int_{N}\Psi(gpn)d\tilde \mu^{\PS}_{gpN}(gpn) d\nu_{gP}(gp).\ee


\subsection{Approximations of $\Psi$}
We fix a left invariant metric $d$ on $G$, which is right $H\cap M$-invariant
and which descends to the hyperbolic metric on $\bH^n=G/K$.
For a subset $S$ of $G$ and $\e>0$, $S_\e$ denotes the $\e$-neighborhood of $e$ in $S$:
$S_{\e}=\{g\in S: d(g, e)\le \e\}$.

We fix  a compact subset $\Omega$ of $\G\ba G$.
Let $r_0:=r_\Omega$ denote the infimum of the injectivity radius over all $x\in \Omega$. That is,
for all $x\in \Omega$,
the map $g \mapsto xg$ in injective on the set $\{g\in G: d(g,e)<r_\Omega\}$.

We fix a function $\kappa_{\Om} \in C^\infty(\G\ba G)$ such that
$0\le \kappa_{\Om}\le 1$, $\kappa_\Om(x)=1$ for all $x$ in the $\tfrac{r_0}{2}$-neighborhood of $\Om$ and
and $\kappa_{\Om}(x)=0$ for $x$ outside the $r_0$-neighborhood of $\Om$.

Fix $\Psi\in C^\infty(\Om)$.
For all small $\e>0$, set
\begin{equation}\label{peee}
\Psi^+_\e(x)=\sup_{g\in G_\e}\Psi(xg)\quad\text{and}\quad \Psi_\e^-(x)=\inf_{g\in G_\e}\Psi(xg).\end{equation}

For each $0<\e \le r_\Omega$,  $x\in \G\ba G$ and $g\in G_\e$,
we have
\be\label{psin} \Psi_\e^-(x)\le \Psi(xg)\le \Psi_\e^+(x)\ee
and $$\;|\Psi^{\pm}_\e(x)-\Psi(x)| \le c_1 \e \S_{\infty, 1}(\Psi)\kappa_\Om (x)$$
for some absolute constant $c_1>0$.

For $\bullet=\Haar, \BR, \BR_*$ or $ \BMS$,
we define
$$
A_\Psi^{\bullet}= \S_{\infty,1}(\Psi) \cdot m^{\bullet}(\supp(\Psi)).
$$

Define for each $g\in G$, $$\phi_0(g)=|\nu_{g(o)}|.$$
 Then
 $\phi_0$ is left $\G$-invariant and right $K$-invariant, and hence
induces a smooth function in $C^\infty(\G\ba G)^K=C^\infty(\bH^n)$. Moreover
$\phi_0$ is an eigenfunction of the Laplacian with eigenvalue $\delta(n-1-\delta)$
\cite{Sullivan1984}.

\begin{lem}\label{brbms}
For a compact subset $\Omega$ of $\G\ba G$,
\begin{enumerate}
 \item $m^{\BR}(\Omega) \ll \sup_{x\in \Omega} \phi_0(x) \cdot m^{\Haar}(\Omega K)$;
\item $m^{\BR}_*(\Omega) \ll \sup_{x\in \Omega} \phi_0(x) \cdot m^{\Haar}(\Omega K )$;
\item $m^{\BMS}(\Omega) \ll  \sup_{x\in \Omega} \phi_0(x)^2 \cdot m^{\Haar} (\Omega K).$
\end{enumerate}
\end{lem}
\begin{proof}
The first two claims follow since for any
$K$-invariant function $\psi$ in $\G\ba G$,
 $m^{\BR}_*(\psi)=m^{\BR}(\psi)=\int_{\G\ba G} \psi(g) \phi_0(g) dm^{\Haar}(g)$.
The third one follows from the smearing argument of Sullivan, see~\cite[Proof of Prop. 5]{Sullivan1984}.
\end{proof}

On the other hand, there exists $\ell \in\N$ such that
for all $\Psi\in C^\infty(\Om)$, $\S_{\infty,1}(\Psi)\ll \S_\ell ({\Psi})$ \cite{Au}.
Hence it follows from Lemma \ref{brbms} that there exists $\ell\in \N$ such that
 for all $\Psi\in C^\infty(\Om)$, any $\bullet=\Haar, \BR,\BR_*$ or $ \BMS$, and any $0<\e <r_\Om$,
\be \label{bulleti} A_\Psi^\bullet  \ll \S_{\infty,1}(\Psi) \cdot m^{\Haar}(\supp(\Psi))\ll
{\S_\ell (\Psi) }
\quad \text{ and }\quad \S_{\ell }(\Psi^{\pm}_\e) \ll \S_\ell (\Psi)\ee
where the implied constants depend only on $\Om$.




\subsection{Thickening of a compact piece of $yH$}\label{th}

For the rest of this section,
fix $y\in \G\ba G$ and
$H_0\subset H$ be a compact subset such that
the map $h \mapsto yh$ is injective on $H_0$.
 Fix $0<\e_0<r_\Om$ which is smaller than the injectivity radius of $yH_0$.

Fix non-negative functions $\Psi\in C^\infty(\Omega)$ and $\phi\in C^\infty(yH_0)$.
Let $M'\subset M$ be a smooth cross section for $H\cap M$ in $M$ and set $P':=M'AN^-$.
As $hp=h'p'$ implies $h=h'm$ and $p=m^{-1}p'$ for $m\in H\cap M$,
it follows that the product map $H\times P'\to G$ is a diffeomorphism onto its image, which
is a Zariski open neighborhood of $e$.
Let $dp'$ be a smooth measure on $P'$ such that
$dp=d_{H\cap M}m dp' $ for $p=mp'$.
For $0<\e<\e_0$, let $\rho_\e\in C^\infty(P'_\e)$ be a non-negative function
such that $\int\rho_\e dp'=1$, and
we define $\Phi_\e\in C_c^\infty(\G\ba G)$
 by \be \Phi_\e(g)=\begin{cases} \phi(yh)\rho_\e(p) &\text{if $g=yhp\in yH_0 P'_\e$} \\
                     0&\text{otherwise}.
                    \end{cases}\ee

\begin{lem} \label{comp}
For all $0<\e<\e_0$ and $t>0$,
$$\int_{\G\ba G} \Psi_{\e}^- (ga_t) \Phi_{\e}(g) dg \le
\int_{h\in H_0} \Psi(yha_t ) \phi(yh)dh\le
\int_{\G\ba G} \Psi_{\e}^+ (ga_t) \Phi_{\e}(g) dg .$$
\end{lem}

\begin{proof}
For all $p\in P'_\e$, $h\in H_0$ and $t>0$,
$yh p a_t=yha_t(a_{-t} pa_t)\in yha_tP_\e$
and hence
$$ \int_{h\in H_0} \Psi(yha_t ) \phi(yh)dh
 \le   \int_{h\in H_0} \Psi_\e^+(yh p a_t ) \phi(yh) dh.$$
Integrating against $\rho_\e$,
we have
\begin{align*}
&\int_{h\in H_0} \Psi(yha_t ) \phi(yh)dh
\\ \le  &  \int_{yhp\in yH_0P_\e'} \Psi_\e^+(yh p a_t ) \phi(yh)\rho_\e(p) dhd p
\\ &=\int_{\G\ba G} \Psi_\e^+(g a_t ) \Phi_\e ( g)  dg
.\end{align*}
The other direction is proved similarly.
\end{proof}

\begin{lem}\label{mbre} For all $0<\e<\e_0$,
$$m^{\BR}_*(\Phi_\e)=(1+O(\e)) \mu^{\PS}_{yH}(\phi).$$
\end{lem}
\begin{proof}
Choose $g_y\in G$ so that $y=\G\ba \G g_y$
and set $\tilde \phi(g_yh):=\phi(yh)$
and $\tilde \Phi_\e(g_yhp):=\tilde \phi(g_yh)\rho_\e(p)$ for $hp\in H_0P_\e'$ and
zero otherwise.
As $0<\e<\e_0$, we have
$m^{\BR}_*(\Phi_\e)=\tilde m^{\BR}_*(\tilde \Phi_\e)$ and
 $ \mu^{\PS}_{yH}(\phi)=\tilde\mu^{\PS}_{yH}(\tilde \phi)$.
By the definition, we have
 \begin{align*}
\tilde m^{\BR}_*(\tilde \Phi_\e)&=\int_{g\in  G/M}
\int_M  \tilde \Phi_\e (gm) dm  \; e^{\delta \beta_{g^+}(o, g)}\;
 e^{(n-1) \beta_{g^-}(o, g) }\; d\nu_o(g^+) dm_o(g^-) ds
\end{align*}
where $s=\beta_{g^-}(o, g)$.
For simplicity, we set $g_y=y\in G$ by abuse of notation.
For $g=yhp\in H_0P'_\e$,
as $|\beta_{g^+}(yh, g)|\le d(yh, yhp)=d(e,p)\le \e$,
we have
$e^{\delta\beta_{g^+}(yh, g)} =1+O(\e)$.
Since $g^+=(yh)^+$,
we have
$$e^{\delta \beta_{g^+}(o, g)}d\nu_o(g^+)= (1+O(\e)) e^{\delta \beta_{(yh)^+}(o, yh)}d\nu_o((yh)^+)
=(1+O(\e)) d\tilde \mu_{y\bar H}^{\PS}(yh).$$
On the other hand, as $\{m_x\}$ is $G$-invariant,
$$dm_o(g^-)=dm_{(yh)^{-1}(o)} (p^-)= e^{(n-1) \beta_{p^-}(o, (yh)^{-1}(o))} dm_o(p^-).$$

Since  $p^-=n_0^-$ for $p=n_0a_tm$, we have  \begin{align*} &\beta_{g^-}(o, g) +\beta_{p^-} (o, (yh)^{-1}(o))\\ &=
\beta_{p^-}((yh)^{-1}(o), p) +\beta_{p^-} (o, (yh)^{-1}(o))\\&=\beta_{p^-}(o,p)
\\ & =\beta_{n_0^-}(o, n_0 a_t)=
\beta_{X_0^-}(o,  a_t)+\beta_{n_0^-}(o, n_0 )
\\ &= -t +\beta_{n_0^-}(o, n_0 ).\end{align*}

As $n_0a_tm\in P_\e$,  we have $e^{-(n-1)t}=1+O(\e)$
 and hence 
 \begin{align*}& e^{(n-1)\beta_{g^-}(o, g)} dm_o(g^-) ds dm
\\ &=e^{(n-1)(\beta_{g^-}(o, g)+  \beta_{p^-} (o, (yh)^{-1}(o))} dm_o(p^-) ds dm
\\ &= 
e^{-(n-1)t}  e^{(n-1)\beta_{n_0^-}(o,n_0)}dm_o(n_0^-) dt dm \\
&=e^{-(n-1)t}  dn_0 dt dm=
(1+O(\e)) dp.\end{align*}

Since $dp=d_{H\cap M}(m)dp'$ for $p=mp'$,
for $\overline \phi(yh):=\int_{H\cap M}\tilde{\phi}(yhm) d_{H\cap M} (m)$, we have
 \begin{align*}
\tilde m^{\BR}_*(\tilde \Phi_\e)&=(1+O(\e)) \int_{P_\e'} \int_{yh\in yH_0/(H\cap M)}
  \overline{\phi}(yh) \rho_\e(p')   d\tilde \mu_{y\bar H}^{\PS}(yh)
 dp'
\\ &= (1+O(\e)) \tilde \mu_{yH}^{\PS}(\tilde \phi) .
\end{align*}
\end{proof}

\begin{cor}\label{psi} There exists $\ell\in \N$ such that
for any $\phi\in C^\infty(yH_0)$,
$\mu_H^{\PS}(\phi)\ll \S_\ell(\phi)$
where the implied constant depends only on the compact subset $yH_0$.
\end{cor}
\begin{proof}
By Lemmas \ref{mbre}, \ref{brbms} and \eqref{bulleti}, there exists $\ell\in\N$ such that
 $$\mu_H^{\PS}(\phi)\ll m_*^{\BR}(\Phi_{\e_0})\ll \S_\ell (\Phi_{\ell})\ll \S_{\ell} (\phi) \S_\ell(\rho_{\e_0})
\ll  \S_{\ell} (\phi).$$
where the implied constants depending only on $\e_0$ and $yH_0$.
\end{proof}


\begin{thm}\label{maince}\label{efn} Suppose that $\G$ is Zariski dense in $G$ and that $L^2(\G\ba G)$ has a spectral gap.
Then there exist $\eta_0>0$ and $\ell\ge 1$ such that
for any $\Psi\in C^\infty(\Omega)$
and $\phi\in C^\infty(yH_0)$,
we have
\begin{multline*} e^{(n-1-\delta)t} \int_{yh\in yH} \Psi(yha_t ) \phi(yh)dh
\\ = \frac{1}{|m^{\BMS}|} m^{\BR}(\Psi)\tilde\mu^{\PS}_{yH}(\phi)  + e^{-\eta_0 t} O( \S_\ell (\Psi) \S_\ell (\phi)), \end{multline*}
with the implied constant depending on $\Om$ and $yH_0$.
\end{thm}

\begin{proof} It suffices to prove the claim for $\Psi$ and $\phi$ non-negative.
Let $\ell \ge 1$ be bigger than those
$\ell$'s in Theorem \ref{mc2}, \eqref{bulleti} and Corollary \ref{psi}.
Let $q_\ell>0$ (depending only on the dimension of $P'$) be such that
$\S_\ell(\rho_\e)= O(\e^{-q_\ell})$, so that
\[
\mathcal S_\ell(\Phi_\e)\ll \S_\ell(\phi) \S_\ell(\rho_\e)\ll \S_\ell(\phi)\e^{-q_\ell}.
\]

Note that
$\S_\ell(\Psi_\e^{\pm})\ll \S_\ell(\Psi)$ and that
$m^{\BR}(\Psi_\e^{\pm})=m^{\BR}(\Psi) + O(\e A_\Psi^{\BR})$.

By Lemma \ref{comp},
$$ \la a_t \Psi_\e^-, \Phi_\e\ra \le
 \int_{yh\in yH} \Psi(yha_t ) \phi(yh)dh \le \la a_t \Psi_\e^+, \Phi_\e\ra.$$

By Lemma
\ref{mbre} and Theorem \ref{mc2}, there exists $\eta>0$ such that
\begin{align*} & e^{(n-1-\delta)t}\la a_t \Psi_\e^{\pm}, \Phi_\e\ra
\\& =\tfrac{1}{|m^{\BMS}|} m^{\BR}(\Psi_\e^{\pm})m^{\BR}_*(\Phi_\e)   +
e^{-\eta t} O(\S_\ell (\Psi) \S_\ell (\phi) \e^{-q_\ell})
\\& =\tfrac{1}{|m^{\BMS}|} m^{\BR}(\Psi)\tilde \mu^{\PS}_{yH}(\phi)  + O(\e A_\Psi^{\BR}  \tilde \mu^{\PS}_{yH}(\phi)) +
e^{-\eta t} O(\S_\ell (\Psi) \S_\ell(\phi) \e^{-q_\ell}). \end{align*}

By taking $\e=e^{-\eta t/(1+q_\ell)}$ and $\eta_0 =\eta/(1+q_\ell)$,
we obtain that
\begin{align*}& e^{(n-1-\delta)t} \int_{yh\in yH} \Psi(yha_t ) \phi(yh)dh
\\&=\tfrac{1}{|m^{\BMS}|} m^{\BR}(\Psi)\tilde \mu^{\PS}_{yH}(\phi)  + e^{-\eta_0 t} O(A_\Psi^{\BR}\tilde \mu^{\PS}_{yH}(\phi) +
\S_\ell (\Psi) \S_\ell (\phi)).
\end{align*}
By \eqref{bulleti} and Corollary \ref{psi}, this
proves the theorem.
\end{proof}

We remark that we don't need to assume $yH$ is closed in the above theorem,
as $\phi$ is assumed to be compactly supported.

When $H$ is horospherical or symmetric with $\text{Pb-corank}_H(\G)=0$,
Theorem \ref{effa} is a special case of Theorem \ref{efn} by Theorem \ref{horo} and Theorem \ref{cri}.

\section{Distribution of $\G\ba \G H a_t$ and Transversal intersections }\label{sec;eff-equi}
Let $\G, H,A=\{a_t\}$, $P=MAN^-$, etc
be as in the last section \ref{localt}. We set $N=N^+$.
Let $\{\mu_x\}$ be a $\G$-invariant conformal density
of dimension $\delta_\mu>0$ and let $\tilde \mu_{gH}$ and $\tilde \mu_{gN}$ be the measures
on $gH$ and $gN$ respectively defined with respect to $\{\mu_x\}$.

\subsection{Transversal intersections}
Fix $x\in \G\ba G$. Let $\e_0>0$ be the injectivity radius at $x$.
 In particular, the product map
$P_{\e_0}\times N_{\e_0} \to \G\ba G$ given
by $(p,n)\mapsto xpn$ is injective.
For any $\e\leq\e_0$ we set $B_{\e}:=P_{\e}N_\e$.

For some $c_1>1$, we have
 $N_{c_1^{-1}\e} P_{c_1^{-1}\e}\subset  B_\e:=P_\e N_\e \subset N_{c_1\e}P_{c_1\e}$ for all $\e>0$.
Therefore, in the arguments below, we will frequently identify $B_\e$ with $N_\e P_\e$, up to a fixed Lipschitz constant.






In the next lemma,
let $\Psi\in C_c^\infty(xB_{\e_0})^{H\cap M}$ and
$\phi\in C_c^\infty(yH)^{H\cap M}$.
For $0<\e\leq \e_0,$  define $\psi_{\e}^\pm\in C^\infty(xP)$ by
$$
\psi_{\e}^\pm(xp)=\int_{xpN} {\Psi}_{\e}^\pm(xpn)  d\mu_{xpN}(xpn)
$$
where  ${\Psi}_{\e}^\pm$ are as given in \eqref{peee}.

Define $\phi_{\e}^{\pm}\in C_c^\infty(yH)$
by
\be\label{dp}\phi_{\e}^{+}(yh)=\sup_{h'\in H_\e} \phi(yhh') \quad \text{and}\quad
\phi_{\e}^{-}(yh)=\inf_{h'\in H_\e} \phi(yhh') .\ee
Since the metric $d$ on $G$ is assumed to be left $G$-invariant and right $H\cap M$-invariant, we have
$mH_\e m^{-1}=H_\e$ and $mN_\e m^{-1}=N_\e$. Therefore the functions $\psi_e^\pm$ and $\phi_\e^{\pm}$
are $H\cap M$-invariant.

The following lemma is analogous to Corollary 2.14 in \cite{OS}; however
we are here working in $\G\ba G$ rather than in $\T^1(\G\ba \bH^n)$ as opposed to \cite{OS}.
Let $$P_x(t):=\{p\in P_{\e_0}/(H\cap M): \supp(\phi) a_t \cap xpN_{\e_0}(H\cap M) \ne \emptyset\}.$$
\begin{lem}\label{com1}
For any $0<\e\ll \e_0$, we have
\begin{multline*}
(1-c\e) \sum_{p\in P_x(t)} {\phi_{ce^{-t}\e_0}^-}(xpa_{-t}) \psi_{c\e}^-(xp)\le
e^{\delta_\mu t} \int_{yH} \Psi(yha_t)\phi(yh) d\mu_{yH}(yh)\\  \le
(1+c\e) \sum_{p\in P_x(t)} {\phi_{ce^{-t}\e_0}^+}(xpa_{-t})\psi_{c\e}^+(xp),
\end{multline*}
where $c>0$ is an absolute constant, depending only on the injectivity radii of $\supp(\phi)$ and
$\supp(\Psi)$.
\end{lem}

\begin{proof}
 By considering a smooth partition of unity for the support of $\phi$, it suffices to prove the lemma,
 assuming $\mbox{supp}(\phi)\subseteq yN_\e P_\e\cap yH\subset yB_\e.$
Fix $g, g'\in G$ so that $y=\G g$ and $x=\G g'$.
Then for $\bar H=H/H\cap M$,
\begin{align*}
&\int_{yH} \Psi(yha_t)\phi(yh) d\mu_{yH}(yh)\\ &=
\sum_{\gamma\in (\G\cap gHg^{-1})\ba \Gamma}
\int_{\gamma gH} \Psi(y ha_t) \phi( yh) d\tilde\mu_{\gamma gH}(\gamma gh)\\
&=\sum_{\gamma\in (\G\cap gHg^{-1})\ba \Gamma}
\int_{\gamma g\bar H} \int_{H\cap M} \Psi(y ha_t m )  \;
  \phi(yhm)\; dm  d\tilde\mu_{\gamma g\bar H}(\gamma gh)\\
&=\sum_{\gamma\in (\G\cap gHg^{-1})\ba \Gamma}
\int_{\gamma g\bar H} {\Psi}( y ha_t)  \phi(y h)  \;
d\tilde\mu_{\gamma g\bar H}(\gamma gh)
\end{align*}
as $\Psi$ and $\phi$ are $H\cap M$-invariant and $dm$ is the probability Haar measure of $H\cap M$.


Suppose $yh\in \supp(\phi)\cap yH,$
and write $h=n_hp_h$ where $n_h\in N_\e$ and $p_h\in P_\e$.
As $h^+=n_h^+$ and $d(h, n_h)=O(\e)$,
we have that for any $\gamma \in\G$
\be\label{e;radon-nik-hn}
\frac{d\tilde\mu_{\gamma g\bar H}(\gamma g h)}{d\tilde\mu_{\gamma gN}(\gamma g n_h)}=1+O(\e).
\ee

Let $\gamma\in (\G\cap gHg^{-1})\ba \Gamma$.
If $\gamma gha_t=g'p_{h,t}n_{h,t}\in g' P_{\e_0}N_{\e_0}$,
then we claim that

\be\label{e;radon-nik-nn}
e^{\delta_\mu t} \frac{ d\tilde\mu_{\gamma gN}(\gamma gn_h)}{d\tilde\mu_{g'p_{h,t}N}(g'p_{h,t}n_{h,t})}= O(e^\e).
\ee

Note that
$\gamma gha_t=g'p_{h,t}n_{h,t}$
implies $\gamma gn_ha_t=g'p_{h,t}n_{h,t} (a_t^{-1}p_h a_t) $.
Hence
$\xi:=(\gamma gn_h)^+=(g'p_{h,t}n_{h,t})^+$, and for $p_{h,t}':=(a_t^{-1}p_h a_t)\in P_{\e}$,
\begin{align*}
&\beta_\xi(o, \gamma  gn_h) =\beta_\xi(o, g'p_{h,t}n_{h,t}) + \beta_\xi(g'p_{h,t}n_{h,t},g'p_{h,t}n_{h,t}p'_{h,t})\\
&+\beta_\xi(g'p_{h,t}n_{h,t}p'_{h,t},g'p_{h,t}n_{h,t}p'_{h,t}a_{-t})  = \beta_\xi(o, g'p_{h,t}n_{h,t})+O(\e) - t ,
\end{align*}
proving the claim \eqref{e;radon-nik-nn}.

Note that $xB_{\e_0}$ is the disjoint union $\cup_{p\in P_{\e_0}} xpN_{\e_0}.$
Since $n_{h, t}\in N_{\e_0}$ and $\gamma gh=g'p_{h,t}a_{-t} (a_t n_{h,t}a_{-t})$
with $a_t n_{h,t}a_{-t}\in N_{e^{-t}\e_0}$,
in view
of~\eqref{e;radon-nik-hn} and~\eqref{e;radon-nik-nn},
we have
\begin{align*}
&e^{\delta_\mu t}\int_{\gamma g\bar H}{\Psi}( y ha_t)\phi(y h) d\tilde\mu_{\gamma g\bar H}(\gamma gh) \\
&=(1+O(\e))  \sum_{p}  \phi^+_{c e^{-t} \e_0}( x pa_{-t}) \cdot
\int_{g'pN} {\Psi}^+_{c\e} ( x p n)d\tilde\mu_{g'pN}(g'pn)  \\
&=(1+O(\e)) \sum_{p}
\phi^+_{c e^{-t} \e_0}( x pa_{-t}) \cdot {\psi}^+_{c\e} ( x p )
\end{align*}
where the both sums are taken over the set of $p\in  {P}_{\e_0}/(H\cap M)$ such that
$ \gamma g H_{\e} a_t\cap  g'pN_{\e_0}(H\cap M) \ne \emptyset$ and $c>0$ is an absolute constant.


Summing over $\gamma\in (\G\cap gHg^{-1})\ba \Gamma$, we obtain
 one side of the inequality and
the other side follows if one argues similarly using $\Psi_{c\e}^-$.
\end{proof}


By a similar argument, we can prove the following:

\begin{lem}\label{com2}
Let $\phi\in C_c(yH)^{H\cap M}$ and $\psi\in C^\infty(xP_{\e_0})^{H\cap M}$.
Assume that $\mu_{xpN}(xpN_{\e_0})>0$ for all $p\in P_{\e_0}.$
 There exists $c>1$ such that for all small $0<\e\leq \e_0$,
\begin{multline*} (1-c\e) \int_{yH} \Psi^-_{c\e} (yha_t) \phi_{ce^{-t}\e_0}^- (yh)d\mu_{yH}(yh)  \le
e^{-\delta_\mu t} \sum_{p\in P_x(t)} \psi(xp) \phi (xpa_{-t}) \\  \le
(1+c\e)  \int_{yH} \Psi^+_{c\e} (yha_t) \phi_{ce^{-t}\e_0}^+(yh)d\mu_{yH}(yh)
\end{multline*}
where $\Psi\in C^\infty(xB_{\e_0})$ is defined by
$\Psi(xpn)=\frac{1}{\mu_{xpN}(xpN_{\e_0})} \psi(xp)$ for each $pn\in P_{\e_0}N_{\e_0}$.

\end{lem}

Similarly to the definitions of $A_{\Psi}^\bullet$, we define for $\phi\in C(yH)$ and $\psi\in C(xP_{\e_0})$,
$$
A_{\phi}^{\PS}:=\S_{\infty,1}(\phi) \cdot  \mu^{\PS}_{yH}(\text{supp}(\phi)),
\quad A_\psi^{\nu}:=\S_{\infty,1}(\psi) \cdot \nu_{xP}({\rm supp}(\psi))$$ where $\nu_{xP}$ is defined
as in \eqref{nup}.

By a similar argument as in \eqref{bulleti},
we have $ A_\phi^{\PS}\ll \S_\ell(\phi)$ and $A_\psi^\nu \le \S_\ell(\psi)$ for
some $\ell\in \N$.

\begin{lem}\label{brl}
Let $\psi\in C(xP_{\e_0})^{H\cap M}$.
For $\Psi\in C^\infty(xB_{\e_0})^{H\cap M}$ given by
$\Psi(xpn)=\frac{1}{\mu_{xpN}^{\Haar}(xpN_{\e_0})} \psi(xp)$, we have
$$m^{\BR}(\Psi)=\nu_{xP}(\psi) \quad\text{ and}\quad A_\Psi^{\BR}\ll A_\psi^{\nu} .$$
\end{lem}

\begin{proof} For $g=xpn$, we have $g^-=(xp)^-$
and $\beta_{(xp)^-}(o, xpn)=\beta_{(xp)^-} (o, xp)$. Based on this, the claims follow from the definition.
The second claim follows from $m^{\BR}(\supp(\Psi))=\nu_{xP}(\text{supp}(\psi))$ and $\S_{\infty, 1}(\Psi)\ll_{\e_0} \S_{\infty,1}(\psi)$.
\end{proof}

In the rest of this section, we assume that 
$$\text{ $\G$ is Zariski dense and $L^2(\G\ba G)$ has a spectral gap.}$$

\begin{thm}\label{tran}
There exist  $\beta>0$ and $\ell\in \N$ such that for any $0<\e\ll \e_0$
and any $\psi\in C^\infty(xP_{\e_0})^{H\cap M},$ and $\phi \in C^\infty_c(yH)^{H\cap M}$, we have
\begin{equation*} e^{-\delta t} \sum_{p\in P_x(t)} \psi(xp) \phi (xpa_{-t})=\frac{1}{|m^{\BMS}|} \nu_{xP}(\psi)
\mu_{yH}^{\PS}(\phi) +e^{-\beta t} O( \S_\ell (\psi) \S_\ell (\phi)),
\end{equation*}
where $P_x(t):=\{p\in P_{\e_0}/(H\cap M): \supp(\phi) a_t \cap xpN_{\e_0}(H\cap M) \ne \emptyset\}$ and
the implied constant depends only on the injectivity radii of ${\supp}(\psi)$ and
${\supp}(\phi)$.
\end{thm}

\begin{proof}
Define $\Psi^\pm_{\e}(xpn)=\frac{1}{\mu_{xpN}^{\Haar}(xpN_{\e})} \psi^{\pm}_\e(xp)$.
Then
$ m^{\BR}(\Psi^{\pm}_{\e} )=\nu_{xP}(\psi^{\pm}_{\e})$ by Lemma \ref{brl}.

We take $\ell$ big enough to satisfy
Theorem \ref{efn}, Corollary \ref{psi}
and that  $A_{\Psi}^{\BR}\ll A_\psi^{\nu}\ll \S_\ell(\psi)$ and $A_\phi^{\PS}\ll \S_\ell(\phi)$.

By Theorem \ref{efn}, for some $\eta_0>0$,
\begin{align*}
&e^{(n-1-\delta)t}\int_{yH} \Psi^\pm_\e (yha_t) \phi_{e^{-t}\e_0}^\pm (yh)dh \\
&= \tfrac{1}{|m^{\BMS}|} m^{\BR}(\Psi^\pm_\e)\mu_{yH}^{\PS}(\phi_{e^{-t}\e_0}^{\pm}) +
e^{-\eta_0 t} O ( \S_\ell (\Psi^\pm_\e) \S_\ell (\phi_{e^{-t}\e_0}^{\pm})) \\
&=m^{\BR}(\Psi)\mu_{yH}^{\PS}(\phi)
+O((\e +e^{-t})A_\Psi^{\BR}A_\phi^{\PS}) + e^{-\eta_0 t} O(\S_\ell (\Psi) \S_\ell (\phi))
\\ &= \nu_{xP}(\psi)\mu_{yH}^{\PS}(\phi)+ O((e^{-\eta_0t}+\e) \S_\ell (\psi) \S_\ell (\phi)).
\end{align*}

Therefore the claim now follows by applying
 Lemma \ref{com2} for $d\mu_{yH}(yh)=dh$ and $\delta_\mu=n-1$ with $\beta=\eta_0/2$ and $\e=e^{-\eta_0 t/2}$.
\end{proof}

Using Theorem \ref{tran},
we now prove the following theorem, which
 is analogous to Theorem \ref{maince} with $dh$ replaced by $ d\mu^{\PS}_{yH}(yh)$.
Translates of $d\mu^{\Haar}_{yH}$ and $d\mu^{\PS}_{yH}$ on $yH$ are closely related as
their transversals are essentially the same.
More precisely, Theorem \ref{tran} provides a link between translates of these two measures.

\begin{thm}\label{loc} 
 There exist  $\beta>0$ and $\ell\in \N$ such that for any $\Psi\in C^\infty(xB_{\e_0})^{H\cap M}$
and $\phi\in C_c^\infty(yH)^{H\cap M}$,
$$\int_{yH} \Psi(yha_t)\phi(yh) d\mu^{\PS}_{yH}(yh)=
\frac{1}{|m^{\BMS}|} m^{\BMS}(\Psi) \mu_{yH}^{\PS}(\phi)+  O(e^{-\beta t}\S_\ell(\Psi)\S_\ell(\phi)) .$$
\end{thm}

\begin{proof}
Define $\psi\in C^\infty(xP_{\e_0})^{H\cap M}$ by
$$\psi(xp)=\int_{xpN_{\e_0}} \Psi(xpn)  d\mu^{\PS}_{xpN}(xpn).$$

We apply Theorem \ref{tran} and Lemma \ref{com1}
for the Patterson-Sullivan density $\{\mu_x\}$ and with this $\psi$.
It follows from the definition of $\psi$ (see~\eqref{mea2}) that $\nu_{xP}(\psi)=m^{\BMS}(\Psi)$
and $A_\psi^\nu \ll \S_\ell(\Psi)$ for some $\ell\ge  1$.
We take $\ell$ large enough to satisfy Theorem \ref{tran}.

Let $\beta$ be as in Theorem~\ref{tran} and let $\e=e^{-\beta t}.$
Now by Lemma \ref{com1}, we get
$$\int_{yH} \Psi(yha_t)\phi(yh) d\mu^{\PS}_{yH}(yh)=
(1+O(\e))e^{-\delta t} \sum_{p\in P_x(t)} {\phi_{e^{-t}\e_0}^{\pm}}(xpa_{-t})\psi^\pm_{\e}(xp).  $$
By Theorem~\ref{tran},
\begin{align*}&
 e^{-\delta t} \sum_{p\in P_x(t)} {\phi_{e^{-t}\e_0}^{\pm}}(xpa_{-t})\psi^\pm_{\e}(xp)\\
&= \tfrac{1}{|m^{\BMS}|} \nu_{xP}(\psi_\e^\pm)\mu_{yH}^{\PS}(\phi_{e^{-t}\e_0}^{\pm})
+ e^{-\beta t} O(\S_\ell (\psi) \S_\ell (\phi)) \\
&=\tfrac{1}{|m^{\BMS}|} \nu_{xP}(\psi)\mu_{yH}^{\PS}(\phi)
+O(\e + e^{-\beta t}) (\S_\ell (\psi) \S_\ell (\phi)).\end{align*}
Since  $\nu_{xP}(\psi)=m^{\BMS}(\Psi)$
and $ \S_\ell(\psi), A_\psi^\nu \ll  \S_\ell (\Psi)$, this finishes the proof. \end{proof}

\subsection{Effective equidistribution of $\G\ba \G Ha_t$}
We now
extend Theorems \ref{maince} and \ref{loc} to bounded functions $\phi\in  C^\infty((\G\cap H)\ba
H)$ which
are not necessarily compactly supported.  Hence the goal is to establish the following:
 set $\G_H:=\G\cap H$.
\begin{thm} \label{meqfinite} \label{meq} Suppose that
 $\G\ba\G H$ is closed and that $|\mu_H^{\PS}|<\infty$. 
There exist $\beta>0$ and $\ell \ge 1$ such that for any compact subset $\Om \subset \G\ba G$,
for any $\Psi\in C^\infty(\Om)$ and any bounded $\phi\in  C^\infty(\G_H \ba H)$,
we have, as $t\to +\infty$,
\begin{multline*} e^{(n-1-\delta)t} \int_{h\in  \G_H \ba  H} \Psi(h a_t )\phi(h) dh
 = \frac{{\mu^{\PS}_{H} (\phi)} }{|m^{\BMS}|} m^{\BR}(\Psi) +
O(e^{-\beta t}  \S_\ell (\Psi)\S_\ell(\phi) ) \end{multline*}
where the implied constant depends only on $\Om$.
\end{thm}

We first prove the following which is an analogous version of Theorem \ref{meqfinite} for 
$\mu_H^{\PS}$:
\begin{thm}\label{t;eq-ps-h-orbit}  Suppose that
 $\G\ba\G H$ is closed and that $|\mu_H^{\PS}|<\infty$.
There exist $\beta_0>0$ and $\ell
\ge 1$ such that for any compact subset $\Om\subset \G\ba G$,
for any $\Psi\in C^\infty(\Om )^{H\cap M}$ and
 for any bounded $\phi\in  C^\infty((\G\cap H)\ba H)^{H\cap M}$,
we have
$$
 \int_{h\in \G_H\ba  H} \Psi(h a_t )\phi(h) d\mu_{H}^{\PS}(h )
= \frac{{\mu^{\PS}_{H}(\phi) } }{|m^{\BMS}|} m^{\BMS}(\Psi) + O(e^{-\beta_0 t} \S_\ell (\phi) \S_\ell (\Psi)).$$
\end{thm}

\begin{proof} Fix $\ell\in \N$ large enough to satisfy Theorem \ref{loc},
${A^{\PS}_\phi} \ll \S_\ell(\phi)$ and ${A^{\BMS}_\Psi} \ll \S_\ell(\Psi)$.
If $H$ is horospherical, set $Y_\Om=\{h\in \G\ba \G HM: ha_t\in \Om \text{ for some $t\in \br$}\}$; if $H$ is symmetric, 
let $Y_\Om$ and $Y_\e$ be as in Theorem \ref{cri} and set $p_0:=\text{Pb-corank}_H(\G)$.
For $\e>0$, we choose $\tau_\e\in C^\infty(Y_\Om)$ which is an $H\cap M$-invariant
smooth approximation of the set
$Y_\e$; $0\le \tau_\e \le 1$, $\tau_\e(x)=1$ for $x\in Y_{\e}$ and $\tau_\e(x)=0$
for $x\notin Y_{\e/2}$; we refer to \cite{BeO} for the construction of such $\tau_\e$.
Let $q_\ell\gg 1$ be such that $\S_\ell(\tau_\e)=O(\e^{-q_\ell})$.
By the definition of $Y_\Om$, we may write the integral
$\int_{h \in \G_H\ba H} \Psi(h a_t ) \phi(h) d\mu_{H}^{\PS}(h )$ as the sum
$$\int_{\G_H\ba  H} \Psi(h a_t ) (\phi \cdot \tau_\e) (h) d\mu_{H}^{\PS}(h )
+  \int_{Y_\Om} \Psi(ha_t ) (\phi -\phi\cdot \tau_\e )(h)d\mu_{H}^{\PS}(h ). $$

Note that by (3) of Theorem \ref{yo}, we have
$\mu^{\PS}_H(Y_\Om- Y_\e) \ll \e^{\delta-p_0}$
and hence $\mu^{\PS}_H(\phi - \phi \cdot \tau_\e) \ll {A^{\PS}_\phi} \cdot\e^{\delta-p_0}\ll \S_\ell(\phi)
\e^{\delta-p_0}$.
Now by Theorem \ref{loc},
\begin{align*} &\int_{\G\ba \G H } \Psi(ha_t )  (\phi\cdot \tau_\e)(h) d\mu_{H}^{\PS}(h )
\\ &=\frac{\mu^{\PS}_{H}(\phi\cdot  \tau_\e)}{|m^{\BMS}|} m^{\BMS}(\Psi) +
O(\e^{-q_\ell} e^{-\beta t}  \S_\ell(\phi) \S_\ell (\Psi)) \\
&=\frac{\mu^{\PS}_{H}(\phi)}{|m^{\BMS}|} m^{\BMS}(\Psi) +
O({A^{\PS}_\phi}A_\Psi^{\BMS} \e^{\delta-p_0})+O(\e^{-q_\ell} e^{-\beta t}  \S_\ell(\phi) \S_\ell (\Psi))
\\&=
\frac{\mu^{\PS}_{H}(\phi)}{|m^{\BMS}|} m^{\BMS}(\Psi)  +O((\e^{\delta-p_0}+ \e^{-q_\ell} e^{-\beta t})  \S_\ell(\phi)
 \S_\ell (\Psi))
.\end{align*}
On the other hand,
\begin{multline*}
\int_{Y_\Om} \Psi(ha_t ) (\phi -\phi\cdot \tau_\e )(h)d\mu_{H}^{\PS}(h )
\\ \ll \S_{\infty,1}(\Psi) \S_{\infty,1}(\phi) \mu^{\PS}_H(Y_\Om- Y_\e)\ll \S_\ell(\Psi)\S_\ell(\phi) \e^{\delta-p_0} .
\end{multline*}
Hence by combining these two estimates, and taking $\e=e^{-\beta/(\delta-p_0+q_\ell)}$
and $\beta_0:=e^{-\beta(\delta-p_0)/(\delta-p_0+q_\ell)}$,
we obtain
$$\int_{h \in \G_H\ba H} \Psi(h a_t ) \phi(h) d\mu_{H}^{\PS}(h )=
\frac{\mu^{\PS}_{H}(\phi)}{|m^{\BMS}|} m^{\BMS}(\Psi) +
O( e^{-\beta_0 t}  \S_\ell(\phi) \S_\ell (\Psi)) .$$
\end{proof}


\begin{proof}[Proof of Theorem~\ref{meqfinite}]
We will divide the integration region into three different regions: compact part, thin part, intermediate range. The compact part
is the region where we get the main term using  Theorem ~\ref{maince}.
The thin region can be controlled using Theorem \ref{yo}.
However there is an intermediate range where we need some control.
This is in some sense  the main technical difference from the
case where $\G$ is a lattice. We control the contribution from this range,
using results proved in this section in particular by relating this integral
to summation over the ``transversal"; see Lemmas~\ref{com1} and~\ref{com2}.

We use the notation from the proof of Theorem \ref{t;eq-ps-h-orbit}. In particular,
if $H$ is horospherical, set $Y_\Om=\{h\in \G\ba \G HM: ha_t\in \Om \text{ for some $t\in \br$}\}$; if $H$ is symmetric, 
let $Y_\Om$ and $Y_\e$ be as in Theorem \ref{cri} and set $p_0:=\text{Pb-corank}_H(\G)$.
Let $0<\e_1<\e_0$.
Here, we regard $Y_{\e_0}$ as a thick part, $Y_{\e_1}-Y_{\e_0}$ as an intermediate range and $\G_H\ba H-Y_{\e_1}$ as a thin part.

As above we choose $\tau_{\e_0} \in C^\infty(Y)$ which is an $H\cap M$-invariant
smooth approximation of $Y_{\e_0}$
and recall that $\mu^{\PS}_H(Y_\Om- Y_{\e_0}) \ll \e_0^{\delta-p_0}$  by (1) of Proposition \ref{mest}.

We may write
\begin{align*}
\int_{h\in \G_H\ba H} \Psi(h&a_t ) \phi(h) dh\\
&  =\int_{\G_H\ba  H} \Psi(ha_t )(\phi\cdot \tau_{\e_0}) (h) dh+ \int_{Y_{\Om}} \Psi(ha_t ) (\phi -\phi \cdot \tau_{\e_0})
(h) dh.
\end{align*}

Then by Theorem \ref{maince} with $\eta_0>0$ therein and Theorem \ref{yo}, we get the asymptotic for the thick part:
\begin{multline}\label{m11}
 e^{(n-1-\delta)t} \int_{\G_H\ba  H } \Psi(ha_t ) (\phi\cdot \tau_{\e_0})(h) dh
\\=\frac{{\mu^{\PS}_{H}(\phi)} }{|m^{\BMS}|} m^{\BR}(\Psi) +
(\e_0 ^{\delta -p_0}+ e^{-\eta_0 t} \e_0^{-q_\ell}) O(\S_\ell(\phi) \S_\ell(\Psi)) .
\end{multline}

On the other hand, by Theorem \ref{yo},
we have, for $\mathcal T_{\e_1}:=\tau_{\e_1}-\tau_{\e_0}$,
\begin{align}\label{e1}
& \int_{Y_{\Om}} \Psi(ha_t ) (\phi -\phi \cdot \tau_{\e_0})(h) dh \notag
 \\
&\ll \S_{\infty,1}(\phi)  \left(\int_{Y_{\Om}} \Psi(ha_t ) \mathcal T_{\e_1}(h) dh
+\int_{Y_{\Om}-Y_{\e_1}} \Psi(ha_t )  dh \right).
   \end{align}

Set $\overline\Psi(x):=\int_{H\cap M}\Psi(xm)dm$.
Applying Lemma~\ref{com1} for the Haar measure $d\mu_H^{\Haar}=dh$
and Lemma~\ref{com2} for the $\PS$ measure $\mu_H^{\PS}$, and
for the function $\mathcal T:=\mathcal T_{\e_1}$,
with the notation as in the proof of Theorem~\ref{tran}, we get the following estimate of the integral over the intermediate range $Y_{\e_1}-Y_{\e_0}$:
\begin{align}\label{e2}
&e^{(n-1-\delta)t}\int_{Y_\Om} \Psi(ha_t )\mathcal T_{\e_1}(h)  dh \notag \\
&\ll e^{(n-1-\delta)t}\int_{Y_\Om} \overline \Psi(ha_t )\mathcal T_{\e_1}(h)  dh \notag \\
&\ll
e^{-\delta t}\sum_{p\in P_e (t)}\overline \psi^+_{\e_1}(xp){\mathcal T}^+_{e^{-t}\e_1}(xpa_{-t}) \notag\\
&\ll \int_{\G\ba \G H} \overline \Psi^+_{\e_1}(ha_t) {\mathcal T}^+_{e^{-t}\e_1}(h) d\mu_{H}^{\PS}(h) \notag \\
&\ll \mathcal S_{\infty,1}(\overline \Psi) \mu_H^{\PS} (Y_{\e_1}-Y_{\e_0}) \ll
\mathcal S_\ell(\Psi)  \e_0^{\delta-p_0}.
\end{align}

Using Theorem \ref{yo}, we also get the following estimate of the integral over the thin part, which is the complement of $Y_{\e_1}$: 
\be\label{e3}
e^{(n-1-\delta)t} \int_{Y_\Om\setminus Y_{\e_1}} \Psi(ha_t )  dh\leq
\S_\ell(\Psi) e^{(n-1-\delta)t}\e_1^{n-1+p_0}.
\ee

Therefore by \eqref{m11}, \eqref{e1},\eqref{e2}, and \eqref{e3},
\begin{align*}
&e^{(n-1-\delta)t}\int_{h\in \G_H\ba H} \Psi(ha_t ) \phi(h) dh \\
&=\frac{{\mu^{\PS}_{H}(\phi)} }{|m^{\BMS}|} m^{\BR}(\Psi) +
O(\e_0 ^{\delta -p_0}+ e^{-\eta_0 t} \e_0^{-q_\ell} +e^{(n-1-\delta)t}\e_1^{n-1-p_0} )
 \S_\ell(\phi) \S_\ell(\Psi) . \end{align*}

 Recalling $\delta>p_0$,
 take $\e_0$ and $\e_1$ by $\e_0=e^{-\eta_0t/(\delta-p_0+q_\ell)}$ and
$\e_1^{n-1-p_0}= \e_0^{\delta -p_0} e^{(\delta-n+1)t}$. We may assume that $\e_1<\e_0$
by taking $\ell$ and hence $q_\ell$ big enough.
Finally, we obtain  the claim with $\beta:=  \eta_0(\delta-p_0)/(\delta-p_0+q_\ell )$.
\end{proof}

We can also prove an analogue of
Theorem \ref{meq} with $a_t$ replaced by $a_{-t}$, by following a similar argument step by step
but using Corollary \ref{reverse} in place of Theorem \ref{mc2}.
Consider the $H\cap M$-invariant measure $\mu_{H, -}^{\PS}$ on $\G_H\ba H$ induced by
the measure $e^{\delta\beta_{h^-}(o, \bar h)} d\nu_o(h^-)$ on $\bar H= H/(H\cap M)$:
\be \label{hmi} d\mu_{H,-}^{\PS}(hm)=e^{\delta\beta_{h^-}(o, \bar h)} d\nu_o(h^-)d_{H\cap M}(m).
\ee
\begin{thm} \label{meqminus} Suppose that
$|\mu_{H,-}^{\PS}|<\infty$. 
There exist $\beta>0$ and $\ell \ge 1$ such that
for any compact subset $\Om$ in $\G\ba G$, any $\Psi\in C^\infty(\Om)$ and any bounded $\phi\in  C^\infty(\G_H \ba H)$,
we have, as $t\to +\infty$,
\begin{multline*} e^{(n-1-\delta)t} \int_{h\in  \G_H \ba  H} \Psi(h a_{-t} )\phi(h) dh
 = \frac{{\mu^{\PS}_{H, -} (\phi)} }{|m^{\BMS}|} m^{\BR}_*(\Psi) +
O(e^{-\beta t}  \S_\ell (\Psi)\S_\ell(\phi) ) \end{multline*}
where the implied constant depends only on $\Om$.
\end{thm}

\subsection{Effective mixing of the BMS measure}\label{sec;eff-mixing}
In this subsection we prove an effective mixing for the $\BMS$ measure:

\begin{thm}\label{em} 
 There exist $\beta>0$ and $\ell \in \N$ such that for any compact subset $\Om \subset \G\ba G$,
and for any $\Psi, \Phi\in C^\infty(\Om)$,
$$\int_{\G\ba G} \Psi(ga_t) \Phi(g) dm^{\BMS}(g) =\tfrac{1}{|m^{\BMS}|} m^{\BMS}(\Psi) m^{\BMS}(\Phi) +
O(e^{-\beta t}  \S_\ell(\Psi)\S_\ell(\Phi) )$$
with the implied constant depending only on $\Om$.
\end{thm}
\begin{proof}

Using a smooth partition of unity for $\Om$, it suffices to prove the claim
for $\Phi\in C_c(xB_{\e_0})$ for $x\in \Om $, $B_{\e_0}=P_{\e_0}N_{\e_0}$ and
$\e_0>0$ smaller than the injectivity radius of $\Om$.

By Theorem \ref{loc} with $H=N$ and for each $p\in P_{\e_0}$,
\begin{align*} &\int_{xpN_{\e_0}} \Psi(xpna_t) \Phi(xpn) d\mu^{\PS}_{xpN}(xpn)\\&=
\tfrac{1}{|m^{\BMS}|} m^{\BMS}(\Psi) \mu_{xpN}^{\PS}(\Phi|_{xpN_{\e_0}}) +
e^{-\beta t} O(\S_\ell  (\Psi) \S_\ell ({\Phi}|_{xpN_{\e_0}}))
\end{align*}
for some $\beta>0$ and $\ell\in \N$.
As
$$\int_{xP_{\e_0}} \mu_{xpN}^{\PS}(\Phi|_{xpN_{\e_0}})  d\nu_{xP}(xp)=m^{\BMS}(\Phi),$$
we have
\begin{align*}
&\int_{xB_{\e_0}} \Psi(ga_t) \Phi(g) dm^{\BMS}(g)\\
&=\tfrac{1}{|m^{\BMS}|} \int_{xp\in xP_{\e_0}}\int_{xpN_{\e_0}} \Psi(xpna_t) \Phi(xpn) d\mu^{\PS}_{xpN}(xpn) d\nu_{xP}(xp)
\\&= \tfrac{1}{|m^{\BMS}|} m^{\BMS}(\Psi ) m^{\BMS}(\Phi)
+O(e^{-\beta t}) \S_\ell (\Psi) \cdot \int_{xP_{\e_0}}  \S_\ell  ({\Phi}|_{xpN_{\e_0}})  d\nu_{xP}(xp).
\end{align*}
Since
$$\int_{xP_{\e_0}}  \S_\ell  ({\Phi}|_{xpN_{\e_0}})  d\nu_{xP}(xp)\ll
 \S_\ell (\Phi)m^{\BR} (\supp \Phi) \ll_\Om  S_{\ell} (\Phi), $$
this finishes the proof.
\end{proof}




\section{Effective uniform  counting}\label{s:count}
\subsection{The case when $H$ is symmetric or horospherical} Let $G, H,A=\{a_t:t\in \br\}$, $K$ etc
be as in the section \ref{localt}. Let $\G$ be a Zariski dense and geometrically finite group with $\delta>(n-1)/2$.
Suppose that $[e]\G$ is discrete in $H\ba G$, equivalently, $\G\ba \G H$ is closed in $\G \ba G$ and that
$|\mu_H^{\PS}|<\infty$.

In this section, we will obtain effective
counting results from Theorem \ref{meq} with $\phi$ being the constant function
$1$ on $(\G\cap H)\ba H$.


\begin{dfn}[Uniform spectral Gap]\label{usg}
A family of subgroups $\{\G_i<\G :i\in I\}$ of finite index
is said to have a uniform spectral gap property if
$$\sup_{i\in I} s_0(\G_i) <\delta\quad \text{ and }\quad \sup_{i\in I}n_0(\G_i) <\infty .$$
where $s_0(\G_i)$ and $n_0(\G_i)$ are defined as in \eqref{sngg}.

The pair $(\sup_{i\in I} s_0(\G_i), \sup_{i\in I}n_0(\G_i))$ will be referred to as the uniform spectral gap data
for the family $\{\G_i:i\in I\}$.
\end{dfn}

As we need to keep track of the main term when varying $\G$ over its subgroups of finite index
for our intended applications to affine sieve, we consider
the following situation:
let $\G_0<\G$ be a subgroup of finite index with $\G_0\cap H=\G\cap H$ and fix $\gamma_0\in \G$.
Throughout this section, we assume that both $\G$ and $\G_0$ have spectral gaps; hence $\{\G,\G_0\}$ is assumed to have
 a uniform spectral gap. By Theorem \ref{nm2},
this assumption is automatic if $\delta>(n-1)/2$ for $n=2,3$ and if $\delta>n-2$ for $n\ge 4$.

\medskip

For a family $\B_T\subset H\ba G$ of compact subsets, we would like to investigate $\#
[e]\G_0\gamma_0\cap \B_T$.

Define a function
$F_T:=F_{\B_T}$ on $\G_0\ba G$ by
$$F_T(g):=\sum_{\gamma\in H\cap \G\ba \Gamma_0} \chi_{\B_T}([e]\gamma g)$$
where $\chi_{\B_T}$ denotes the characteristic function of $\B_T$.
Note that
$$F_T(\gamma_0)=\# [e]\Gamma_0\gamma_0\cap \B_T .$$

Denote by $\{\nu_x\}$ the Patterson-Sullivan density for $\G$ normalized so that
$|\nu_o|=1$. Clearly,
$\{\nu_x\}$ is the unique PS density for $\G_0$ with $|\nu_o|=1$.
Recall the Lebesgue density $\{m_x\}$ with $|m_o|=1$.

Therefore if $\tilde m^{\BMS}$, $\tilde m^{\BR}$ and $\tilde m^{\Haar}$ are
the BMS measure, the BR measure, the Haar measure on $G$,
the corresponding measures
 $ m^{\BMS}_{\G_0}$, $ m^{\BR}_{\G_0}$ and $m^{\Haar}_{\G_0}$
 on $\G_0\ba G$ are naturally induced from them.
 In particular, for each $\bullet=\BMS,\BR,\Haar$,
$|m^{\bullet}_{\G_0}|= [\G:\G_0]\cdot | m^{\bullet}|$.
Since $H\cap \G=H\cap \G_0$,
we have $|\mu_{H}^{\PS}|=|\mu_{\G_0,H}^{\PS}|$
and $|\mu_{H,-}^{\PS}|=|\mu_{\G_0,H,-}^{\PS}|$.

\subsection{Weak-convergence of counting function} \label{not}
Fix $\psi\in C_c^\infty(G)$. For $k\in K$ and $\gamma\in \G$,
define $\psi^k,\psi_\gamma \in C_c^\infty(G)$ by $\psi^k(g)=\psi(gk)$
and $\psi_\gamma(g)=\psi(\gamma g)$.
Also define  $\Psi, \Psi_\gamma
\in C_c^\infty(\G_0\ba G)$ by
$$\Psi(g):=\sum_{\gamma'\in \G_0} \psi(\gamma' g)\quad\text{and}\quad
\Psi_\gamma (g):=\sum _{\gamma'\in \G_0} \psi(\gamma \gamma' g) .$$


For a function $f$ on $K$,
define a function $\psi*_Kf$, or simply $\psi*f$, on $G$
by $$\psi*f(g)=\int_{k\in K}\psi(gk)f(k)\; dk .$$

 For a subset $\B\subset H\ba G$,
define a function $f_{\B}^{\pm}$ on $K$ by $$f^{\pm}_{\B}(k)=\int_{a_{\pm t} \in \B k^{-1}\cap [e] A^\pm}
e^{\delta t} dt .$$
We adopt the notation $\tilde m^{\BR}_+=\tilde m^{\BR}$ and
$\tilde m^{\BR}_-=\tilde m^{\BR}_*$ below.
\begin{prop}\label{sqz} There exist $\beta_1>0$ and $\ell\ge 1$ (depending only on the uniform spectral gap data of
$\G$ and $\G_0$) such that
for any $T\gg 1$,and any $\gamma\in \G$,
$\la F_{T}, \Psi_{\gamma} \ra,$ the pairing is in $\Gamma_0\backslash G,$ is given by
\begin{equation*}
  \begin{cases} \tfrac{ |\mu_{H}^{\PS}|}{[\G:\G_0] \cdot |m^{\BMS}|}
\tilde m^{\BR}(\psi *f^+_{\B_T})
+O(\max_{a_t\in \B_T} e^{(\delta-\beta_1)t} \cdot \mathcal S_\ell (\psi))&\text{if $G=HA^+K$} \\
\sum \tfrac{ |\mu_{H^{\pm}}^{\PS}|}{[\G:\G_0] \cdot |m^{\BMS}|}
\tilde m^{\BR}_{\pm} (\psi *f^\pm_{\B_T})
+O(\max_{a_t\in \B_T} e^{(\delta-\beta_1)|t|} \cdot \mathcal S_\ell (\psi))
&\text{otherwise}.\end{cases}\end{equation*}
 \end{prop}
 \begin{proof} For the Haar measure $d\tilde m^{\Haar}( g)=dg$, we may write
$dg=\rho(t) dhdtdk$ where $g=ha_tk$ and $\rho(t)=e^{(n-1)|t|} (1+O(e^{-\alpha_1 |t|}))$ for some $\alpha_1>0$ (cf. \cite{OS}).

Setting $\kappa^{\pm}(\G_0):=\frac{ |\mu_{\G_0,H,{\pm}}^{\PS}|}{|m^{\BMS}_{\G_0}|}$,
we have  $\kappa^\pm(\G_0)=  \frac{1}{[\G:\G_0]} \kappa^\pm(\G)$.
We will only prove the claim for the case $G=HA^+K$, as the other case can be
 deduced in a similar fashion,
based on Theorem \ref{meqminus}.
We apply Theorem \ref{meq} and obtain:

\begin{align*}&\la F_{T}, \Psi_\gamma \ra=
\int_{[e] a_t k\in \B_T} \left(\int_{(H\cap \G)\ba H} \Psi_\gamma(ha_tk) dh \right)\rho(t)dt dk
\\ &
= \kappa^+(\G_0) \int_{[e] a_t k\in \B_T}   e^{(\delta-n+1)t}m^{\BR}_{\G_0}(\Psi^k_\gamma)  \rho(t)  dtdk
\\ &+   \int_{[e]a_tk \in  \B_T}   e^{(\delta-n+1-\beta)t} \rho(t)  dtdk \cdot  O(\mathcal S_\ell (\Psi))
\\ &= \kappa^+(\G_0) \int_{[e] a_t k\in \B_T}   e^{\delta t}\tilde m^{\BR}(\psi^k_\gamma)    dtdk
+O(\max_{a_t\in \B_T} e^{(\delta-\beta_1)t} \cdot \mathcal S_\ell (\psi))
\end{align*}
where $\beta_1=\min \{\beta, \alpha_1\}$.

By the left $\G$-invariance of $\tilde m^{\BR}$, we have
$\tilde m^{\BR}(\psi^k_\gamma) =\tilde m^{\BR}(\psi^k) $.
Hence $$\int_{[e] a_t k\in \B_T}   e^{\delta t}\tilde m^{\BR}(\psi^k_\gamma)    dtdk =
\int_{k\in K}\int_{a_t\in \B_T k^{-1}} e^{\delta t} \tilde m^{\BR}(\psi^k)   dt dk = \tilde m^{\BR}(\psi*f^+_{\B_T}).$$
This finishes the proof.
\end{proof}

\subsection{Counting and the measure $\mathcal M_{H\ba G}$}
Define a measure $\mathcal M_{H\ba G}=\mathcal M_{H\ba G}^\Gamma$ on $H\ba G$: for $\phi\in C_c(H\ba G)$,
\be\label{mde}
\mathcal M_{H\ba G} (\phi)=\begin{cases}\tfrac{ |\mu_{H}^{\PS}|}{|m^{\BMS}|} \int_{a_tk\in A^+K} \phi(a_tk)
  e^{\delta t} dt d\nu_o(k^{-1}X_0^-)&\text{if $G=HA^+K$}
\\ \sum \tfrac{ |\mu_{H,\pm}^{\PS}|}{|m^{\BMS}|}\int_{a_{\pm t}k\in A^\pm K}  \phi(a_{\pm t}k)
e^{\delta t} dt d\nu_o (k^{-1}X_0^{\mp})
 &\text{otherwise}, \end{cases}\ee
where
 $X_0\in \T^1(\bH^n)$ is the vector fixed by $M$,.

Observe that the measure $\mathcal M_{H\ba G}$ depends on $\G$ but is independent of the normalization
of the PS-density.

\begin{thm}\label{mcla}
If $\{\B_T\subset H\ba G\}$ is effectively well-rounded with respect to $\G$ (see Def. \ref{adm}),
 then there exists $\eta_0>0$ (depending only
on a uniform spectral gap data for $\G$ and $\G_0$) such that for any $\gamma_0\in \G$
$$\# ([e]\G_0\gamma_0\cap \B_T )= \tfrac{1}{[\G:\G_0] }
\mathcal M_{H\ba G}(\B_T)  + O( \mathcal M_{H\ba G}(\B_T)^{1-\eta_0})$$
with the implied constant independent of $\G_0$ and $\gamma_0\in \G$.
\end{thm}

\begin{proof}
Let $\psi^\e\in C^\infty(G)$ be an $\e$-smooth approximation of $e$: $0\le \psi^\e \le 1$,
$\supp(\psi^\e)\subset G_\e$ and $\int \psi^\e dg=1$.
Set $\B_{T,\e}^+:=\B_T G_\e$ and $\B_{T,\e}^-:=\cap_{g\in G_\e} \B_Tg$.
Then
$$\la F_{\B_{T,\e}^-}, \Psi^\e_{\gamma_0^{-1}} \ra \le F_T(\gamma_0)\le \la F_{\B_{T,\e}^+}, \Psi^\e_{\gamma_0^{-1}} \ra .$$

Again, we will provide a proof only for the case $G=HA^+K$; the other case can be done similarly, based on
Proposition \ref{sqz}.
By Proposition \ref{sqz}, for $ \kappa^+(\G_0):=\frac{ |\mu_{\G_0,H}^{\PS}|}{|m^{\BMS}_{\G_0}|}$,
$$ \la F_{\B_{T,\e}^\pm }, \Psi^\e_{\gamma_0^{-1}} \ra =
\kappa^+(\G_0) \tilde m^{\BR}(\psi^\e *f_{\B_{T,\e}^\pm})
+O(\max_{a_t\in \B_T} e^{(\delta-\beta_1)t} \e^{-q_\ell}).$$
where $q_\ell$ is so that $\S_\ell(\psi^\e)=O(\e^{-q_\ell})$.
For $g=a_rnk'\in ANK$, define $H(g)=r$ and $\kappa(g)=k'$.

Now, using the strong wave front property for $ANK$ decomposition
\cite{GOS}, and the definition \ref{adm}, there exists $c>1$  such that
for any $g\in G_\e$ and $T\gg 1$,
$$ f_{\B_{T,c\e}^-}(\kappa(k^{-1} )) \le f_{\B_{T}}(\kappa(k^{-1} g))\le f_{{\B_{T,c\e}^+}}(\kappa(k^{-1} )).$$

We use the formula for $\tilde m^{\BR}$ (cf. \cite{OS}):
$$d\tilde m^{\BR}(ka_rn)=e^{-\delta r}dndrd\nu_o(kX_0^-)$$
and  deduce \begin{align}\label{brc}
&\kappa^+(\G) \tilde m^{\BR}(\psi^\e  *f_{{\B_{T,\e}^+}})\notag
\\ &=\kappa^+(\G) \int_{k'\in K} \int_{KAN} \psi^\e (ka_r n k')f_{{\B_{T,\e}^+}}(k') e^{-\delta r}dk' dndrd\nu_o(kX_0^-)
\notag \\ &= \kappa^+(\G) \int_{k\in K} \int_{G} \psi^\e (kg)f_{{\B_{T,\e}^+}}(\kappa(g)) e^{(-\delta +(n-1))H(g)}
dg d\nu_o(k X_0^-) \notag \\
 &= \kappa^+(\G) \int_{k\in K} \int_{G} \psi^\e (g)f_{{\B_{T,\e}^+}}(\kappa(k^{-1} g)) e^{(-\delta +(n-1))H(k^{-1} g)}
dg d\nu_o(k X_0^-)
\notag \\ &\le  (1+O(\e))\kappa^+(\G)  \int_{k\in K} \int_{G} \psi^\e (g)f_{{\B_{T,c\e}^+}}(k^{-1})
dg d\nu_o(kX_0^-)
\notag \\&= (1+O(\e))\mathcal M_{H\ba G}({\B_{T,c\e}^+}) =  (1+ O(\e^p)) \mathcal M_{H\ba G}(\B_T)
\end{align}
 since $\int \psi^\e  dg =1$ and $\kappa^+(\G) \int_{k\in K} f_{\B_T}(k^{-1})d\nu_o(kX_0^-) =
 \mathcal M_{H\ba G}(\B_T)$.

Similarly,
$$\kappa^+(\G) \tilde m^{\BR}(\psi^\e *f_{\B_{T, c\e}^-})= (1+ O(\e^p)) \mathcal M_{H\ba G}(\B_T) .$$
Since $\max_{a_t\in \B_T} e^{(\delta -\beta_1)t} \ll
  \mathcal M_{H\ba G}(\B_T)^{1-\eta}$ for some $\eta>0$,
$$\# (\G_0\gamma_0\cap \B_T )= \tfrac{1}{[\G:\G_0]} \mathcal M_{H\ba G}(\B_T) +  O(\e^p \mathcal M_{H\ba G}(\B_T))
+ O(\e^{-q_{\ell}}   \mathcal M_{H\ba G}(\B_T)^{1-\eta} ).$$
Hence by taking $\e=  \mathcal M_{H\ba G}(\B_T)^{-\eta/(p+q_\ell)}$
and $\eta_0=  -p\eta/(p+q_\ell)$, we complete the proof. \end{proof}



\subsection{Effectively well-rounded families of $H\ba G$}

\subsubsection{Sectors}
For $\om \subset K$, we consider the following sector in $H\ba G$:
$$S_T(\om):=[e]\{a_t: 0\le t\le \log T\} \om .$$

In this subsection, we show that the family of sectors $\{S_T(\om):T\gg 1\}$ is effectively well-rounded
provided $\om$ is admissible in the following sense:

\begin{dfn} \label{admissible}\rm We will call  a Borel subset $\om\subset K$ with $\nu_o(\om^{-1}X_0^-)>0$
 {\it admissible} if there exists $0<p\le 1$ such that for all small $\e>0$,
 \begin{equation}\label{ade}\nu_o((\om^{-1} K_\e  -\cap_{k\in K_\e} \om^{-1} k)X_0^-)\ll \e^p \end{equation}
with the implied constant depending only on $\om$.
\end{dfn}

\begin{lem}\label{tri}
Let  $\om \subset K$ be a Borel subset. If $\nu_o(\om^{-1}X_0^-)>0$ and $\partial(\om^{-1})\cap \Lambda(\G)=\emptyset$, then $\om$
 is admissible. 
\end{lem}
\begin{proof}
As $\partial(\om^{-1})$ and $\Lambda(\G)$ are compact subsets, we can find $\e_0>0$
such that the $\e_0$-neighborhood of  $\partial(\om^{-1})$ is disjoint from $\Lambda(\G)$.
Hence we can find $\e_1>0$ such that $\partial(\om^{-1})K_{\e_1}$ is disjoint from $\Lambda(\G)$;
so $\nu_o(\partial(\om^{-1})K_{\e_1} X_0^-)=0$.
\end{proof}

\begin{prop}\label{exam} Let $\kappa_o:=\max_{\xi\in \Lambda_p(\G)} \text{rank }(\xi)$.
If $$\delta >\max \{n-2,  \frac{n-2 +\kappa_0 }{2}\},$$
then any Borel subset $\om \subset K$ such that  $\nu_o(\om^{-1}X_0^-)>0$ and
$\partial(\om^{-1})$ is piece-wise smooth
is admissible.
\end{prop}

\begin{proof}
Let $s_\xi=\{\xi_t: t\in [0, \infty)\} $ be a geodesic ray emanating from $o$ toward $\xi$ and let
$b(\xi_t)\in \overline \bH^n$ be the Euclidean ball centered at $\xi$ whose boundary
is orthogonal to $s_\xi$ at $\xi_t$. For simplicity, we write $\nu_o(\om)=\nu_o(\om X_0^-)$ for a Borel subset $\om$ of $K$.
Then by Sullivan \cite{Sullivan1984}, there exists a $\G$-invariant collection of pairwise disjoint horoballs
$\{\mathcal H_\xi: \xi\in \Lambda_p(\G)\}$ for which the following holds: there exists a constant $c>1$ such
that for any $\xi\in \Lambda(\G)$ and for any $t>0$,
$$c^{-1} e^{-\delta t} e^{d(\xi_t, \G(o)) (k(\xi_t) -\delta)}\le
\nu_o(b(\xi_t)) \le c e^{-\delta t} e^{d(\xi_t, \G(o)) (k(\xi_t)-\delta)}$$
where $k(\xi_t)$ is the rank of $\xi'$ if $\xi_t\in \mathcal H_{\xi'}$ for some $\xi'\in \Lambda_p(\G)$
and $\delta$ otherwise.

Therefore, using $0\le d(\xi_t, \G(o)) \le t$, it follows that for any $\xi\in \Lambda(\G)$ and $t>1$,
\be \label{compare} \nu_o(b(\xi_t)) \ll \begin{cases}   e^{(-2\delta +k(\xi_t))t} &\text{ if $k(\xi_t)\ge \delta$}
    \\ e^{-\delta t} &\text{otherwise}.
   \end{cases}
\ee

By  standard computations in hyperbolic geometry,  there exists $c_0>1$ such that
$ B(\xi, c_0^{-1} e^{-t})\subset b(\xi_t) \subset B(\xi, c_0e^{-t})$
where $B(\xi, r)$ denotes the Euclidean ball in $\partial(\bH^n)$ of radius $r$.
Hence it follows from \eqref{compare} that if we set $\kappa_0:=\max_{\xi'\in \Lambda_p(\G)} \text{rank }(\xi')$, then
 for all small $\e>0$ and $\xi\in \Lambda(\G)$,
$$\nu_o(B(\xi, \e))\ll \e^\delta + \e^{2\delta -\kappa_0} .$$

Clearly, this inequality holds for all $\xi\in \partial(\bH^n)$, as the support of $\nu_o$ is equal to $\Lambda(\G)$.

Now if $\partial(\om^{-1})$ is a piece-wise smooth subset of $K$,
we can cover its $\e$-neighborhood by $O(\e^{1-d_K})$ number of $\e$-balls, where $d_K$ is the dimension of $K$.

Since for any $k\in K$, $$\nu_o(B(k, \e))\ll \e^{d_M} \cdot \nu_o(B(k(X_0^+), \e))\ll  \e^{\delta+d_M} + \e^{2\delta -\kappa_0+d_M};$$
where $X_0\in \T^1(\bH^n)$ is fixed by $M$ and $d_M$ is the dimension of $M$,
we obtain that the $\nu_o$ measure of an $\e$-neighborhood of $\partial(\om^{-1})$
is at most of order
$$\e^{\delta+d_M-d_K+1} + \e^{2\delta -\kappa_0+d_M-d_K+1} =
\e^{\delta-n+2} + \e^{2\delta -\kappa_0-n+2}.$$
Hence $\om$ is admissible if $\delta$ is bigger than both $(n-2)$ and $\frac{n-2+\kappa_0}{2}$.
\end{proof}

\begin{cor}
 If $\delta>n-2$ and $\text{rank }(\xi)<\delta$ for all $\xi\in \Lambda_p(\G)$, then
  any Borel subset $\om \subset K$ such that  $\nu_o(\om^{-1}X_0^-)>0$ and
$\partial(\om^{-1})$ is piece-wise smooth
is admissible.
\end{cor}

The following strong wave front property of $HAK$ decomposition is a crucial ingredient
 in proving an effective well-roundedness of a given family:
\begin{lem}[Strong wave front property] \cite[Theorem 4.1]{GOS}\label{str}
 There exists $c>1$ and $\e_0>0$ such that for any $0<\e<\e_0$ and for any $g= ha_tk\in HA^+K$ with $t>1$,
$$gG_\e\subset (hH_{c\e})\;( a_t A_{c\e})\;( k K_{c\e})$$
where $H_{c\e}=H\cap G_{c\e}$ and $A_{c\e}$ and $K_{c\e}$ are defined similarly. \end{lem}

\begin{prop}\label{sectore}
  Let $\om \subset K$ be an admissible subset. Then the family $\{S_T(\om): T\gg 1\}$ is effectively well-rounded and
$$\mathcal M_{H\ba G}(S_T(\om))=
\tfrac{ |\mu_{H}^{\PS}| \cdot \nu_o(\om^{-1})X_0^-}{\delta\cdot
 |m^{\BMS}|} (T^\delta -1).$$
\end{prop}
\begin{proof}
We compute
\begin{align*}
\mathcal M_{H\ba G}(S_T(\om))&=
\tfrac{ |\mu_{H}^{\PS}|}{|m^{\BMS}|} \int_{t=0}^{\log T}
  e^{\delta t} dt \int_{k\in \om} d\nu_o(k^{-1}X_0^-)\\ &=
\tfrac{ |\mu_{H}^{\PS}| \cdot \nu_o(\om^{-1}X_0^-)}{\delta\cdot
 |m^{\BMS}|} (T^\delta -1).
\end{align*}

By Lemma \ref{str},
there exists $c\ge 1$ such that for all  $T>1$ and $\e>0$
$$S_T(\om) G_\e\subset [e]\{a_t: \log (1-c\e)  \le t\le \log (1+c\e) T\} \om_{c\e}^+$$
where $\om_{c_\e}^+= \om K_{c\e}$ and $K_{c\e}$ is a $c\e$-neighborhood of $e$ in $K$.
Hence with $p>0$ given in \eqref{ade},
\begin{align*}
\mathcal M_{H\ba G}(S_T(\om)G_\e)& \ll
\tfrac{ |\mu_{H}^{\PS}| \cdot \nu_o((\om_{c\e}^+)^{-1} X_0^-)}{\delta\cdot
 |m^{\BMS}|} (1+c\e)^\delta T^\delta  \\&\ll
\tfrac{ |\mu_{H}^{\PS}| \cdot (1+ O(\e^p)) \nu_o(\om^{-1}X_0^-)}{\delta\cdot
 |m^{\BMS}|} (1+c\e)^\delta T^\delta
\\ & =  (1+O(\e^{p}))  \mathcal M_{H\ba G}(S_T(\om)) .
\end{align*}
Similarly, we can show that
$$\mathcal M_{H\ba G}(\cap_{g\in G_\e} S_T(\om)g)=(1+O(\e^{p}))  \mathcal M_{H\ba G}(S_T(\om)).$$
Hence the family $\{S_T(\om)\}$
is an effectively well-rounded family for $\G$.
\end{proof}

Therefore we deduce from Theorem \ref{mcla}:
\begin{cor}\label{sector}
 Let $\om \subset K$ be an admissible subset. Then
there exists $\eta_0>0$ (depending only
on a uniform spectral gap data for $\G$ and $\G_0$) such that for any $\gamma_0\in \G$
$$ \# ([e]\G_0\gamma_0\cap S_T(\om) )= \tfrac{ |\mu_{H}^{\PS}| \cdot \nu_o(\om^{-1}X_0^-)}{[\G:\G_0] \cdot |m^{\BMS}|
\cdot \delta}
{T}^\delta  + O( T^{\delta -\eta_0}).$$
\end{cor}

\subsubsection{Counting in norm-balls}
Let $V$ be a finite dimensional vector space on which $G$ acts linearly from  the right and let $w_0\in V$.
We assume that $w_0\G$ is discrete and that $H:=G_{w_0}$
is either a symmetric subgroup or a horospherical subgroup.
We let $A=\{a_t\}, K, M$ be
as in section \ref{th}.
Let $\lambda\in \mathbb N$ be the log of the largest eigenvalue of $a_1$ on the $\br$-span of $w_0G$,
and set $$w_0^{\pm \lam}:=\lim_{t\to\infty} e^{-\lam t} w_0a_{\pm t} .$$
Fixing a norm $\|\cdot \|$ on $V$, let $B_T:=\{v\in w_0G: \|v\|<T\}$.

\begin{prop}\label{normw}
For any admissible $\om \subset K$, the family $\{B_T\cap w_0A^\pm \om\}$ is
effectively well-rounded. In particular, $\{B_T\}$
is effectively well-rounded.

We also compute that for some $0<\eta<\delta/\lam$,
$$\mathcal M_{H\ba G}(B_T\cap w_0A^\pm \om) =
 \tfrac{|\mu_{H, \pm} ^{\PS}|}{\delta \cdot |m^{\BMS}|} \cdot
\int_{\om} \|w_0^{ \pm \lam}  k\|^{-\delta/\lam}d\nu_o^{\pm} (k^{-1}X_0^-) \cdot   {T}^{\delta/\lam} + O( T^{\eta}).$$
\end{prop}
\begin{proof}
By the definition of $\lambda$ and $w_0^\lam$, it follows that
$w_0a_t k=e^{\lam t}w_0^\lam k +O(e^{\lam_1 t})$
for some $\lam_1<\lam$. Noting that $\|w_0a_tk\|\le T$ implies that $e^{\lam t}=O(T)$ and $e^{\lam_1 t}=O(T^{\lam_1/\lam})$,
we have
 \begin{align*} &\M_{H\ba G}(B_T\cap w_0A^+\om)  \\&
=\tfrac{|\mu_H^{\PS}|}{|m^{\BMS}| }\int_{k\in \om} \int_{\|w_0a_tk\|< T} e^{\delta t} dt d\nu_o(k^{-1}X_0^-)\\
&=\tfrac{|\mu_H^{\PS}|}{|m^{\BMS}| } \int_{k\in \om} \int_{e^{\lam t}\le {\|w_0^\lam k\|}^{-1} T +O(T^{\lam_1/\lam}) } e^{\delta t} dt d\nu_o(k^{-1}X_0^-)\\
 &= \tfrac{|\mu_H^{\PS}|}{|m^{\BMS}|\cdot \delta } T^{\delta/\lam}  \int_{k\in \om} \|w_0^\lam k\|^{-\delta/\lam} d\nu_o(k^{-1}X_0^-)+ O( T^{\eta})
\end{align*}
for some $\eta<\delta/\lam$.
The claim about $\M_{H\ba G}(B_T\cap w_0A^-\om)$ can be proven similarly.
To show the effective well-roundedness, we first note that
by Lemma \ref{str}, for some $c>1$,
we have $$(B_T\cap w_0A^+\om)G_\e \subset B_{(1+c\e)T}\cap w_0A^+ \om_{c\e}^+.$$

Therefore, using the admissibility of $\om$, and with $p$ given in \eqref{ade}, we deduce
\begin{align*}
&\M_{H\ba G}((B_T\cap w_0A^+\om)G_\e -(B_T\cap w_0A^+\om))
\\ & \ll \int_{k\in \om_{c\e}^+-\om}  \int_{\|w_0a_tk\|<(1+c\e)T} e^{\delta t} dt d\nu_o(k^{-1}X_0^-)
+  \int_{k\in \om_{c\e}^+} \int_{T\le \|w_0a_tk\|<(1+c\e)T} e^{\delta t} dt d\nu_o(k^{-1}X_0^-)
\\ & \ll \e^p \cdot  T^{\delta/\lam} + ((1+c\e )T)^{\delta/\lam} -T^{\delta/\lam}) \\ &
\ll \e^p T^{\delta/\lam}\ll \e^p \M_{H\ba G}(B_T\cap w_0A^+\om).
 \end{align*}
Similarly we can show that
$$\M_{H\ba G}((B_T\cap w_0A^+\om)-\cap_{g\in G_\e} (B_T\cap w_0A^+\om)g) \ll \e^p \M_{H\ba G}(B_T\cap w_0A^+\om).$$
This finishes the proof for the effective well-roundedness of $\{B_T\cap w_0A^+\om\}$.
The claims about $\{B_T\cap w_0A^-\om\}$ can be shown in a similar fashion.
 \end{proof}

Put
\begin{equation*}
\Xi_{w_0} (\G):=\begin{cases}
\tfrac{|\mu_{H} ^{\PS}|}{\delta \cdot |m^{\BMS}|} \cdot
\int_{K} \|w_0^{ \lam}  k\|^{-\delta/\lam}d\nu_o(k^{-1}X_0^-)&\text{if $G=HA^+K$} \\
\sum \tfrac{|\mu_{H^{\pm}} ^{\PS}|}{\delta \cdot |m^{\BMS}|} \cdot
\int_{K} \|w_0^{\pm \lam}  k\|^{-\delta/\lam}d\nu_o (k^{-1}X_0^{\mp})&\text{otherwise} .\end{cases}\end{equation*}

 We deduce the following from
Proposition \ref{normw}
and Theorem \ref{mcla}:

\begin{cor}\label{vsector}
\begin{enumerate}
\item For any admissible $\om \subset K$, there exists $\eta_0>0$ such that for any $\gamma_0\in \G$,
 \begin{multline*}
\# \{v\in w_0 \G_0\gamma_0\cap w_0A^+\om: \|v\|<T\} \\ = \tfrac{|\mu_{H} ^{\PS}|}{\delta\cdot  {[\G:\G_0]}  \cdot |m^{\BMS}|} \cdot
\int_{\om} \|w_0^{ \lam}  k\|^{-\delta/\lam}d\nu_o(k^{-1}X_0^-)  {T}^{\delta/\lambda} + O( T^{\delta/\lambda -\eta_0}) .
 \end{multline*}

 \item There exists $\eta_0>0$ (depending only
on a uniform spectral gap data for $\G$ and $\G_0$) such that for any $\gamma_0\in \G$,
$$\# \{v\in w_0 \G_0\gamma_0: \|v\|<T\}= \tfrac{1 }{[\G:\G_0]}  \Xi_{w_0}
 (\G)  {T}^{\delta/\lambda} + O( T^{\delta/\lambda -\eta_0}) .$$
\end{enumerate}
\end{cor}


\subsection{The case when $H$ is trivial}
In this subsection,
we will prove  the following theorem
directly from the asymptotic of the matrix coefficient functions
in Theorem \ref{mc2}.

Recall from the introduction the following Borel measure $\mathcal M_{G}=\mathcal M_G^\G$ on $G$: for $\psi\in C_c(G)$,
\begin{equation*}\label{mdegg}
\mathcal M_{G} (\psi)=\tfrac{ 1}{|m^{\BMS}|} \int_{k_1a_tk_2\in K A^+K} \psi(k_1a_tk_2)
  e^{\delta t} d\nu_o(k_1X_0^+)dt d\nu_o(k_2^{-1}X_0^-).\end{equation*}

\begin{thm}\label{mcla3} Let $\G_0<\G$ be a subgroup of finite index.
If $\{\B_T\subset G\}$ is effectively well-rounded with respect to $\G$ (see Def. \ref{adm}),
 then there exists $\eta_0>0$ (depending only
on a uniform spectral gap data for $\G$ and $\G_0$) such that for any $\gamma_0\in \G$
$$\# (\G_0\gamma_0\cap \B_T )= \tfrac{1}{[\G:\G_0] }
\mathcal M_{G}(\B_T)  + O( \mathcal M_{ G}(\B_T)^{1-\eta_0})$$
with the implied constant independent of $\G_0$ and $\gamma_0\in \G$.
\end{thm}

Consider the following function on $\G_0\ba G\times \G_0\ba G$: for a compact subset $\B\subset G$,
$$F_{\B}(g,h):=\sum_{\gamma\in \G_0}\chi_{\B}(g^{-1}\gamma h) $$
where $\chi_{\B}$ is the characteristic function of $\B$. We set $F_T:=F_{\B_T}$ for simplicity.
Observe that $F_T(e,\gamma_0 )=\# (\G_0 \gamma_0\cap \B_T)$.
Let $\B_{T,\e}^{\pm}$ be as in the definition \ref{adm} and let $\phi^\e\in C^\infty(G)$ and $\Phi^\e\in C^\infty(\G_0\ba G)$ be
as in the proof of Theorem \ref{mcla}.
We then have
$$\la F_{\B_{T,\e}^-}, \Phi^\e\otimes \Phi^\e_{\gamma_0^{-1}}  \ra \le
F_T(e, \gamma_0)\le  \la F_{\B_{T,\e}^+}, \Phi^\e\otimes \Phi^\e_{\gamma_0^{-1}}  \ra .$$

Note that for $\Psi_1,\Psi_2 \in C_c(\G_0\ba G)$
$$ \la F_T, \Psi_1\otimes \Psi_2\ra_{\G_0\ba G\times \G_0\ba G}
=\int_{g\in \B_T} \la \Psi_1, g.\Psi_2\ra_{L^2(\G_0\ba G)} \;  dm^{\Haar}(g) .$$
For a Borel subset $\B$ of $G$, consider a function $f_{\B}$ on $K\times K$ given by
 $$f_{\B}(k_1, k_2)=\int_{a_t\in k_1^{-1}\B k_2^{-1}\cap A^+} e^{\delta t} dt ,$$
and
define a function on $G\times G$ by
$$\left((\psi^\e\otimes \psi^\e)*f_\B \right)(g,h)=\int_{K\times K} \psi^\e(gk_1^{-1}) \psi^\e(hk_2) f_\B(k_1, k_2) dk_1 dk_2.$$

We deduce by applying Theorem \ref{harmixing} and using the left $\G$-invariance of the measures
$\tilde m^{\BR}$ and $\tilde m^{\BR}_*$ that for some $\eta', \eta>0$,
\begin{align*}
 &\la F_{\B}, \Psi^\e\otimes \Psi^\e_{\gamma_0^{-1}}  \ra_{\G_0\ba G\times \G_0\ba G}
 \\ &=\int_{x\in \B}\int_{\G_0\ba G} \Psi^\e(g) \Psi^\e_{\gamma_0^{-1}}(gx)dm_{\G_0}^{\Haar}(g) dx\\
 &=\int_{k_1a_tk_2\in \B}\left(\int_{\G_0\ba G} \Psi^\e(gk_1^{-1}) \Psi^\e_{\gamma_0^{-1}}(ga_tk_2) dm_{\G_0}^{\Haar}(g)\right) e^{(n-1)t}(1+O(e^{-\eta' t}))  dtdk_1 dk_2 \\
 &=\frac{1}{|m_{\G_0}^{\BMS}|} \int_{k_1a_tk_2\in \B} e^{\delta t} (1+O(e^{-\eta  t})) m^{\BR}_{\G_0}(k_2 \Psi_{\gamma_0^{-1}}^\e)
 m^{\BR}_{*,\G_0}(k_1^{-1}\Psi^\e) dt dk_1dk_2 
 \\ &=\frac{1}{|m_{\G_0}^{\BMS}|} \int_{k_1a_tk_2\in \B} e^{\delta t} (1+O(e^{-\eta t})) \tilde m^{\BR}(k_2 \psi^\e) \tilde m^{\BR}_{*}(k_1^{-1}\psi^\e)dt dk_1dk_2 .
\end{align*}
Therefore
\begin{align}\label{when}
&\la F_{\B_{T,\e}^\pm}, \Phi^\e\otimes \Phi^\e_{\gamma_0^{-1}}  \ra_{\G_0\ba G\times \G_0\ba G}
\\
&=  \tfrac{1}{|m^{\BMS}_{\G_0}|} ( \tilde m_*^{\BR}\otimes \tilde m^{\BR}) ((\psi^\e\otimes \psi^\e)*f_{\B_{T,\e}^\pm})
+ O(\max_{a_t\in \B_T} e^{(\delta -\eta)t}\e^{q_\ell})\notag.
\end{align}

Recall
$$dm^{\BR}(ka_r n^+)=e^{-\delta r} dn^+drd\nu_o(kX_0^-)\quad\text{ for $ka_rn^+\in KAN^+$};$$
$$dm_*^{\BR}(ka_r n^-)=e^{\delta r} dn^-drd\nu_o(kX_0^+)\quad\text{for $ka_rn^-\in KAN^-$}$$
and $dg=d\tilde m^{\Haar}(a_rn^{\pm} k)= drdn^{\pm}dk$.

For $x\in G$,
let $\kappa^{\pm}(x)$ denote the $K$-component of $x$ in $AN^{\pm}K$ decomposition and let $H^{\pm}(x)$ be uniquely given by the requirement $x\in e^{H^\pm(x)} N^{\pm} K$.

We obtain
\begin{align*}&( \tilde m_*^{\BR}\otimes \tilde m^{\BR}) ((\psi^\e\otimes \psi^\e)*f_{\B})=
 \\ &\int_{K\times K} \int_{G\times G}\psi^\e (g_1k_1^{-1}) \psi^\e(h_1k_2) f_{\B}(k_1, k_2) d\tilde m_*^{\BR}(g_1)d\tilde m^{\BR}(h_1)  dk_1 dk_2 =
 \\&\int_{K\times K} \int_{G\times G}\psi^\e (kg) \psi^\e(k_0h)f_{\B}(\kappa^-(g)^{-1},\kappa^+(h)) e^{(\delta-n+1)(H^-(g)-H^+(h))}
\\&  dgdh  d\nu_o(k_0X_0^-) d\nu_o(kX_0^+) =
\\ &
\int_{K\times K} \int_{G\times G}\psi^\e (g) \psi^\e(h)f_{\B}(\kappa^-(gk^{-1})^{-1},\kappa^+(hk_0^{-1})) e^{(\delta-n+1)(H^-(gk^{-1})-H^+(hk_0^{-1}))}
\\&  dgdh  d\nu_o(k_0X_0^-) d\nu_o(kX_0^+); 
\end{align*} first replacing
$k_1$ with $k_1^{-1}$,  substituting $g_1=ka_rn\in KAN^+$ and $h_1=k_0a_{r_0}n_0\in KAN^-$ and again substituting $a_rnk_1=g$ and $a_{r_0}n_0k_2=h$.

Therefore, using the strong wave front property for $AN^{\pm}K$ decompositions \cite{GOS} and the assumption that $\int \psi^\e dg=1$, we 
 have, for some $p>0$,
\begin{align*} &( \tilde m_*^{\BR}\otimes \tilde m^{\BR}) ((\psi^\e\otimes \psi^\e)*f_{\B_{T,\e}^\pm})
\\ &=(1+O(\e^p)) \int_{K\times K} f_{\B_T}(k,k_0^{-1})  d\nu_o(kX_0^+)  d\nu_o(k_0X_0^-)
\\ &=(1+O(\e^p)) \int_{ka_tk_0^{-1}\in \B_T} e^{\delta t} d\nu_o(kX_0^+)  d\nu_o(k_0X_0^-) \\
&=(1+O(\e^p)) \int_{ka_tk_0\in \B_T} e^{\delta t} d\nu_o(kX_0^+)  d\nu_o(k_0^{-1} X_0^-) 
\\&=
(1+O(\e^p)) \M_G(\B_T) \end{align*}
with $\M_G$ defined as in Definition \ref{group}.
Since $|m^{\BMS}_{\G_0}|=|m^{\BMS}_{\G}|\cdot [\G:\G_0]$, putting the above together, we get
$$F_T(e,\gamma_0)= \frac{1}{[\G:\G_0]} \M_G(\B_T) +  O(\M_G(\B_T)^{1-\eta_0})$$
for some $\eta_0>0$ depending only on a uniform spectral gap
data of $\G$ and $\G_0$. This proves Theorem \ref{mcla3}.


\begin{cor}\label{bik}
Let $\om_1, \om_2\subset K$ be Borel subsets in $K$ such that $\om_1^{-1}$ and $\om_2$ are admissible
in the sense of \eqref{admissible}.
Set $S_T(\om_1, \om_2):=\om_1\{a_t: 0<t<\log T\} \om_2$. Then
the family $\{S_T(\om_1, \om_2):T\gg 1\}$
is effectively well-rounded with respect to $\G$,
 and for some $\eta_0>0$,
$$\# (\G_0\gamma_0\cap S_T(\om_1, \om_2) )=
 \frac{\nu_o(\om_1 X_0^+)\cdot \nu_o (\om_2^{-1} X_0^-)}{\delta \cdot |m^{\BMS}|\cdot [\G:\G_0]}T^\delta
  + O( T^{\delta-\eta_0})$$
with the implied constant independent of $\G_0$ and $\gamma_0\in \G$.
\end{cor}
Using Proposition \ref{str} for $H=K$, we can prove the effective well-roundedness of
$\{S_T(\om_1, \om_2):T\gg 1\}$ with respect to $\G$ in a similar fashion to the proof of
Proposition \ref{sectore}. Hence Corollary \ref{bik} follows from Theorem \ref{mcla3};
we refer to Lemma \ref{tri} and Proposition \ref{exam} for admissible subsets of $K$.






\subsection{Counting in bisectors of $HA^+K$ coordinates}
We state a counting result for bisectors in $HA^+K$ coordinates.

Let ${\tau}_1\in C^\infty_c(H)$ with its support being injective to $\G\ba G$ and $\tau_2\in C^\infty(K)$,
and define $\xi_T\in C^\infty(G)$ as follows: for $g=hak\in HA^+K$,
$$\xi_T(g)= \chi_{A_T^+}(a) \cdot \int_{H\cap M} \tau_1(hm)\tau_2(m^{-1} k) dm  $$
where $\chi_{A_T^+}$ denotes the characteristic function of $A_T^+=\{a_t: 0<t<\log T\}$ for $T>1$.
Since if $hak=h'ak'$, then $h=h'm$ and $k=m^{-1}k'$ for some $m\in H\cap M$, the above function is well-defined.

\begin{thm}\label{bisec2} Let $\G_0<\G$ be a subgroup of finite index.
There exist $\eta_0>0$ (depending only
on a uniform spectral gap data  for $\G$ and $\G_0$)  and $\ell\in \N$ such that for any $\gamma_0\in \G$,
$$\sum_{\gamma\in  \G_0} \xi_T(\gamma \gamma_0) = \frac{\tilde \mu_{H}^{\PS}(\tau_1)\cdot \nu_o^* (\tau_2)}{\delta \cdot |m^{\BMS}|\cdot [\G:\G_0]}T^\delta + O( T^{\delta -\eta_0}
 \S_\ell(\tau_1) \S_\ell(\tau_2))$$
where $\nu_o^*(\tau_2):=\int_{k}\tau_2(k)d\nu_o(k^{-1}X_0^-)$.
\end{thm}
\begin{proof}
Now define a function $F_T$ on $\G_0\ba G$ by
$$F_T(g)=\sum_{\gamma\in \G_0 } \xi_T(\gamma g) .$$

For any $\psi\in C_c^\infty(G)$ and $\Psi\in C_c^\infty(\G_0\ba G)$
given by $\Psi (g)=\sum_{\g\in \G_0}\psi(\gamma g)$, we have:
\begin{align*}\la F_T, \Psi\ra_{\G_0\ba G} &=
 \int_{k\in K}\tau_2(k) \int_{a_t\in A_T^+} \left( \int_{h\in H}\tau_1(h)\Psi (ha_t k) dh \right) \rho(t)dhdkdt .
\end{align*}
As $\Psi\in C(\G_0\ba G)$, $\supp(\tau_1)$ injects to $\G_0\ba G$ and $H\cap \G=H\cap \G_0$,
we have $\mu_{\G_0, H}^{\PS} (\tau_1)=\tilde \mu_{H}^{\PS}(\tau_1)$ and
 $\int_{h\in H}\tau_1(h)\Psi (ha_t k) dh =\int_{h\in \G_0\ba \G_0 H}\tau_1(h)\Psi (ha_t k) dh $.
Therefore, by applying Theorem \ref{efn} to the inner integral,
we obtain $\eta>0$ and $\ell\in \N$ such that
\begin{align}\label{lft} &\la F_T, \Psi\ra_{\G_0\ba G} \notag \\  &=
\frac{\tilde \mu_{H}^{\PS}(\tau_1)}{ |m^{\BMS}_{\G_0}|} \int_{k\in K} \int_{a_t\in A_T^+} \tau_2(k)  m^{\BR}_{\G_0}(\Psi_k)  e^{\delta t} dkdt
+O(\S_\ell(\tau_1)\S_\ell(\psi) T^{\delta -\eta})\notag \\&=
\frac{\tilde \mu_{H}^{\PS}(\tau_1)\cdot \tilde m^{\BR} (\psi*\tau_2)}{\delta \cdot |m^{\BMS}|\cdot [\G:\G_0]}
\cdot  T^\delta +O( \S_\ell(\tau_1) \S_\ell(\tau_2)\S_\ell(\psi)   T^{\delta -\eta}).
\end{align}

Let $\tau_i^{\e, \pm}$ be $\e$-approximations of $\tau_i$;
$\tau_i^{\e, \pm}(x)$ are respectively the supremum and the infimum of $\tau_i$ in the $\e$-neighborhood of $x$.
Then for a suitable $\ell\ge 1$,
 $\tilde \mu_H^{\PS}(\tau_1^{\e,+} -\tau_1^{\e, -})=O(\e\cdot  \S_\ell(\tau_1))$,
and $\nu_o((\tau_2^{\e,+} -\tau_2^{\e, -})X_0^-)=O(\e \cdot \S_\ell(\tau_2))$.

Let $F_T^{\e, \pm}$ be a function on $\G\ba G$ defined similarly as
$F_T$, with respect to
  $\xi_{T}^{\e,\pm} (hak)=\chi_{A_{(1\pm \e) T}^+}(a) \cdot \int_{H\cap M} \tau_1^{\e,+}(hm)\tau_2^{\e,+}(m^{-1} k) dm  $.

As before,
let $\psi^\e\in C^\infty(G)$ be an $\e$ smooth approximation of $e$: $0\le \psi^\e \le 1$,
$\supp(\psi^\e)\subset G_\e$ and $\int \psi^\e dg=1$. Let $\Psi_{\gamma_0^{-1}}^\e$ be
defined as in the subsection \ref{not} with respect to $\psi^\e$.
Lemma \ref{str} implies
that  there exists $c>0$ such that for all $g\in G_{c\e}$,
$$F_T^{\e, -}(\gamma_0 g)\le F_T(\gamma_0)\le F_T^{\e, +}(\gamma_0 g)$$
and hence
\be \label{lft2} \la F_T^{\e, -},\Psi_{\gamma_0^{-1}}
^\e\ra \le F_T(\gamma_0)\le \la F_T^{\e, +},\Psi_{\gamma_0^{-1}}^\e\ra .\ee
By a similar computation as in the proof of Theorem \ref{mcla} (cf. \cite[proof of Prop. 7.5]{OS}),
we have $ \tilde m^{\BR} (\psi^\e *\tau_2)=\nu_o^*(\tau_2)+O(\e) \S_\ell(\tau_2).$

Therefore, for $q_\ell$ given by $\S_\ell(\psi^\e)=O(\e^{-q_\ell})$,
 we deduce from \eqref{lft} and \eqref{lft2}
that, using the left $\G$-invariance of the measure $\tilde m^{\BR}$,
\begin{align*} &\delta \cdot |m^{\BMS}|\cdot [\G:\G_0] \cdot F_T(\gamma_0)\\
 & = \tilde \mu_{H}^{\PS}(\tau_1)\cdot
 \tilde m^{\BR} (\psi^\e *\tau_2)\cdot  T^\delta +O(\S_\ell(\tau_1) S_\ell(\tau_2)\S_\ell(\psi^\e)  T^{\delta -\eta})
\\ &= \tilde \mu_{H}^{\PS}(\tau_1) \nu_o^* (\tau_2)T^\delta + O(\e T^\delta  +  \e^{-q_\ell} T^{\delta -\eta}) \S_\ell(\tau_1) \S_\ell(\tau_2)
\\&= \tilde \mu_{H}^{\PS}(\tau_1) \nu_o^* (\tau_2)T^\delta + O( T^{\delta -\eta_0})
 \S_\ell(\tau_1) \S_\ell(\tau_2)\end{align*}
for some $\eta_0>0$, by taking $\e=T^{-\eta/(1+q_\ell)}$.
\end{proof}

Corollary \ref{bik} as well as its analogues in the $HAK$ decomposition
can  be deduced easily from Theorem \ref{bisec2} by approximation admissible sets by smooth functions.
\renewcommand{\O}{\mathcal O}
\section{Affine sieve}\label{sieve}
In this final section, we prove Theorems \ref{affine_u} and \ref{affine_l}.
We begin by recalling the combinatorial sieve (see \cite[Theorem 7.4]{HR}).

Let $\cA=\{a_n\}$ be a sequence of non-negative numbers and let $B$ be a finite set of primes.
For $z >1$,
let $P$ be the product of primes $P=\prod_{p\notin B, p<z}p$.
We set $$S(\cA,P):=\sum_{(n,P)=1} a_n .$$
To estimate $S(\cA,P)$, we need to understand how $\cA$ is distributed along arithmetic progressions.
For $d$ square-free, define
$$
\cA_d:=\{a_n\in\cA:n\equiv0(d)\}
$$
and set
$
|\cA_d|:=\sum_{n\equiv0(d)}a_n.
$

\newcommand{\X}{\mathcal X}\renewcommand{\deg}{\op{deg}}

We will use the following combinatorial sieve:
\begin{thm}\label{sie} \begin{itemize}
\item[$(A_1)$] For $d$ square-free with no factors in $B$,
suppose that $$|\cA_d|=g(d) \X + r_d(\cA)$$
where $g$ is a function on square-free integers with $0\le g(p)<1$,
 $g$ is multiplicative outside $B$, i.e.,
$g(d_1d_2)=g(d_1)g(d_2)$ if $d_1$ and $d_2$ are square-free integers  with $(d_1, d_2)=1$ and $(d_1 d_2, B)=1$,
and for some $c_1>0$, $g(p)<1-1/c_1$ for all prime $p\notin B$.
\item[$(A_2)$]  $\cA$ has level distribution $D(\X)$, in the sense that for some $\e>0$ and $C_\e>0$,
$$\sum_{d<D}|r_d(\cA)|\le  C_\e \X^{1-\e}.$$
\item[$(A_3)$]  $\cA$ has sieve dimension $r$ in the sense that there exists $c_2>0$ such that for all $2\le w\le z$,
$$-c_2\le \sum_{(p,B)=1, w\le p\le z} g(p)\log p -r \log \frac{z}{w} \le c_2 .$$
Then for $s>9r$, $z=D^{1/s}$ and $\X$ large enough,
$$S(\cA,P)\asymp \frac{\X}{(\log \X)^r} .$$
           \end{itemize}
\end{thm}

Let $\bf{G}$, $G$ $V=\c^m $, $\G$, $w_0\in V(\z)$, etc., be as in Theorem \ref{affine_u}.
We consider the spin cover $\tilde {\bf G}\to {\bf G}$. Noting that the image of
 $\tilde{\bf G}(\br)$ is precisely $G={\bf G}(\br)^\circ$, we replace $\G$ by its preimage under the spin cover.
This does not affect the orbit $w_0\G$ and all our counting statements hold equally.
Set $W:=w_0{ \bf G}$ (resp. $w_0{ \bf G}\cup \{0\} $) if $w_0{\bf G}$  (resp. $w_0{\bf G}\cup \{0\}$) is Zariski closed,

Let $F\in \q[W]$ be an integer-valued polynomial on $w_0\G$  and let
$F=F_1\cdots F_r$ where $F_i\in \q[W]$ are all irreducible also in $\c[W]$ and integral on the orbit $w_0\G$.
We may assume without loss of generality that $\gcd \{F({x}):{ x}\in w_0\G\} =1$, by replacing $F$ by
$m^{-1}F$ for $m:=\gcd \{F({x}):{ x}\in w_0\G\} $.

Let $\{\B_T\subset w_0G\}$ be an effectively well-rounded family of subsets with respect to $\G$.
Set $\mathcal O:=w_0\G$. For $n\in \N$, $d\in \N$, and $T>1$, we also set
$$ a_n(T):=\#\{x\in  \mathcal O\cap \B_T: F(x)=n\};$$
$$ \G_{w_0}(d):=\{\gamma\in \G: w_0\gamma \equiv w_0 \; (d)\},$$
$$|\cA(T)|:=\sum_n a_n(T)=\# \O\cap \B_T;$$
$$ |\cA_d(T)|:=\sum_{n\equiv 0 (d)} a_n(T)= \#\{x\in  \mathcal O\cap \B_T: F(x)\equiv 0 \;(d)\}.$$

Let $\G_d:=\{\gamma\in \G: \gamma \equiv e\;\; (d)\}$.

\begin{thm}\label{usp} If $\delta >n-2$, then there exists
a finite set $S$ of primes such that
the family $\{\G_d:  \text{$d$ is square-free with no factors in $S$}\}$ has a uniform spectral gap.
\end{thm}
\begin{proof} As $\delta>(n-1)/2$,
by \cite{SV} and by the transfer property obtained in
 \cite{BGS}, there exists
a finite set $S$ of primes such that
 the family $L^2(\G_d\ba \bH^n)$ has a uniform spectral gap where $d$ runs over all square-free integers with
no prime factors in $S$, that is, there exists $s_1<\delta$ such that
$L^2(\Gamma_d\ba G)$ does not contain
 a spherical complementary series representation of parameter $ s_1< s<\delta$.
 By Theorem \ref{nm2} and the classification of $\hat G$ \cite{KS},
$L^2(\Gamma_d \ba G)$ does not contain
 a non-spherical complementary series representation of parameter $s>(n-2)$.

It follows that $L^2(\Gamma_d \ba G)=\mathcal H_\delta \oplus \mathcal W_d$
 where $\mathcal H_\delta=U(1, (\delta-n+1)\alpha)$ is the spherical complementary series representation of
 parameter $\delta$; hence $n_0(\G_d)=1$
 and $\mathcal W_d$ does not weakly contain any complementary series representation of parameter $\max(n-2, s_1) <s<\delta$.
So $\sup s_0(\G_d)\le \max(n-2, s_1)<\delta$ and $\sup n_0(\G_d)=1$ where $d$ runs over all square-free integers with
no prime factors in $S$.
\end{proof}

Denote by $\G(d)$ the image of $\G$ under the reduction map $\tilde {\bf G} \to \tilde {\bf G}(\z/d\z)$ and
set $\O_d$ to be the orbit of $w_0$ in $(\z/d\z)^m$ under $\G(d)$; so $\# \mathcal O_d=[\G:\G_{w_0}(d)]$.
We also set
$$\O_F(d):=\{x\in \O_d: F(x)\equiv 0 \;\;(d)\}.$$

\begin{cor}\label{ns1}
Put $\M_{w_0G}(\B_T)=\X$. 
 Suppose that for some finite set $S$ of primes,
the family $\{\G_d:  \text{$d$ is square-free with no factors in $S$}\}$ has a uniform spectral gap.
Then there exists $\eta_0>0$ such that for any square-free integer $d$ with no factors in $S$,
 we have
$$|\cA_d(T)|=g(d) \X + r_d(\cA)$$
where $g(d)= \frac{\# \O_F(d)}{\# \O_d}$ and $r_d(\cA)=  {\# \O_F(d)} \cdot O({\X^{1-\eta_0}} ).$
\end{cor}
\begin{proof}
Since $\G_d\subset \G_{w_0} (d)$, the assumption implies that
the family $\{\G_{w_0}(d): \text{$d$ is square-free with no factors in $S$}\}$ has a uniform spectral gap.
Therefore, Theorem \ref{mc} on $\# (w_0\G_{w_0}(d)\gamma \cap \B_T )$ implies that for some uniform $\e_0>0$,
\begin{align*}
           |\cA_d(T)|&=\sum_{\gamma\in \G_{w_0}(d)\ba \G , F(w_0\gamma)\equiv 0 (d)} \# (w_0\G_{w_0}(d)\gamma \cap \B_T )
\\ &=\sum_{\gamma\in \G_{w_0}(d)\ba \G , F(w_0\gamma)\equiv 0 (d)} \left(\tfrac{1}{[\G:\G_{w_0}(d)]}\X   +O(\X^{1-\e_0})\right) .
 \end{align*}
Since $\# \O_F(d)=\#\{ \gamma\in \G_{w_0}(d)\ba \G , F(w_0\gamma)\equiv 0 \;\;(d)\}$, the claim follows.
\end{proof}

In the following we verify the sieve axioms $(A_1)$, $(A_2)$ and $(A_3)$ in this set-up.
This step is very similar to \cite[sec. 4]{NS} as we use the same combinatorial sieve and the only difference is
that we use the variable $\X=\M_{w_0G}(\B_T)$ instead of $T$. This is needed for us, as we are working with very general sets
$\B_T$; however if $\M_{w_0G}(\B_T)\asymp T^{\alpha}$ for some $\alpha>0$, we could also use the parameter $T$.

Using a theorem of Matthews, Vaserstein and Weisfeiler \cite{MVW},
and enlarging $S$ if necessary,  the diagonal embedding of $\G$ is dense in
$\prod_{p\notin S} \tilde {\bf G}(\mathbb F_p)$. The multiplicative property of $g$ on square-free integers
with no factors in $S$ follows from this (see \cite[proof of Prop. 4.1]{NS}).

 Letting $W_j=\{x\in W: F_j(x)=0\}$,
$W_j$ is an absolutely irreducible affine variety over $\q$ of dimension $\text{dim}(W)-1$
and hence by Noether's theorem,
 $W_j$ is absolutely irreducible over $\mathbb F_p$ for all $p\notin S$, by enlarging
$S$ if necessary. We may also assume that $W(\mathbb F_p)=w_0{\bf G}(\mathbb F_p)$ (possibly after adding $\{0\}$)
for all $p\notin S$ by Lang's theorem \cite{Lang}.
Using Lang-Weil estimate \cite{LW} on $\#W(\mathbb F_p)$ and $\# W_j(\mathbb F_p)$,
we obtain that for $p\notin S$,
$$\# \O_F(p)=r\cdot  p^{{\dim}(W)-1} +O(p^{\dim W-3/2})\;\text{ and }\; \# \O_p=p^{\dim W} +O(p^{\dim W-1/2}) .$$
Hence $g(p)= r\cdot p^{-1} +O(p^{-3/2})$ for all $p\notin S$.
This implies $A_3$ (cf. \cite[Thm 2.7]{MV}), as well as the last claim of $A_1$.

Moreover this together with Corollary \ref{ns1} imply that
$r(\cA, d)\ll d^{\dim W -1} \X^{1-\eta_0}$.
Hence for $D\le  \X^{{\eta_0}/({2\dim W})}$ and $\e_0=\eta_0/2$,
$$\sum_{d\le D}r(\cA, d)\ll D^{\dim W } \X^{1-\eta_0}  \le \X^{1-\e_0},$$
providing $(A_2)$.
Therefore
for any $z=D^{1/s}\le  \X^{{\eta_0}/({2 s\dim W})}$ and $s>9r$, and for all large $\X$,
we have
\be\label{car2}  S(\cA, P)\asymp \frac{\X}{(\log \X )^r}.\ee

\noindent{\bf Proof of Theorem \ref{affine_u}}.
Using arguments in the proof of Corollary \ref{ns1}, we first observe (cf. \cite[Lem. 4.3]{NS}) that
there exists $\eta>0$ such that for any $k\in \N$, $$\#\{x\in \mathcal O\cap \B_T: F_j(x)=k\} \ll  \X^{1-\eta}.$$
Fixing $0<\e_1 <\eta$, it implies that
\be\label{car1}  \#\{x\in \mathcal O\cap \B_T: |F_j(x)| \le \X^{\e_1}\} \ll  \X^{1-\eta+\e_1}.\ee

Now
\begin{multline*}
 \#\{x\in \mathcal O\cap \B_T: \text{all } F_j(x) \text{ prime}\}
\le \sum_{j=1}^r \# \{ x\in \mathcal O\cap \B_T: |F_j(x)| \le \X^{\e_1}\} \\
+ \# \{ x\in \mathcal O\cap \B_T: |F_j(x)| \ge \X^{\e_1} \text{ for all $1\le j\le r$ and }
\text{ all }F_j(x) \text{ prime} \}.
\end{multline*}
Now for $z\le \X^{{\eta_0}/({2 s\dim W})}$ such that $P=\prod_{p<z} p\ll \X^{\e_1}$, we have
\begin{multline*} \{ x\in \mathcal O\cap \B_T: |F_j(x)| \ge \X^{\e_1} \text{ for all $1\le j\le r$ and }
\text{ all }F_j(x) \text{ prime} \}\\ \subset
\{ x\in \mathcal O\cap \B_T: (F_j(x), P)=1 \}
\end{multline*}
and the cardinality of the latter set is
$S(\cA, P)$ according to our definition of $a_n$'s.

Therefore, we obtain the desired upper bound:
$$\#\{x\in \mathcal O\cap \B_T: \text{all } F_j(x) \text{ prime}\}\ll
\X^{1-\eta+\e_1} +\frac{\X}{(\log \X )^r} \ll \frac{\X}{(\log \X )^r}.$$

\noindent{\bf Proof of Theorem \ref{affine_l}}. By the assumption, for some $\beta>0$,
\be \label{fd}\max_{x\in \B_T} \|x\|\ll \M_{w_0G}(\B_T)^{\beta}=\X^{\beta}.\ee
It follows that
\be \label{fd2}\max_{x\in \B_T} |F(x)|\ll \M_{w_0G}(\B_T)^{\beta \deg (F)}=\X^{\beta \deg (F)}.\ee

Then for $z= \X^{{\eta_0}/({2 s\dim W})}$ and $P=\prod_{p<z, p\notin S}p$, $R=\frac{\beta \cdot \deg (F) {2 s\dim W} }{\eta_0}$, we have
 $$\{x\in \O\cap \B_T: (F(x), P)=1 \} \subset
\{ x\in \O\cap \B_T: F(x) \text{ has at most $R$ prime factors}\},$$
since all prime factors of $F(x)$ has to be at least the size of $z$ if $(F(x), P)=1$ and $ |F(x)|\ll \X^{\beta \deg (F)}$ if $x\in \B_T$.
Since $S(\cA,P)=\#\{x\in \O\cap \B_T: (F(x), P)=1 \} $,
we get the desired lower bound $\frac{\X}{(\log \X )^r}$ from \eqref{car2}.

\begin{rmk}\rm  When $\G$ is an arithmetic subgroup  of a simply connected semisimple
algebraic $\q$-group $G$, and $H$ is a symmetric subgroup,
the analogue of Theorem \ref{mc} has been obtained in \cite{BeO}, assuming that $H\cap \G$ is a lattice in $H$.
Strictly speaking, \cite[Theorem 1.3]{BeO}
is stated only for a fixed group $\G$; however it is clear from its proof that the statement also
holds uniformly over its congruence subgroups with the correct main term, as in Theorem \ref{mc}.
Based on this, one can use the combinatorial sieve \ref{sie} to obtain analogues of
Theorems \ref{affine_u} and \ref{affine_l}, as it was done for a group variety in \cite{NS}.
Theorem \ref{affine_l} on lower bound for $\G$ arithmetic was obtained in \cite{GN} further assuming that
$H\cap \G$ is co-compact in $H$.
\end{rmk}



\begin{thebibliography}{10}


\bibitem{Au}
T. Aubin.
\newblock Nonlinear analysis on manifolds.
\newblock {\em GM 252 Springer}, 1982.


\bibitem{Babillot2002}
Martine Babillot.
\newblock On the mixing property for hyperbolic systems.
\newblock {\em Israel J. Math.}, 129:61--76, 2002.

\bibitem{BS}
M. W. Baldoni Silva and D. Barbasch
\newblock The unitary spectrum for real rank one groups.
\newblock {\em Inventiones Math.}, 72:27--55, 1983.

\bibitem{BeO}
Yves Benoist and Hee Oh.
\newblock  Effective equidistribution of S-integral points on symmetric varieties.
\newblock {\em Annales de L'Institut Fourier}, Vol 62 (2012), 1889-1942.


\bibitem{BGS2}
Jean Bourgain, Alex Gamburd, and Peter Sarnak.
\newblock Generalization of Selberg's $3/16$ theorem and Affine sieve.
\newblock {\em Acta Math.} 207, (2011) 255-290.


\bibitem{BGS}
Jean Bourgain, Alex Gamburd, and Peter Sarnak.
\newblock Affine linear sieve, expanders, and sum-product.
\newblock {\em Inventiones} 179, (2010) 559-644.


\bibitem{BKS}
Jean Bourgain, Alex Kontorovich, and Peter Sarnak.
\newblock Sector estimates for {H}yperbolic isometries.
\newblock {\em GAFA} 20, 1175--1200, 2010.

\bibitem{BK1}
Jean Bourgain and Alex Kontorovich.
\newblock On representations of integers in thin subgroups of $\SL(2,\z).$
\newblock {\em GAFA} 20,  1144--1174, 2010.

\bibitem{BK2}
Jean Bourgain and Alex Kontorovich.
\newblock On the strong density conjecture for integral Apollonian gaskets.
\newblock {\em Preprint} arXiv:1205.4416.


\bibitem{Bo}
B.~H. Bowditch.
\newblock Geometrical finiteness for hyperbolic groups.
\newblock {\em J. Funct. Anal.}, 113(2):245--317, 1993.

\bibitem{Bowen1971}
Rufus Bowen.
\newblock Periodic points and measures for {A}xiom {$A$} diffeomorphisms.
\newblock {\em Trans. Amer. Math. Soc.}, 154:377--397, 1971.

\bibitem{Burger1990}
Marc Burger.
\newblock Horocycle flow on geometrically finite surfaces.
\newblock {\em Duke Math. J.}, 61(3):779--803, 1990.





\bibitem{Da}
F.~Dal'bo.
\newblock Topologie du feuilletage fortement stable.
\newblock {\em Ann. Inst. Fourier (Grenoble)}, 50(3):981--993, 2000.



\bibitem{Do} Dmitry Dolgopyat.
\newblock On decay of correlations in Anosov flows.
\newblock {\em Ann. Math}, 147:357--390, 1998.

\bibitem{Di} J. Diximer.
\newblock Sur les repr\'esentations de certains groupes orthogonaux.
\newblock {\em C. R. Acad. Sc. Paris}, vol 89 (1960), 3263-3265.

\bibitem{DRS}
W.~Duke, Z.~Rudnick, and P.~Sarnak.
\newblock Density of integer points on affine homogeneous varieties.
\newblock {\em Duke Math. J.}, 71(1):143--179, 1993.


\bibitem{EM}
Alex Eskin and C.~T. McMullen.
\newblock Mixing, counting, and equidistribution in {L}ie groups.
\newblock {\em Duke Math. J.}, 71(1):181--209, 1993.





\bibitem{FS}
L. Flaminio and R. Spatzier.
\newblock Geometrically finite groups, Patterson-Sullivan measures and Ratner's theorem.
\newblock {\em Inventiones}, 99 (1990), 601-626.



\bibitem{G}
Alex Gamburd.
\newblock On the spectral gap for infinite index congruence subgroups of $\SL_2(\z).$
\newblock {\em Israel J. Math}, 127 (2002), 157-200.



\bibitem{GN1}
Alex Gorodnik and Amos Nevo.
\newblock The ergodic theory of Lattice subgroups.
\newblock {\em Annals of Math Studies.} Vol 172, 2009.

\bibitem{GN}
Alex Gorodnik and Amos Nevo.
\newblock Lifting, Restricting and Sifting integral points on affine homogeneous varieties.
\newblock {\em To appear in Compositio Math.} arXiv:1009.5217.


\bibitem{GO}
Alex Gorodnik and Hee Oh.
\newblock Orbits of discrete subgroups on a symmetric space and the
  {F}urstenberg boundary.
\newblock {\em Duke Math. J.}, 139(3):483-525, 2007.


\bibitem{GOS1}
Alex Gorodnik, Hee Oh and Nimish Shah.
\newblock Integral points on symmetric varieties and Satake compactifications
\newblock {\em American J. Math.}, 131:1-57, 2009.

\bibitem{GOS}
Alex Gorodnik, Hee Oh and Nimish Shah.
\newblock Strong wavefront lemma and counting lattice points in sectors.
\newblock {\em Israel J. Math.}, 176:419-444, 2010.

\bibitem{HR}
H. Halberstam and H. Richert.
\newblock Sieve methods.
\newblock {\em Academic Press.,} (1974) 167-242.

\bibitem{Hi}
Takeshi Hirai.
\newblock On irreducible representations of the Lorentz group of $n$-th order.
\newblock {\em Proc. Japan. Acad.}, Vol 38., 258-262,  1962.


\bibitem{HM}
Roger Howe and Calvin Moore.
\newblock Asymptotic properties of unitary representations.
\newblock {\em J. Funct. Anal.}, 72-96, 1979.



\bibitem{IK}
Henryk Iwaniec and Emmanuel Kowalski.
\newblock {\em Analytic number theory}, volume~53 of {\em American Mathematical
  Society Colloquium Publications}.
\newblock American Mathematical Society, Providence, RI, 2004.

\bibitem{KM}
D.~Y. Kleinbock and G.~A. Margulis.
\newblock Bounded orbits of nonquasiunipotent flows on homogeneous spaces.
\newblock In {\em Sina\u\i's {M}oscow {S}eminar on {D}ynamical {S}ystems},
  volume 171 of {\em Amer. Math. Soc. Transl. Ser. 2}, pages 141-172. Amer.
  Math. Soc., Providence, RI, 1996.
\bibitem{KS}
A. Knapp and E. Stein.
\newblock Intertwining operators for semisimple groups
\newblock {\em Annals of Math.}, 489-578, 1971.

\bibitem{Kim}
Inkang Kim.
\newblock Counting, Mixing and Equidistribution of horospheres in geometrically finite
rank one locally symmetric manifolds.
\newblock {\em Preprint}, arXiv:1103.5003.


\bibitem{Kn} A. Knapp.
\newblock Representation theory of semisimple Lie groups.
\newblock {\em Princeton University press}.

\bibitem{Ko}
Alex Kontorovich.
\newblock The hyperbolic lattice point count in infinite volume with
  applications to sieves.
\newblock {\em Duke Math. J.}, 149(1):1-36, 2009.


\bibitem{KO2}
Alex Kontorovich and Hee Oh.
\newblock Almost prime {P}ythagorean triples in thin orbits.
\newblock {\em J. Reine Angew. Math.} 667, 89-131, 2012.


\bibitem{KO}
Alex Kontorovich and Hee Oh.
\newblock Apollonian circle packings and closed horospheres on hyperbolic
  3-manifolds.
\newblock {\em Journal of AMS.}, 603-648, 2011.



\bibitem{Kow}
E. Kowalski.
\newblock Sieve in expansion.
\newblock {\em S\'eminaire Bourbaki} arXiv:1012.2793.

\bibitem{Kow1}
E. Kowlaski.
\newblock Sieve in discrete groups, especially sparse.
\newblock {\em MSRI publications, to appear} arXiv:1207.7051.


\bibitem{Lang}
S. Lang.
\newblock Algebraic groups over finite fields.
\newblock {\em American J. Math.} 78 (1956) 555-563.

\bibitem{LW}
S. Lang and A. Weil.
\newblock Number of points of varieties in finite fields.
\newblock {\em American J. Math.} 76 (1954) 819-827.


\bibitem{LaxPhillips}
Peter~D. Lax and Ralph~S. Phillips.
\newblock The asymptotic distribution of lattice points in {E}uclidean and
  non-{E}uclidean spaces.
\newblock {\em J. Funct. Anal.}, 46(3):280-350, 1982.

\bibitem{LO} Min Lee and Hee Oh.
\newblock Effective count for Apollonian circle packings and closed horospheres.
\newblock {\em GAFA } Vol 23 (2013), 580-621.


\bibitem{LS}
J. Liu and P. Sarnak.
\newblock Integral points on quadrics in three variables whose coordinates have few prime factors.
\newblock {\em Israel J. Math } 178 (2010) 393-426.



\bibitem{Magee} M. Magee.
\newblock Quantitative spectral gap for thin groups of hyperbolic isometries.
\newblock {\em Preprint}. arXiv:1112.2004.


\bibitem{Mau} F. Maucourant.
\newblock Homogeneous asymptotic limits of Haar measures of semisimple Lie groups and their lattices.
\newblock {\em Duke Math. J.}. vol 136 (2) 1007, 357-399.

\bibitem{Margulisthesis}
Gregory Margulis.
\newblock  On some aspects of theory of {A}nosov systems.
\newblock Springer Monographs in Mathematics. Springer-Verlag, Berlin, 2004.
\newblock With a survey by Richard Sharp: Periodic orbits of hyperbolic flows,
  Translated from the Russian by Valentina Vladimirovna Szulikowska.

\bibitem{MVW} C. Matthews, L. Vaserstein and B. Weisfeiler.
\newblock Congruence properties of Zariski dense subgroups.
\newblock {\em Proc. London Math. Soc}. 48, 1984, 514-532.

\bibitem{MV} H. Montgomery and R. Vaughan.
\newblock Multiplicative number theory. I. Classical theory.
\newblock {\em Camb. Studies in Advanced Math}. 97, Camb, Univ. Press. 2007.


\bibitem{NS}
Amos Nevo and Peter Sarnak.
\newblock Prime and Almost prime integral points on principal homogeneous spaces.
\newblock {\em Acta Math.}, 205 (2010), 361-402


\bibitem{Oh1}
Hee Oh.
\newblock Orbital counting via mixing and unipotent flows
\newblock {\em Homogeneous flows, Moduli spaces and Arithmetic.}  Clay Math. Proceedings (2010), Vol 10,
339-375.



\bibitem{Oh0}
Hee Oh.
\newblock  Dynamics on geometrically finite hyperbolic manifolds with applications to Apollonian circle packings and beyond.
\newblock {\em  Proceedings  of ICM}. 2010, Vol III, 1308-1331. 



\bibitem{Oh}
Hee Oh.
\newblock  Harmonic analysis, Ergodic theory and
Counting for thin groups.
\newblock {\em MSRI publications, to appear}. 
arXiv:1208.4148.






\bibitem{OS}
Hee Oh and Nimish Shah.
\newblock  Equidistribution and counting for orbits of geometrically finite hyperbolic groups.
\newblock {\em Journal of AMS}. Vol 26 (2013) 511-562.

\bibitem{OS1}
Hee Oh and Nimish Shah.
\newblock The asymptotic distribution of circles in the orbits of {K}leinian
  groups.
\newblock {\em Inventiones}, Vol 187, 1-35, 2012.

\bibitem{Patterson1976}
S.J. Patterson.
\newblock The limit set of a {F}uchsian group.
\newblock {\em Acta Mathematica}, 136:241--273, 1976.

\bibitem{PP}
J. Parkkonen and F. Paulin.
\newblock Counting common perpendicular arcs in negative curvature.
\newblock {\em Preprint}, arXiv:1305.1332.


\bibitem{Roblin2003}
Thomas Roblin.
\newblock Ergodicit\'e et \'equidistribution en courbure n\'egative.
\newblock {\em M\'em. Soc. Math. Fr. (N.S.)}, (95):vi+96, 2003.



\bibitem{SS}
A. Salehi Golsefidy and P. Sarnak.
\newblock Affine Sieve.
\newblock {\em To appear in Journal of AMS}, arXiv:1109.6432.

\bibitem{SV} A. Salehi Golsefidy and P. Varj\'u.
\newblock Expansion in perfect groups.
\newblock {\em GAFA.} Vol 22 (2012), 1832-1891.


\bibitem{Sch}
Barbara Schapira.
\newblock Equidistribution of the horocycles of a geometrically finite surface.
\newblock {\em Int. Math. Res. Not.}, (40):2447-2471, 2005.

\bibitem{Sh}
Yehuda Shalom.
\newblock Rigidity, unitary representations of semisimple groups, and
  fundamental groups of manifolds with rank one transformation group.
\newblock {\em Ann. of Math. (2)}, 152(1):113-182, 2000.



\bibitem{St}
L. Stoyanov.
\newblock Spectra of Ruelle transfer operators for axiom A flows.
\newblock {\em Nonlinearity}, 24 (2011), 1089-1120.

		
\bibitem{Sullivan1979}
Dennis Sullivan.
\newblock The density at infinity of a discrete group of hyperbolic motions.
\newblock {\em Inst. Hautes \'Etudes Sci. Publ. Math.}, (50):171--202, 1979.

\bibitem{Sullivan1984}
Dennis Sullivan.
\newblock Entropy, {H}ausdorff measures old and new, and limit sets of
  geometrically finite {K}leinian groups.
\newblock {\em Acta Math.}, 153(3-4):259-277, 1984.

\bibitem{V}
Ilya Vinogradov.
\newblock Effective bisector estimate with applications to Apollonian circle packings.
\newblock preprint, arXiv:1204.5498.




\bibitem{Wa}
Garth Warner.
\newblock Harmonic Analysis on semisimple Lie groups I.
\newblock {\em M\'em. Soc. Math. Fr. (N.S.)}, (95):vi+96, 2003.


\bibitem{Wa2}
Garth Warner.
\newblock Harmonic Analysis on semisimple Lie groups II.
\newblock {\em M\'em. Soc. Math. Fr. (N.S.)}, (95):vi+96, 2003.

\bibitem{Wi}
Dale Winter.
\newblock {\em In preparation}.



\end{thebibliography}
\end{document}